\documentclass{amsart}

\usepackage{amssymb, amsmath, graphicx,  amsfonts, amsthm, url}
\usepackage{mathrsfs}
\usepackage{color}
\usepackage{xcolor}

\begin{document}

\title[The curvature of the manifold of  H\"older equilibrium probabilities]{The sectional curvature of the  infinite dimensional manifold of  H\"older equilibrium probabilities}




\author{Artur O. Lopes}
\address{Inst. de Matematica e Estatistica - UFRGS - Porto Alegre - Brazil}
\curraddr{}
\email{arturoscar.lopes@gmail.com}
\thanks{}

\author{Rafael O. Ruggiero}
\address{Dept. de Matematica - PUC - Rio de Janeiro - Brazil}
\curraddr{}
\email{rafael.o.ruggiero@gmail.com }
\thanks{Partially supported by CNPq}

\subjclass[2010]{37D35; 37A60}

\keywords{}

\date{25-8-2023}

\dedicatory{}

\begin{abstract}  Here we consider the  discrete time  dynamics described by a transformation $T:M \to M$, where $T$ is either  the action of shift $T=\sigma$  on the symbolic space $M=\{1,2,...,d\}^\mathbb{N}$, or, $T$ describes  the action of a $d$ to $1$  expanding transformation $T:S^1 \to S^1$ of class $C^{1+\alpha}$ (\,for example   $x \to T(x) =d\, x $ (mod $1) $\,), where $M=S^1$ is the unit circle.
It is known that the infinite-dimensional manifold $\mathcal{N}$ of  H\"older equilibrium probabilities is an analytical manifold  and carries a natural Riemannian metric. Given a certain normalized H\"older potential $A$ denote by $\mu_A \in \mathcal{N}$ the associated equilibrium probability. The set of tangent vectors $X$ (functions $X: M \to \mathbb{R}$) to the manifold $\mathcal{N}$ at the point
$\mu_A$ (a subspace of the Hilbert space $L^2(\mu_A)$) coincides with  the kernel of the Ruelle operator for the normalized potential $A$.
The Riemannian norm $|X|=|X|_A$ of the vector $X$, which is tangent to $\mathcal{N}$ at the point $\mu_A$,  is described via  the asymptotic variance, that is, satisfies
\smallskip

$\,\,\,\,\,\,\,\,\,\,\,\,\,|X|^2\,\,= \,\,\langle X, X \rangle \,\,=\,\,\lim_{n \to \infty}
\frac{1}{n} \int (\sum_{i=0}^{n-1} X\circ T^i )^2 \,d \mu_A$.
\smallskip

Consider an orthonormal basis $X_i$,
$i \in \mathbb{N}$, for the tangent space at $\mu_A$.
For any two orthonormal vectors  $X$ and $Y$ on the basis the curvature $K(X,Y)$ is
$$K(X,Y)  =  \frac{1}{4}[\, \sum_{i=1}^\infty (   \int X \,Y\, X_i \,d \mu_A)^2 -   \sum_{i=1}^\infty   \int X^2 X_i \,d \mu_A\,  \,\int Y^2 X_i \,d \mu_A \,].$$

When the equilibrium probabilities $\mu_A$ is the set of invariant Markov probabilities on
$\{0,1\}^\mathbb{N}\subset \mathcal{N}$, introducing an orthonormal basis $\hat{a}_y$, indexed by finite words $y$, we show explicit expressions for  $K(\hat{a}_x,\hat{a}_z)$, which is a finite sum. These values can be positive or negative depending on $A$ and the words $x$ and $z$. Words $x,z$ with large length can eventually produce large negative  curvature $K(\hat{a}_x,\hat{a}_z)$. If $x, z$ do not begin with the same letter, then $K(\hat{a}_x,\hat{a}_z)=0$.

\end{abstract}

\maketitle

\newtheorem{theorem}{Theorem}[section]
\newtheorem{lemma}[theorem]{Lemma}
\newtheorem{proposition}[theorem]{Proposition}
\newtheorem{corollary}[theorem]{Corollary}
\newtheorem{question}{Question}

\theoremstyle{definition}
\newtheorem{definition}[theorem]{Definition}
\newtheorem{remark}[theorem]{Remark}
\newtheorem{example}[theorem]{Example}

\newcommand{\fspace}[1]{\mathcal{#1}}
\newcommand{\qspace}[1]{\widehat{\mathcal{#1}}}
\newcommand{\spacem}[1]{\mathcal{#1}}
\newcommand{\op}[1]{\mathscr{#1}}
\newcommand{\dd}{\mathrm{d}}
\newcommand{\supp}{\operatorname{supp}}
\newcommand{\Var}{\operatorname{Var}}
\newcommand{\entropy}{\operatorname{h_{\fspace{X}}}}
\newcommand{\pressure}{\operatorname{Pr}}
\newcommand{\lspan}{\operatorname{span}}
\newcommand{\Id}{\operatorname{Id}}
\newcommand{\interior}{\operatorname{int}}
\newcommand{\rv}{\operatorname{rv}}
\newcommand{\Rot}{\operatorname{Rot}}
\newcommand{\Holder}{\operatorname{Hol}}
\newcommand{\eqdef}{\mathbin{\overset{\footnotesize{\mathrm{def}}}{=}}}

\newcommand{\cL}{\op{L}}
\newcommand{\bR}{\mathbb{R}}
\newcommand{\bN}{{\mathbb N}}

\newcommand{\fr}{\partial}

\section{Introduction}

We denote  by  $T:M \to M$  a transformation  acting  on the metric space $M$, which  is either  the shift  $\sigma$ acting on       $M=\{1,2,...,d\}^\mathbb{N}$, or, $T$ is the action of a $d$ to $1$   expanding transformation $T:S^1 \to S^1$, of class $C^{1+\alpha}$,   where $M=S^1$ is the unit circle.

For a fixed $\alpha>0$ we denote by $\text{Hol}$ the set of $\alpha$-H\"older functions on $M$.

For a H\"older potential $A: M \to \mathbb{R}$ we define the Ruelle operator (sometimes called transfer operator) - which acts on H\"older functions $f: M \to \mathbb{R}$ -   by
\begin{equation} \label{K37}f \to  \op{L}_A f(x) = \sum_{T(y) = x} e^{A(y)} f(y)\end{equation}
It is known (see for instance \cite{PP} or \cite{Bala}) that $\op{L}_A$ has a positive, simple
leading eigenvalue $\lambda_A$ with a positive H\"older eigenfunction $h_A$. Moreover, the
dual operator acting on measures $\op{L}_A^\ast$ has a unique
eigenprobability $\nu_A$ which is associated to the same eigenvalue $\lambda_A$.

Given a  H\"older potential $A$ we say that the probability $\mu_A$ - defined on the Borel sigma-algebra of $M$ - is the equilibrium probability for $A$, if $\mu_A$ maximizes the values
$$ h(\mu) + \int A\,\ d \mu,$$
among  Borel  $T$-invariant probabilities $\mu$ and where $h(\mu)$ is the Kolmogorov-Sinai  entropy of $\mu$.

The theory of thermodynamics formalism shows that the probability $\mu_A$ is unique and is given by the expression $\mu_A = h_A\, \nu_A$.

In some particular cases, the  equilibrium probability (also called Gibbs probability) $\mu_A$ is the one observed on the thermodynamical equilibrium in the Statistical Mechanics of the one-dimensional lattice $\mathbb{N}$ (under an interaction described by the potential $A$). As an example  (where the spin in each site of the lattice $\mathbb{N}$  could be $+$ or $-$) one can take
$M=\{+,-\}^\mathbb{N}$, $A: M \to \mathbb{R}$  and $T$ is the shift.

Taking into account the above definitions, we  say that a   H\"older potential $A$ is normalized if   $\op{L}_A \,1\,=1.$
In this case $\lambda_A=1$ and $\mu_A=\nu_A$.

Two  potentials $A, B$ in $\text{Hol}$ will be called cohomologous to each other (up to a constant),  if there exists a  continuous function $g:M \to \mathbb{R}$ and a constant $c$, such that,
\begin{equation} \label{K38}A = B + g - g \circ T - c.\end{equation}

Note  that the equilibrium probability for $A$, respectively  $B$, is the same if $A$ and $B$ are coboundaries to each other. In each coboundary class (an equivalence relation) there exists a unique normalized potential $A$ (see \cite{PP}).
Therefore, the set of equilibrium probabilities for H\"older potentials $\mathcal{N}$  can be indexed by  H\"older potentials $A$ which are normalized. We will use this point of view here:\, $A\,\leftrightarrow\,    \mu_A$.

 The infinite-dimensional manifold $\mathcal{N}$ of  H\"older equilibrium probabilities $\mu_A$  is an analytic manifold  (see \cite{Ru}, \cite{SiSS}, \cite{PP}, \cite{Chae})  and it was shown in \cite{GKLM}  that it carries a natural Riemannian structure. In order to provide a context for our main result, let us review first some of the main properties of this infinite-dimensional manifold and some definitions described on \cite{GKLM}.

 The set of tangent vectors $X$ (a function $X: M \to \mathbb{R}$) to $\mathcal{N}$ at the point
$\mu_A$ coincides with  the kernel of $\op{L}_A $.
The Riemannian norm $|X|=|X|_{\mu_A}$ of the vector $X$, which is tangent to $\mathcal{N}$ at the point $\mu_A$,  is described (see Theorem D in \cite{GKLM}) via  the asymptotic variance, that is, satisfies
\smallskip
\begin{equation} \label{K34}|X|\,\,= \,\,\sqrt{\langle X,X \rangle}\,\,=\,\,\sqrt{\lim_{n \to \infty}
\frac{1}{n} \int (\sum_{j=0}^{n-1} X\circ T^j )^2 \,d \mu_A}\end{equation}
The associated bilinear form on the tangent space at the point $\mu_A$ can be  described (see Theorem D in \cite{GKLM}) by
\begin{equation} \label{K33}\langle \,X\,,\,Y\,\rangle \,=\, \int X\, Y\, d \mu_A.\end{equation}
This bilinear form is positive semi-definite and in order to make it definite one can  consider equivalence classes  (cohomologous up to a constant) as described by  Definition 5.4 in \cite{GKLM}. In this way, we finally get a Riemannian structure on $\mathcal{N}$ (as anticipated in some paragraphs above). Elements $X$ on the tangent space at $\mu_A$ have the property $\int \,X\, d \mu_A=0.$
The tangent space to $\mathcal{N}$ at $\mu_A$ is denoted by $T_{A}\mathcal{N}$.

Given a normalized potential $A$ let $\{ X_i \}$ be an orthonormal basis of $T_{A}\mathcal{N}$, $i \in \mathbb{N}$.

Our main result is :

\begin{theorem} \label{main}  Let $A$ be a normalized potential, and let $\{ X_i \}$ be an orthonormal basis of $T_{A}\mathcal{N}$. Let $X=X_{1}, Y= X_{2}$, then the sectional curvature $K(X,Y)$ is given by
\begin{equation} \label{vale} K(X,Y)  =  \frac{1}{4}[\, \sum_{i=1}^\infty (   \int X \,Y\, X_i \,d \mu_A)^2 -   \sum_{i=1}^\infty   \int X^2 X_i \,d \mu_A\,  \,\int Y^2 X_i \,d \mu_A \,].
\end{equation}
\end{theorem}

\smallskip

The expression of $K(X,Y)$ applies of course to any pair of vectors in the basis $\{ X_{i}\}$, we can always change the enumeration of the vectors in the basis without changing the basis. The work consists of two distinct parts: the first part, from Sections 2 to 5, has a more geometric nature and deals with the calculation of the Levi-Civita connection and the curvature tensor.  This estimate becomes quite complex because we are dealing with an infinitely dimensional Riemannian manifold. Our goal was to express the sectional curvature for sections on the tangent space at $\mu_A$ in terms of integrals of functions with respect to $\mu_A$.
An important tool which will be used here is item (iv) on Theorem 5.1 in \cite{GKLM}: for all  normalized  $A\in\mathcal{N}$, $X \in T_{A}\mathcal{N}$  and $\varphi$ a continuous function
it holds:
 \begin{equation} \label{dif} \frac{\dd}{\dd t} \left.\int \varphi \,d\mu_{A\,+\,t\,X}\right|_{t=0}
  = \int\,  \varphi \,X\, d \,\mu_A.
  \end{equation}

In Section \ref{susub} we describe the expression of sectional curvature  $K(X,Y)$ in terms of the calculus of thermodynamics formalism.

The nature of the second part of the paper, from Sections 6 to 9, is more dynamic, analytical   and considers
$M=\{0,1\}^\mathbb{N}.$ We denote by $\mathcal{K}$  the set of stationary Markov probabilities taking values in $\{0,1\}$. The set of shift invariant probabilities $\mu\in \mathcal{K}$ is contained in
$ \mathcal{N}$. The probabilities $\mu$  are defined on the space
$\{0,1\}^\mathbb{N}$. The two dimensional manifold $\mathcal{K}$ is the set of equilibrium probabilities  for potentials  $A$ depending on the two first  coordinates (see \cite{PP}), that  is, when $A(x_1,x_2,x_3,.,.,x_n,..)= A(x_1,x_2).$

For each point $\mu_A$ in $\mathcal{K}$ we are able to exhibit a special orthonormal basis $\{ \hat{a}_y\}$ for the tangent space $T_{A}\mathcal{N}$, indexed by finite words $y$ on the alphabet $\{0,1\}$ (see expression (\ref{tororo15})). This orthonormal family will be denoted by $\mathcal{F}.$
We focus, for each point in $\mathcal{K}$, on the sectional curvatures for  pairs of vectors on $\mathcal{F}$.
We get explicit results in this case. This second part of the article is perhaps the more technical and subtle part; after some computations we will get the explicit expression for sectional curvature $K(\hat{a}_x,\hat{a}_z)$ (see expression \eqref{tororo182} in Theorem \ref{memo} and Propositions \ref{pop23} and \ref{pop24}).

A remarkable fact appearing in the proof of Theorem \ref{main} is that  the expression (\ref{vale}) of  the sectional curvature  $K(\hat{a}_x,\hat{a}_z)$ is actually a  sum of a finite number of parcels (see expression \eqref {tororo182} in Theorem \ref{memo} and Remark \ref{rte}).

We highlight some properties that will be demonstrated in the future and that describe the eventual values of the sectional curvature $K(\hat{a}_x,\hat{a}_z)$ depending on the pair of vectors $\hat{a}_x,\hat{a}_z$ and the point in $\mathcal{K}$ under consideration.
\smallskip

1.  Each  vector $ \hat{a}_y$  is a function which is constant in cylinders of finite size (see expressions
\eqref{tororo15} and \eqref{pede}).
More precisely, given a finite word $y=(y_1,y_2,...,y_n)$, $n \geq 1$,  we denote by $[y]=[y_1,y_2,...,y_n]$ the associated cylinder set in $\{0,1\}^\mathbb{N}$. The function
$\hat{a}_y$ is constant in each of the cylinder sets $[a,y_1,y_2,...,y_n,b]$, where $a,b=0,1$. The support of $\hat{a}_y$  is the union of these cylinder sets. In this way if the word $y$ has large length, then the support of $\hat{a}_y$  is contained on very small sets. We will have to consider the empty word which will give rise to  two tangent vectors  $\hat{a}_{\emptyset}^0$ and  $\hat{a}_{\emptyset}^1$,  which are functions
with support on cylinders of size two.

2. The values $K(\hat{a}_x,\hat{a}_z)$  can be positive or negative depending on the point in $\mathcal{K}$ and the words $x$ and $z$ (see Example \ref{bor1}).

3.  We say that $z$ is a  subprefix of $x$, if $x$ and $z$ satisfy
$$  [x]=[x_1,x_2,...x_k,x_{k+1},..., x_n] \subset [z]=[x_1,x_2,...,x_k],$$
where $n \geq k$.
If $x$ and $z$ do not  begin with the same letter
(do not share a  common subprefix), then $K(\hat{a}_x,\hat{a}_z)=0$ (see Proposition \ref{bor}). As an example take $x=(0,1,1,0)$ and $z=(1,1,0)$.

4. Words $x$ and $z$ with large length can eventually produce extremely negative  curvature $K(\hat{a}_x,\hat{a}_z)$. This may happen when $x$ and $z$ have several common subprefixes. This is due to expression \eqref{tororo182}. As an example take $x=(0,1,1,0,0,1)$ and $z=(0,1,1,0,0,0,1)$.
But even in this case, it is possible to get positive curvature depending on the point in $\mathcal{K}$ (see Example \ref{bor1} for a discussion in a particular case).

5. We also show that if $\mu_A$ (a point in $\mathcal{K}$) corresponds to the measure of maximal entropy on  $\{0,1\}^\mathbb{N}$, most of the sectional curvatures $K(\hat{a}_x,\hat{a}_z)$ are equal to $-1/2$ (see Proposition \ref{maxma}). Proposition \ref{zeze} shows, in this case, an example where the sectional curvature $K(\hat{a}_{[\emptyset]}^0,\hat{a}_0)=1/2$. The different possibilities also include the case   $K(\hat{a}_{\emptyset}^0,\hat{a}_{\emptyset}^1)=0$.

6. Considering the two dimensional manifold $\mathcal{K}$ (of the Markov invariant probabilities) it is natural to consider that vectors on $T M$ should be functions depending on two coordinates. In our setting,  the corresponding elements on the basis $\mathcal{F}$ are  $\hat{a}_{\emptyset}^0$ and  $\hat{a}_{\emptyset}^1$.  We show that for any points in $\mathcal{K}$ the sectional curvature    $K(\hat{a}_{\emptyset}^0,\hat{a}_{\emptyset}^1)=0$ (see Theorem \ref{meme}).   In this way, considering $\mathcal{K}$   as a surface in itself, we get that $\mathcal{K}$ is  a flat surface (see Remark \ref{erre}).
\smallskip



In \cite{McM} , \cite{BCS} and \cite{PS} the authors consider a similar kind of Riemannian structure.
The bilinear form considered in \cite{McM} is the one we consider here  divided by the entropy of $\mu_A$.  As mentioned in section 8 in  \cite{GKLM} in that case the curvature  can be positive and also negative in some parts.

The main motivation for the results obtained on \cite{McM} (and also \cite{BCS})  is related to the study of a particular norm on the Teichm\"uller space.

The results presented in \cite{GKLM} and  here are related  to the topic of Information Geometry (see \cite{Ama} for general results on the subject) and  this is described in Section 5 in \cite{LR}. We point out that in the setting of Thermodynamic Formalism the asymptotic variance is the Fisher information (see Definition 4.3 and Proposition 4.4 in \cite{Ji}). Results about Kullback-Leibler
divergence on Thermodynamic Formalism appeared recently in \cite{LM}.

General references for analyticity (and inverse function theorems and  implicit function theorems) in Banach spaces are \cite{Chae} and \cite{W}.

A reference for general results in infinite-dimensional Riemannian manifolds is \cite{BM}.

In section 6 in \cite{GKLM} it is explained that the Riemannian metric considered here is not compatible with the $2$-Wasserstein Riemannian  structure on the space of probabilities.

We would like thanks to Paulo Varandas, Miguel Paternain, and Gonzalo Contreras for helpful conversations on questions related to the topics considered  in this paper.

We thank the referee for  extremely careful reading and criticism of previous versions of our paper.  Related results appear in \cite{LR1}.

\section{Preliminaries of Riemannian geometry} \label{Prel}

Let us introduce some basic notions of Riemannian geometry. Given an infinite-dimensional  $C^{\infty}$ manifold $(M,g)$ equipped with a smooth Riemannian metric $g$, let $T\,M$ be the tangent bundle and
$T_{1} M$ be the set of unit norm tangent vectors of $(M,g)$, the unit tangent bundle. Let $\chi(M)$ be the set of $C^{\infty}$ vector fields of $M$.

In \cite{BM} several results for Riemannian metrics on infinite-dimensional manifolds are presented. We will not use any of the results of that paper.

The only  infinite-dimensional  manifold we will be interested in here is $\mathcal{N}$ which is the set   of
H\"older equilibrium probabilities (which was initially defined in \cite{GKLM}). Tangent vectors, differentiability, analyticity, etc,  should be always considered in the sense of  the setting described in sections 2.3 and 5.1  in \cite{GKLM} (see also \cite{BCV} and \cite{SiSS}). We will elaborate on this later.

So in our case, $M= \mathcal{N}$, and $g$ is the $L^{2}$ metric, $g_{A}(X,Y) = \int X\, Yd\mu_{A}$,

For practical purposes, we shall call \textit{Energy} the function $E(v) = g(v,v)$, $v \in T \mathcal{N}$, although in mechanics the energy is rather defined by $\frac{1}{2}g(v,v)$.

Given a smooth function $f :\mathcal{N} \longrightarrow \mathbb{R}$, the derivative of $f$ with respect to a vector field $X \in \chi (\mathcal{N} )$ will be denoted by $X(f)$.
The Lie bracket of two vector fields $X, Y \in \chi(\mathcal{N} )$ is the vector field whose action on the set of functions $f: \mathcal{N}  \longrightarrow \mathbb{R}$ is
given by $[X,Y](f) = X(Y(f)) - Y(X(f))$.

The \textit{Levi-Civita connection} of $(\mathcal{N} ,g)$, $\nabla : \chi(\mathcal{N} )\times \chi(\mathcal{N} ) \longrightarrow \chi(\mathcal{N} )$, with notation $\nabla(X,Y) = \nabla_{X}Y$, is the affine operator
characterized by the following properties:
\begin{enumerate}
\item Compatibility with the metric $g$:
$$ Xg(Y,Z) = g(\nabla_{X}Y, Z) + g(Y, \nabla_{X}Z) $$
for every triple of vector fields $X, Y, Z$.
\item Absence of torsion: $$ \nabla_{X}Y - \nabla_{Y}X = [X,Y].$$
\item For every smooth scalar function $f$ and vector fields $X,Y \in \chi(\mathcal{N} )$ we have
\begin{itemize}
\item $ \nabla_{fX}Y = f\nabla_{X}Y$,
\item Leibniz rule: $ \nabla_{X}(fY) = X(f)Y + f\nabla_{X}Y$.
\end{itemize}
\end{enumerate}

The expression of $\nabla_{X}Y$ can be obtained explicitly from the expression of the Riemannian metric, in dual form. Namely, given two vector fields $X, Y \in \chi(\mathcal{N} )$,
and $Z \in \chi(\mathcal{N} )$ we have
\begin{eqnarray*}
g(\nabla_{X}Y, Z) & = & \frac{1}{2}(Xg(Y,Z) + Yg(Z, X) -Zg(X,Y) \\
& - & g([X,Z], Y) -g([Y,Z],X) -g([X,Y], Z)) ,
\end{eqnarray*}

\subsection{Curvature tensor and sectional curvatures}

We follow \cite{doCa} for the definitions in the subsection. To simplify the notation, from now on we shall adopt the convention $g(X,Y) = \langle X, Y \rangle$. The curvature tensor
\begin{equation*} \label{K3}\mathcal{R} : \chi(\mathcal{N}) \times \chi(\mathcal{N}) \times \chi(\mathcal{N}) \longrightarrow \chi(\mathcal{N})\end{equation*}
is defined in terms of the Levi-Civita connection as follows
\begin{equation} \label{K4} \mathcal{R} (X, Y)Z = \nabla_{Y}\nabla_{X}Z - \nabla_{X}\nabla_{Y}Z + \nabla_{[X,Y]}Z .\end{equation}

The sectional curvature of the plane generated by two vector fields $X, Y$ at the point $A \in \mathcal{N}$, which are orthonormal at $A$,  is given by
\begin{equation} \label{K5} K(X,Y) = \langle \nabla_{Y}\nabla_{X}X - \nabla_{X}\nabla_{Y}X +  \nabla_{[X,Y]}X, Y \rangle = \langle \mathcal{R}(X,Y)X,Y\rangle.\end{equation}

Let $A$ be a normalized H\"older potential. Let us consider a local smooth surface $S(t,s)$, for $\mid t \mid, \mid s \mid \leq \epsilon$ small, tangent to the plane $ \{ A + tX + sY\} $ generated by $X, Y$ at the point $A = S(0,0)$. Let $\bar{X}$, $\bar{Y}$ be the coordinate vector fields of the surface, and suppose that $\bar{X}_{A} =X$, $\bar{Y}_{A} =Y$. In Subsection 4.2 we shall exhibit such local surfaces.

\begin{lemma} \label{Sectional-cur}
	The expression of the sectional curvature of the plane generated by the two orthonormal vectors $X, Y$ is
\begin{equation} \label{K2}
K(X,Y) = -\frac{1}{2}(\bar{X}(\bar{X}(\parallel \bar{Y} \parallel^{2}) )+  \bar{Y}(\bar{Y}(\parallel \bar{X} \parallel^{2}) )) + \parallel \nabla_{\bar{Y}}\bar{X} \parallel^{2} + \bar{Y}(\bar{X}\langle \bar{X}, \bar{Y} \rangle ) - \langle \nabla_{\bar{X}}\bar{X} , \nabla_{\bar{Y}}\bar{Y} \rangle .
\end{equation}
\end{lemma}

\begin{proof}
	
	The fact that $\bar{X}$ and $\bar{Y}$ commute implies that $\nabla_{\bar{X}}\bar{Y} = \nabla_{\bar{Y}}\bar{X}$ and
$$ \langle \mathcal{R}(\bar{X},\bar{Y})\bar{X},\bar{Y} \rangle  = \langle \nabla_{\bar{Y}}\nabla_{\bar{X}}\bar{X} - \nabla_{\bar{X}}\nabla_{\bar{Y}}\bar{X} , \bar{Y} \rangle .$$

The first term of $\langle \mathcal{R}(\bar{X},\bar{Y})\bar{X},\bar{Y} \rangle$ gives
\begin{eqnarray*}
 \langle \nabla_{\bar{Y}}\nabla_{\bar{X}}\bar{X} , \bar{Y} \rangle & = &  \bar{Y}\langle \nabla_{\bar{X}}\bar{X}, \bar{Y} \rangle - \langle \nabla_{\bar{X}}\bar{X}, \nabla_{\bar{Y}}\bar{Y} \rangle  \\
 & = &  \bar{Y}(\bar{X} \langle \bar{X}, \bar{Y} \rangle - \langle \bar{X}, \nabla_{\bar{X}}\bar{Y}\rangle ) - \langle \nabla_{\bar{X}}\bar{X}, \nabla_{\bar{Y}}\bar{Y} \rangle \\
& = &  \bar{Y}(\bar{X} \langle \bar{X}, \bar{Y} \rangle - \langle \bar{X}, \nabla_{\bar{Y}}\bar{X}\rangle ) - \langle \nabla_{\bar{X}}\bar{X}, \nabla_{\bar{Y}}\bar{Y} \rangle  \\
& = & \bar{Y}(\bar{X} \langle \bar{X}, \bar{Y} \rangle - \frac{1}{2}\bar{Y} (\parallel \bar{X} \parallel^{2}))- \langle \nabla_{\bar{X}}\bar{X}, \nabla_{\bar{Y}}\bar{Y} \rangle \\
& = &  - \frac{1}{2}\bar{Y}(\bar{Y} (\parallel \bar{X} \parallel^{2})) +  \bar{Y}(\bar{X} \langle \bar{X}, \bar{Y} \rangle ) - \langle \nabla_{\bar{X}}\bar{X}, \nabla_{\bar{Y}}\bar{Y} \rangle
\end{eqnarray*}

The second term of the formula gives
\begin{eqnarray*}
 \langle \nabla_{\bar{X}}\nabla_{\bar{Y}}\bar{X}  , \bar{Y} \rangle & = & \bar{X}\langle \nabla_{\bar{Y}}\bar{X},\bar{Y} \rangle - \langle \nabla_{\bar{Y}}\bar{X}, \nabla_{\bar{X}}\bar{Y}\rangle \\
 & = & \bar{X}\langle \nabla_{\bar{X}}\bar{Y}, \bar{Y}\rangle - \langle \nabla_{\bar{Y}}\bar{X}, \nabla_{\bar{Y}}\bar{X}\rangle \\
 & = & \frac{1}{2} \bar{X}(\bar{X}(\parallel \bar{Y} \parallel^{2})) - \parallel \nabla_{\bar{Y}}\bar{X} \parallel^{2}
\end{eqnarray*}

Subtracting the second term from the first one we obtain the lemma.

\end{proof}

\section{The analytic structure of the set of normalized potentials}

\begin{definition}  Let $ (X, | .|)$ and  $(Y, |.|)$  Banach spaces and $V$ an open subset of $ X.$
Given $k\in  \mathbb{N}$, a function $F : V\to Y$ is called $k$-differentiable in $x $, if for each $j=1, ..., k$,
there exists a $j$-linear bounded transformation
$$D^j F(x) : \underbrace{X \times X \times ... \times X}_j  \to Y,$$
such that,
$$D^{j -1}F(x + v_j )(v_1, ..., v_{j-1}) \,-\, D^{j-1}F(x)(v_1, ..., v_{j-1}) = D^jF(x)(v_1, ..., v_j ) + o_ j (v_j ),\,\,$$
where
$$\,\, o_j
: X \to Y, \,\,\text{ satisfies,}\,\,  \lim_{v\to 0}
\frac{|o_j (v)|_Y}{
|v|_X
}= 0
$$

By definition $F$ has derivatives of all orders in $V$, if for any  $x\in V$ and any $k\in  \mathbb{N}$, the function $F$ is
$k$-differentiable in $x$.

\end{definition}

\begin{definition} Let $X, Y$ be Banach  spaces and $V$ an open subset of $X$. A function
$F : V \to X$ is called analytic on $V$, when $F$ has derivatives of all orders in $V$, and for each
$x \in V$ there exists an open neighborhood $V_x$ of $x$ in $V$, such that, for all $v\in  V_x$, we have that
$$
F(x + v) \,-\, F(x) = \sum_{j=1}^\infty\,
\frac{1}{n !}\,\,
D^j F(x)v^j,$$
where
$D^j F(x)v^j = D^j F(x)(v, . . . , v) $ and $ D_j F(x) $ is the $j $-th derivative of $F$ in $x$.

\end{definition}

Above we use the notation of section 3.2 in \cite{SiSS}.

\medskip

$\mathcal{N}$ can be expressed locally in coordinates via analytic charts (see \cite{GKLM}).

\subsection{Some more estimates from Thermodynamic Formalism} \label{yu}

Given a potential $B \in \text{Hol} $ we consider the associated Ruelle operator $\op{L}_B$
and the corresponding main eigenvalue $\lambda_{B}$ and eigenfunction $h_B$.

The function
\begin{equation} \label{K20} \Pi (B) = B + \log(h_{B}) - \log(h_{B}(T)) -\log(\lambda_{B}) \end{equation}
describes  the projection of the space of potentials $B$ on $\text{Hol}$  onto the analytic manifold of normalized potentials $\mathcal{N}$.

We identify below $T_A \mathcal{N}$ with the affine subspace $\{A + X\,:\, X \in T_A \mathcal{N}\}.$

The function $\Pi$ is analytic (see \cite{GKLM}) and therefore has first and second derivatives.  Given the potential $B$, then the map $D_B \Pi : T_{B}\mathcal{N} \longrightarrow T_{\Pi(B)}\mathcal{N} $ given by
$$ D_B \Pi (X) = \frac{\partial}{\partial t}(\Pi(B+ tX)_{t=0} $$

should be considered as a linear map from Hol to itself (with the H\"{o}lder norm on Hol). Moreover, the second derivative $D^2_B \Pi$ should be interpreted as a bilinear form from Hol $\times$ Hol to Hol, and is given by

$$ D^2_B \Pi(X,Y)  = \frac{\partial^2}{\partial t \partial s}(\Pi(B+ tX +sY)_{t=s=0}. $$

We denote by $||A||_\alpha$ the $\alpha$-H\"{o}lder norm of an $\alpha$-H\"{o}lder function $A$.

When $B$ is normalized the eigenvalue is $1$ and the eigenfunction is equal to $1$.  We would like to study the geometry of the
projection $\Pi$ restricted to the tangent space $T_{A}\mathcal{N}$ into the manifold $\mathcal{N}$
(namely, to get bounds for its first and second derivatives with respect to the potential viewed as a variable) for a given normalized potential $A$.

The space $T_{A}\mathcal{N}$ is a linear subspace of functions and the derivative map $D\, \Pi$ is analytic when restricted to it.

We denote by $E_0 =E_0^A$ the set of H\"{o}lder functions $g$, such that, $\int g d \mu_A =0,$ where $\mu_A$ is the equilibrium probability for the normalized potential $A$. Note that $E_0^A$ is contained in $T_{A} (\mathcal{N}).$

Most of the claims of the next Lemma are based mainly on results of \cite{GKLM} (see also  \cite{SiSS}, \cite{BCV}).

\begin{lemma} \label{derivative-bounds0}
Let $\Lambda : \text{Hol} \longrightarrow \mathbb{R}$,
$H : \text{Hol} \longrightarrow \text{Hol}$ be given, respectively,  by $ \Lambda (B) = \lambda_{B},$ $ H(B) = h_{B}$.
Then we have
\begin{enumerate}
\item The maps $\Lambda$, $H$, and $A \longrightarrow \mu_{A}$  are analytic.
\item For a normalized $B$ we get that $D_{B}\log(\Lambda) (\psi) = \int \psi d\mu_{B},$
\item $ D^{2}_{B}\log(\Lambda) (\eta, \psi) = \int \eta \psi d\mu_{B},$
where $\psi , \eta $ are at $T_{B}\mathcal{N}$, and $B$ is normalized.
\item If $A$ is a normalized potential, then for every function $X \in T_{A}\mathcal{N}$ we have
\begin{itemize}
\item $\int X d\mu_{A} =0$.
\item $D_{A}\Pi(X) = X$.
\end{itemize}

\end{enumerate}
\end{lemma}

In order to simplify the notation, from now on, unless is necessary for  the understanding,  we will denote $(I - \op{L}_{T,A}|_{E_0^A})^{-1}$ by $ (I - \op{L}_{T,A})^{-1}.$

Items (2) and (3) are taken from Theorem D in \cite{GKLM}. Item  $\int X d\mu_{A} =0$ in  (3) follows from Theorem A and Corollary B in \cite{GKLM}, and  the other item in (4) is trivial.

The analyticity of $\Lambda$ and $H$ of the item (1) are well-known facts (see chapter 4 in \cite{PP}  or Corollary B in \cite{GKLM}) which was also proved in \cite{BCV}.

The law that takes a H\"older potential $B$ to its normalization $A$ is differentiable according to section 2.2 in \cite{GKLM}.

Note that the derivative  linear operator  $X \to D_{A}H(X)$ is zero when $A$ is normalized.
\medskip


{\bf Remark 1:}  Item (1) above means that for a fixed H\"{o}lder function $f$ the map $A \to \int f d \mu_A$ is differentiable on $A$ (see theorem B in \cite{BCV})
\medskip

Questions related to second derivatives on Thermodynamic Formalism are considered in \cite{Pol}, \cite{PS1} and \cite{PS}.

\section{Evaluating the sectional curvatures of the Riemannian metric}

The goal of the section is to calculate the sectional curvature $K(X,Y)$ of the plane generated by two orthogonal vector fields tangent to $A \in \mathcal{N}$ applying the calculus of Thermodynamics formalism. We start with a technical result that is a consequence of formula \ref{dif}. This lemma will be extensively used in the article.

\subsection{Leibniz rule of differentiation }

\begin{lemma} \label{Leibniz}
Let $A \in \mathcal{N}$ and let $\gamma: (-\epsilon, \epsilon) \longrightarrow \mathcal{N}$ be a smooth curve such that $\gamma(0) =A$. Let $X(t) = \gamma'(t)$, and let $Y$ be a smooth vector field tangent to $\mathcal{N}$ defined in an open neighborhood of $A$. Denote by $Y(t)= Y(\gamma(t))$. Then the derivative of $\int Y(t) d\mu_{\gamma(t)}$ with respect to the parameter $t$ is
$$ 	\frac{d}{dt} \int Y(t)d\mu_{\gamma(t)} = \int \frac{dY(t)}{dt}d\mu_{\gamma(t)} + \int Y(t)X(t)d\mu_{\gamma(t)} $$
for every $ t\in (-\epsilon, \epsilon)$.
\end{lemma}

\begin{proof}

The idea of the proof is very simple and based on the fact that the function $Q: \chi(\mathcal{N})\times m_{T} \longrightarrow \mathbb{R}$ given by
$$Q(X,\mu) = \int Xd\mu$$
is a bilinear form, where $\chi(\mathcal{N})$ is the set of $C^{1}$ vector fields tangent to $\mathcal{N}$ and $m_{T}$ is the set of invariant measures of the map $T$. So the derivative of a function of the type $Q(X(t), \mu(t))$ satisfies a sort of Leibniz rule. Let us check.

Let us calculate the derivative at $t=0$, for every other $t \in (-\epsilon, \epsilon)$ the calculation is analogous. We have
\begin{eqnarray*}
	 \frac{d}{dt}\int Y(t) d\mu_{\gamma(t)} \mid_{t=0} & = & \lim_{t \rightarrow 0}\frac{1}{t}(\int Y(t)d\mu_{\gamma(t)} - \int Y(0) d\mu_{A} )\\
	 & = & \int \lim_{t\rightarrow 0}\frac{1}{t}(Y(t) -Y(0)) d\mu_{\gamma(t)} \\
	 & + & \lim_{t \rightarrow 0} \frac{1}{t}(\int Y(0) d\mu_{\gamma(t)} - \int Y(0)d\mu_{A}) \\
	 &= & \int \frac{dY(t)}{dt}d\mu_{A}+ \lim_{t \rightarrow 0} \frac{1}{t}(\int Y(0) d\mu_{A + tX(0)} - \int Y(0)d\mu_{A})
	 \end{eqnarray*}
where in the last step we use the fact that the derivative with respect to $t$ only depends on the vector $X(0)$ and not on the curve through $A$ tangent to $X(0)$. By equation (\ref{dif}) the second term in the above equality is just $\frac{d}{dt} \int Y(0) d\mu_{A+tX(0)} \mid_{t=0}$, which equals $\int X(0) Y(0) d\mu_{A }$. This finishes the proof of the lemma.

\end{proof}

From now on, we shall adopt the notations $\frac{\partial Y}{\partial t}= Y' = Y_{t}$, the second one applies when there is only one parameter involved in the calculations, the third one will be used otherwise.

\subsection{Auxiliary local surfaces in $\mathcal{N}$}
\bigskip
Next, given a normalized potential $A$ and $X, Y$ orthonormal vector fields in the tangent space of $A$, we proceed to construct a local surface $S(t,s)$, $\mid t \mid, \mid s \mid <\epsilon$ small, such that $S(0,0) =A$, and the tangent space
of $S(t,s)$ at $A$ is the plane generated by $X, Y$. Let us consider the plane
$$P(t,s) = A + tX + sY $$
where $t, s, \in \mathbb{R} $, that is a subset of $T_{A} \mathcal{N}$, and let $\Pi $ be the projection into $\mathcal{N}$ defined in equation \ref{K20}. The vector fields $X_{P(t,s) } = \frac{\partial}{\partial t} P(t,s) = X$, $Y_{P(t,s)} = \frac{\partial}{\partial s} P(t,s) = Y$ are tangent to the plane $P$ of course.

Let $S(t,s) = \Pi(P(t,s))$. By Lemma \ref{derivative-bounds0} item (5), the restriction of the map $\Pi$ to the plane $P(t,s)$ is a local diffeomorphism onto its image, so there exists $\epsilon >0$ small such that $S(t,s)$ is an analytic embedding of the rectangle $\{ \mid t \mid <\epsilon\} \times \{ \mid s \mid < \epsilon \}$.

The coordinate vector fields of $S(t,s)$ are $\bar{X}_{S(t,s)} = \frac{\partial }{\partial t}( \Pi(P(t,s))) = D_{P(t,s)} \Pi(X)$, $\bar{Y}_{S(t,s)} =  \frac{\partial }{\partial s}( \Pi (P(t,s)) = D_{P(t,s)} \Pi(Y)$, so $\bar{X}, \bar{Y}$ are extensions of $X, Y$.

Moreover, we have the following result from Thermodynamic Formalism (for  derivatives of high order see (3.4) in  \cite{Pol}):

\begin{lemma} \label{der-H} Suppose  $\psi:\{1,2,..,d\}^\mathbb{N} \to \mathbb{R}$ is H\"{o}lder, normalized and $\mu$ denotes the associated equilibrium probability.
 Assume also that the H\"{o}lder function $\phi$ satisfies $\op{L}_\psi (\phi)=0$. Denote by $\lambda_t$ and  $w_t$, $t\in \mathbb{R}$,  respectively, the eigenvalue and the eigenfunction for the Ruelle operator  $\op{L}_{\psi + t \phi}$. 
 Then, we have
\begin{enumerate}
\item  The derivative of $w_{t}$ satisfies 
\begin{equation} \label{jj1}  \frac{d}{dt}  w_t(x)|_{t=0}=c, \,\text{for all}\,\, x,
 \end{equation}
for some constant $c$.

\item Moreover, as $\psi$ is normalized
 \begin{equation} \label{jj2}  \frac{d}{dt}  \log (w_t(x)|_{t=0})=c, \,\text{for all}\,\, x
 \end{equation}

\item Suppose $\overline{X}$ is an analytic vector field, extending the tangent  vector $X$, defined in a neighborhood of $\phi$.  Let $\gamma : (-\epsilon, \epsilon)\to \mathcal{N}$ be  an integral curve  of $\overline{X}$, with $\gamma(0)=\phi$,  and let $\overline{\omega}_t$ be the curve of eigenfunctions for the Ruelle operator of $\gamma(t)$. Then, 
\begin{equation} \label{jj11}  \frac{d}{dt}  w_t(x)=c_t, \,\text{for all}\,\, x,
 \end{equation}
 is a curve of constant functions which is analytic on $t$.
 
\item For any tangent vector $X$ (in the kernel of the Ruelle operator), the directional derivative
 \begin{equation} \label{jj3} D_{\psi}H (X) =c_X =  D_{\psi}\log H (X),
 \end{equation}
 where $c_X$ depends on $X$ and $\psi$.
 
\item From \eqref{jj1}, we get
 \begin{equation} \label{jj1a}  \frac{d}{dt}  w_t(T(x))|_{t=0}=c, \,\text{for all}\,\, x,
 \end{equation}
and for the same constant $c$ of  \eqref{jj1}.
\end{enumerate}

\end{lemma}
\begin{proof}
 We are going to take derivative on the H\"{o}lder direction $\phi$. Assume that $\phi$ satisfies 
 $\op{L}_\psi (\phi)=0,$ which implies that $\int \phi d \mu=0.$

Denote by $w(t,x)=w_t(x)$  the normalized eigenfunction for $\op{L}_{\psi + t \phi}$ associated to the eigenvalue $\lambda_t$. That is 
\begin{equation} \label{jjd} \op{L}_{\psi + t \phi} (w_t)= \lambda_t w_t .
\end{equation}

Taking derivative on $t$:
$$ \frac{d}{dt}  \op{L}_{\psi + t \phi} (w(t,.) ) (x) = \op{L}_{\psi  + t \phi } ( \phi(.) \, w(t,.)) (x)+\op{L}_{\psi +t \phi}   (\frac{d}{ dt}   w(t,.)) (x). $$

Therefore, for all $x$, when $t=0$, we get
$$ \frac{d}{dt}  \op{L}_{\psi + t \phi} (w(t,.) ) (x)|_{t=0} = \op{L}_{\psi  } ( \phi(.) ) (x)+\op{L}_{\psi }   (\frac{d}{ dt}   w(t,.)|_{t=0} ) (x)=$$
$$0 +\op{L}_{\psi }   (\frac{d}{ dt}   w(t,.)|_{t=0} ) (x) . $$

On the other hand, for all $x$ and $t$
$$ \frac{d}{dt}[ \lambda_t\, w(t,x)\,]      = \, w(t,x) \frac{d}{dt}\lambda_t\,+ \lambda_t\,\frac{d}{dt}  w(t,x). $$
Then, taking $t=0$, 
$$ \frac{d}{dt}[ \lambda_t\, w(t,x)\,]|_{t=0}      = \, w(0,x)\, \frac{d}{dt}\lambda_t|_{t=0}\,+ \lambda_t\,\frac{d}{dt}|_{t=0}  w(t,x)= $$
$$\int \phi d \mu +\frac{d}{dt}|_{t=0}  w(t,x)= \frac{d}{dt}|_{t=0}  w(t,x). $$

Denote $g(x) = \frac{d}{ dt}   w(t,.)|_{t=0} ) (x) .$

Then,  $\forall  \,x$, we get from the above and \eqref{jjd}
$$\op{L}_{\psi }  (g) (x)  = g(x),$$
for the normalized potential $\psi$. But, the only continuous eigenfunctions for $\op{L}_{\psi }$, which are  associated to the eigenvalue $1$ are  the constant functions.

Therefore, there exists $c$ such that $ \frac{d}{dt}  w_t(x)|_{t=0}=c$, for all $x$.

As, for all $t$ and $x$
$$ \frac{d}{dt}  \log (w_t(x)|_{t=0})= \frac{ \frac{d}{dt}  w_t(x)|_{t=0}}{ w_0(x)}=\frac{d}{dt}  w_t(x)|_{t=0} ,$$
we get \eqref{jj2}.

\eqref{jj3} follows at once from the above.

Expression \eqref{jj11} is obtained in the same way as it was derived \eqref{jj1} (applying the argument for each value $t$), and, finally, \eqref{jj1a}  follows trivially  from \eqref{jj1}.

\end{proof}
\medskip

We will use the above result on the next Lemma.

\begin{lemma} \label{derivative-coor}
The derivatives with respect to $t , s$ of the coordinate vector fields $\bar{X}$, $\bar{Y}$ at the point $A$  (a normalized potential)  are
\begin{enumerate}
\item $\frac{\partial }{\partial t} \bar{X} =\frac{\partial }{\partial s} \bar{Y}  = - 1 $
\item $\frac{\partial }{\partial s} \bar{X} = \frac{\partial }{\partial t}\bar{Y} = 0$.
\end{enumerate}
\end{lemma}

\begin{proof}
 We assume that the tangent vector is H\"{o}lder and in the kernel of the Ruelle operator $\op{L}_A$ . The proof of  the Lemma will be  a direct consequence of Lemma \ref{der-H} taking $\psi=A$ and $\phi=X$.  We will prove first the item (1) above.

The local surface $S(t,s)$ is contained in the manifold of normalized potentials, and we denote, respectively,  the corresponding eigenvalue by $\lambda _{S(t,s)}$ and  the associated  eigenfunction by $h_{S(t,s)}$    (of the Ruelle operator associated to  $S(t,s)$).  

Let $I$ be the identity map. The expression of the projection $\Pi$ (equation (\ref{K20})) is
$$ \Pi (B) = I(B) + \log(h_{B}) - \log(h_{B}(T)) -\log(\lambda_{B}). $$

By definition, we have
\begin{equation}
 \frac{\partial }{\partial t}( \bar{X}_{S(t,0)})_{t=0}  =   \frac{\partial }{\partial t}( D_{P(t,0)}\Pi(X_{P(t,0)}))_{t=0}
 \end{equation}

Lemma \ref{derivative-bounds0} grants that all the functions involved in the expression of $\Pi$ are differentiable, so we get at the point $t=0$,
\begin{eqnarray} \label{popu}
\begin{aligned}
 \frac{\partial }{\partial t}( D_{P(t,0)}\Pi(X_{P(t,0)}))_{t=0} & =    \frac{\partial}{\partial t}(D_{P(t,0)}I(X_{P(t,0)}) )_{t=0}\\
& +   \frac{\partial}{\partial t} ((D_{P(t,0)}\log (H) )(X_{P(t,0)}))_{t=0} \\
 &  -    \frac{\partial}{\partial t}((D_{P(t,0)}\log(H\circ T))(X_{P(t,0)}))_{t=0} \\
 &  -    \frac{\partial}{\partial t} ((D_{P(t,0)}  \log (\Lambda)) (X_{P(t,0)}))_{t=0}
\end{aligned}
\end{eqnarray}
The first term gives at $t=0$,
$$  \frac{\partial}{\partial t}(D_{P(t,0)}I(X_{P(t,0)}))_{t=0} = \frac{\partial}{\partial t}(X)_{t=0}=0,$$
since $X$ does not depend on $t$.
\bigskip

{\bf  Claim 1: }

The second and third term  cancel due to \eqref{jj1} and \eqref{jj1a}.

\medskip

Indeed, the curves 
$$\alpha(t) = D_{P(t,0)}\log (H) )(X_{P(t,0)}), \mbox{ }\beta(t) = D_{P(t,0)}  \log (H \circ T) (X_{P(t,0)})$$ 
coincide by \eqref{jj1} and \eqref{jj1a} with the expression
$$ \alpha (t) = \beta (t) = \frac{\frac{d}{dt}c_{X_{P(t,0)}}}{c_{X_{P(t,0)}}}$$ for each $t$, where $c_{X}$ is given in Lemma 4.2. These curves are analytic and therefore differentiable, so their derivatives with respect to $t$ coincide. Since derivatives $\alpha'(t)$, $\beta '(t)$ appear with opposite signs in equation (18), they add up to zero in this formula. This proves the Claim. 
\bigskip

Finally, the fourth line of  equation (18) gives by Lemma 3.3 item (3),
$$-    \frac{\partial}{\partial t} ((D_{P(t,0)}  \log (\Lambda)) (X_{P(t,0)}))_{t=0} = -\int X^{2} d\mu = -1$$
since $X$ has $L^{2}$ norm equal to $1$.  The same argument applies replacing $X$ by $Y$ in the above proof, so this finishes the proof of item (1). 

\medskip

Item (2) follows the same type of reasoning  and using \eqref{jj11}. By definition we have,

\begin{eqnarray*}
 \frac{\partial }{\partial s} (\bar{X}_{S(0,s)})_{s=0} &= &\frac{\partial }{\partial s}( D_{P(t,s)}\Pi(X_{P(t,s)})_{t=0})_{s=0} .
\end{eqnarray*}

This expression, according to equation \eqref{popu} is
\begin{eqnarray*}
 \frac{\partial }{\partial s}( D_{P(t,s)}\Pi(X_{P(t,s)}))_{t=s=0} & = &   \frac{\partial}{\partial s}(D_{P(t,s)}I(X_{P(t,s)}) )_{t=s=0}\\
& + &  \frac{\partial}{\partial s} ((D_{P(t,s)}\log (H) )(X_{P(t,s)}))_{t=s=0} \\
 &  - &   \frac{\partial}{\partial s}((D_{P(t,s)}\log(H\circ T))(X_{P(t,s)}))_{t=s=0} \\
 &  - &    \frac{\partial}{\partial s} ((D_{P(t,s)}  \log (\Lambda)) (X_{P(t,s)}))_{t=s=0}
\end{eqnarray*}
The first term gives at $t=0$,
$$  \frac{\partial}{\partial s}(D_{P(t,0)}I(X_{P(t,0)}))_{t=0} = \frac{\partial}{\partial s}(X)_{t=0}=0$$
since $X$ does not depend on $t,s$. The fourth term is, by Lemma 3.3 item (3),
$$-     \frac{\partial}{\partial s} ((D_{P(t,s)}  \log (\Lambda) (X_{P(t,s)}))_{t=s=0} = - D^{2}_{A}\log (\Lambda)(Y,X) =   -\int XY d\mu_{A} =0 $$
\bigskip

As for the second and third terms we have 
\bigskip

\textbf{Claim 2}:
$$\frac{\partial}{\partial s} ((D_{P(t,s)}\log (H) )(X_{P(t,s)}))_{t=s=0}= \frac{\partial}{\partial s}((D_{P(t,s)}\log(H\circ T))(X_{P(t,s)}))_{t=s=0}.$$
\medskip

The proof goes as in Claim 1, letting 
$$\alpha_{s}(t) = D_{P(t,s)}\log (H) )(X_{P(t,s)}), \mbox{ }\beta_{s}(t) = D_{P(t,s)}  \log (H \circ T) (X_{P(t,s)})$$
we have by Lemma \ref{der-H} items (3) and (5) that $\alpha_{s}(t) = \beta_{s}(t)$ is an analytic curve of constant functions for each given $s$. Therefore, the function 
$$ \omega(t,s) = \alpha_{s}(t) = \beta_{s}(t) $$ 
is an analytic function of the parameters $t,s$ and therefore, the derivatives of $\alpha_{s}(t)$ and $\beta_{s}(t)$ with respect to $s$ coincide and give a family of constant functions in the local surface $S(t,s)$. This finishes the proof of Claim 2.
\medskip

Claim 2 yields that the sum of the second and third tems of the expression of $\frac{\partial }{\partial s}( D_{P(t,s)}\Pi(X_{P(t,s)}))_{t=s=0}$ vanishes, just finishing the proof of item (2).

\end{proof}

\subsection{The expression of $K(X,Y)$ in terms of the calculus of thermodynamics formalism}  \label{susub}

Let us first state some notations. Let $\bar{X}_{t}$ be the derivative of the vector field $\bar{X}$ with respect to the parameter $t$ and $\bar{X}_{s}$ be the derivative of
the vector field $\bar{X}$ with respect to the parameter $s$. The same convention applies to $\bar{Y}_{t}$, $\bar{Y}_{s}$. The notations $\bar{X}(Y) = \frac{\partial}{\partial t}\bar{Y} = \bar{Y}_{t}$ will always
represent derivatives with respect to the vector field $\bar{X}$, while $\bar{X}\bar{Y}$ or $\bar{X}\times \bar{Y}$ will represent the product of the functions $\bar{X}$ and $\bar{Y}$. Through the section this double character of
the vectors tangent to the manifold $\mathcal{N}$ which are also functions will show up in all statements and proofs.

\begin{theorem} \label{curvature1}
Let $A$ be a normalized potential, let $X, Y \in T_{A} \mathcal{N}$ be a pair of orthonormal vector fields, and let $S: (-\epsilon, \epsilon)\times (-\delta, \delta) \longrightarrow \mathcal{N}$ be the  local surface defined in the previous subsection with $S(0,0)= A$, $\bar{X} $, whose coordinate vector fields are $\bar{X}$, $\bar{Y}$, with $\bar{X}(A) =X$, $\bar{Y}(A) = Y$. Then the sectional curvature $K(X,Y)$ at $A$ of the plane generated by $X, Y$ is given by the expression
$$
 K(X,Y)=  \parallel \nabla_{\bar{Y}}\bar{X} \parallel^{2} - \langle \nabla_{\bar{X}}\bar{X}, \nabla_{\bar{Y}}\bar{Y} \rangle
$$

\end{theorem}

We shall subdivide the proof into several steps.

\begin{lemma} \label{1}
We have that $\bar{X}_{s} = \bar{Y}_{t}$ in the local surface $S$.
\end{lemma}

This is a straightforward consequence of the fact that the vector fields $\bar{X}, \bar{Y}$ commute.

Next, let us evaluate the terms of the sectional curvature in Lemma \ref{Sectional-cur},

\begin{eqnarray*}
K(X,Y) & = &  -\frac{1}{2}(\bar{X}(\bar{X}(\parallel \bar{Y} \parallel^{2})) +  \bar{Y}(\bar{Y}(\parallel \bar{X} \parallel^{2}))  ) +  \parallel \nabla_{\bar{Y}}\bar{X} \parallel^{2} \\
& + &\bar{Y}(\bar{X}\langle \bar{X}, \bar{Y} \rangle ) - \langle \nabla_{\bar{X}}\bar{X} , \nabla_{\bar{Y}}\bar{Y} \rangle .
\end{eqnarray*}

\begin{lemma} \label{step1}
At every point $p \in S(t,s)$ we have

\begin{enumerate}
\item $ \bar{X}(\bar{X}(\parallel \bar{Y} \parallel^{2}) )= 2\int \bar{Y}\bar{Y}_{tt}d\mu_{p} - \int \bar{Y}^{2} d\mu_{p} +\int \bar{X}^{2}\bar{Y}^{2} d\mu_{p} .$
\item $ \bar{Y}(\bar{Y}(\parallel \bar{X} \parallel^{2} ))= 2\int \bar{X}\bar{X}_{ss}d\mu_{p} - \int \bar{X}^{2} d\mu_{p} +\int \bar{X}^{2}\bar{Y}^{2} d\mu_{p}  .$
\end{enumerate}
In particular, if $p = A$ we have
\begin{enumerate}
\item $ \bar{X}(\bar{X}(\parallel \bar{Y} \parallel^{2}) )= 2\int \bar{Y}\bar{Y}_{tt}d\mu_{A} -1 +\int \bar{X}^{2}\bar{Y}^{2} d\mu_{A} .$
\item $ \bar{Y}(\bar{Y}(\parallel \bar{X} \parallel^{2} ))= 2\int \bar{X}\bar{X}_{ss}d\mu_{A} - 1 +\int \bar{X}^{2}\bar{Y}^{2} d\mu_{A}  .$
\end{enumerate}

\end{lemma}

\begin{proof}

The expression follows from the application of the Leibniz rule to differentiate $\parallel \bar{Y} \parallel^{2}= \int \bar{Y}^{2} d\mu_{p}$ (we shall omit for convenience the $p$ in the notation of the measure $d\mu_{p}$ ):

\begin{eqnarray*}
\bar{X}(\bar{X}\int \bar{Y}^{2} d\mu)  & = & \bar{X} ( 2\int \bar{Y} \bar{Y}_{t} d\mu + \int \bar{X} \bar{Y}^{2}d\mu) \\
& =& 2 \int (\bar{Y}_{t})^{2} d\mu + 2 \int \bar{Y}\bar{Y}_{tt} d\mu + 2\int \bar{Y} \bar{X} \bar{Y}_{t} d\mu \\
& + & \int \bar{X}_{t} \bar{Y}^{2} d\mu + 2 \int \bar{X}\bar{Y} \bar{Y}_{t} d\mu + \int \bar{X}^{2} \bar{Y}^{2} d\mu \\
& = & 2\int (\bar{Y}_{t})^{2} d\mu  + 2\int \bar{Y}\bar{Y}_{tt}d\mu + 4\int \bar{X}\bar{Y}\bar{Y}_{t} d\mu \\
& + & \int \bar{X}_{t}\bar{Y}^{2} d\mu + \int \bar{X}^{2}\bar{Y}^{2} d\mu .
\end{eqnarray*}
Since by Lemma \ref{derivative-coor} we have that $\bar{X}_{s} = \bar{Y}_{t}=0$, $\bar{X}_{t}= \bar{Y}_{s} = -1$, we get item (1) just by replacing this values in the integral expressions above.

Interchanging $\bar{X}$ and $\bar{Y}$, $ t$ and $s$, in the above formula, we get item (2). At the point $p = A$ we have that $\int \bar{X}^{2} d\mu_{A} = \int \bar{T}^{2} d\mu_{A} = 1$, so replacing these values in the formula we finish the proof of the lemma.

\end{proof}

\begin{lemma} \label{step2}
The expression of $\bar{Y}(\bar{X}\langle \bar{X}, \bar{Y} \rangle ) = \bar{Y}(\bar{X}\int \bar{X}\bar{Y} d\mu_{p}) $ is
$$
 \bar{Y}(\bar{X}\int \bar{X}\bar{Y} d\mu_{p})   =  \int \bar{Y}\bar{X}_{ts} d\mu_{p} + 1 - \int  \bar{Y}^{2} d\mu_{p}
 +  \int \bar{X} \bar{Y}_{ts} d\mu_{p} - \int \bar{X}^{2} d\mu_{p} + \int \bar{X}^{2} \bar{Y}^{2} d\mu_{p}
$$
at every point $p \in S(t,s)$. In particular, at $p=A$ we have
$$
 \bar{Y}(\bar{X}\int \bar{X}\bar{Y} d\mu_{A})   =  \int \bar{Y}\bar{X}_{ts} d\mu_{A}
 +  \int \bar{X} \bar{Y}_{ts} d\mu_{A} - 1 + \int \bar{X}^{2} \bar{Y}^{2} d\mu_{A}
$$
\end{lemma}

\begin{proof}

We apply the Leibniz rule,
\begin{eqnarray*}
\bar{Y}(\bar{X}\int \bar{X}\bar{Y} d\mu) & = & \bar{Y}( \int \bar{X}_{t}\bar{Y} d\mu + \int \bar{X}\bar{Y}_{t} d\mu + \int \bar{X}^{2} \bar{Y} d\mu )\\
& = & \int \bar{X}_{ts} \bar{Y} d\mu + \int \bar{X}_{t} \bar{Y}_{s} d\mu + \int \bar{X}_{t} \bar{Y}^{2} d\mu \\
& + & \int \bar{X}_{s} \bar{Y}_{t} d\mu + \int \bar{X} \bar{Y}_{ts}d\mu + \int \bar{X} \bar{Y}_{t} \bar{Y} d\mu \\
& + & \int \bar{Y}_{s}\bar{X}^{2} d\mu + 2\int \bar{Y} \bar{X} \bar{X}_{s} d\mu + \int \bar{X}^{2}\bar{Y}^{2} d\mu
\end{eqnarray*}
Since by Lemma \ref{1} we have that $\bar{X}_{s} = \bar{Y}_{t}$ we get the following formula just adding the terms in the above formula:
\begin{eqnarray*}
 \bar{Y}(\bar{X}\int \bar{X}\bar{Y} d\mu) &  = & \int \bar{Y}\bar{X}_{ts} d\mu + \int \bar{X}_{t}\bar{Y}_{s} d\mu + \int \bar{X}_{t} \bar{Y}^{2} d\mu + \int (\bar{X}_{s})^{2} d\mu \\
& + & \int \bar{X} \bar{Y}_{ts} d\mu + 3\int \bar{X}\bar{X}_{s}\bar{Y} d\mu + \int \bar{Y}_{s}\bar{X}^{2} d\mu + \int \bar{X}^{2} \bar{Y}^{2} d\mu .
\end{eqnarray*}

By Lemma \ref{derivative-coor} $\bar{X}_{s}= \bar{Y}_{t}=0$, $\bar{X}_{t}= \bar{Y}_{s}=-1$, and replacing these values in the integral expression above we obtain the formula in the statement. Moreover, if $p=A$ we know that $\int \bar{X}^{2} d\mu_{A} = \int ^{2} d\mu_{A} = 1$, as well as $\int \bar{Y}^{2} d\mu_{A} = \int Y^{2} d\mu_{A} = 1$, thus concluding the proof of the Lemma.

\end{proof}

\begin{corollary} \label{step3}
The term $-\frac{1}{2}(\bar{X}(\bar{X}(\parallel \bar{Y} \parallel^{2}) )+  \bar{Y}(\bar{Y}(\parallel \bar{X} \parallel^{2})))  +  \bar{Y}(\bar{X}\langle \bar{X}, \bar{Y} \rangle )$ in the expression of $K(X,Y)$ at the point $A$ vanishes.
\end{corollary}

\begin{proof}

To shorten notation, we shall omit the dependence of $A$ in the expressions. According to Lemma \ref{1}, we have that
\begin{enumerate}
\item $ \int \bar{X}\bar{X}_{ss} d\mu = \int \bar{X} \bar{Y}_{ts}d\mu$.
\item $\int \bar{Y} \bar{X}_{st} d\mu = \int \bar{Y} \bar{Y}_{tt}d\mu$.
\end{enumerate}

Replacing the above equalities in the expressions of Lemmas \ref{step1}, \ref{step2}, and adding the resulting formulae we get Corollary \ref{step3}.

\end{proof}

Theorem \ref{curvature1} follows at once from Corollary \ref{step3}.

\section{Cristoffel coefficients at the expression of $K(X,Y)$}

We denote by $\{ X_{i} \}$, $i \in \mathbb{N}$, a complete orthonormal base of the vector space  $T_{A} \mathcal{N} \subset L^2(\mu)$ (for the Gibbs probability $\mu$ associated to the normalized potential $A$).

The main goal of the section is to obtain the expression for the sectional curvature in Theorem \ref{main}.

Namely, let $A \in \mathcal{N}$ be a point in the manifold of normalized potentials, let $X,Y \in T_{A} \mathcal{N}$ be two orthonormal tangent vectors. Then the expression of the curvature of the plane generated by $X, Y$ is
 \begin{equation} \label{wqr} K(X,Y)  =  \frac{1}{4}[\, \sum_{i=1}^\infty (   \int X \,Y\, X_i \,d \mu)^2 -   \sum_{i=1}^\infty   \int X^2 X_i \,d \mu\,  \,\int Y^2 X_i \,d \mu \,].
 \end{equation}

In Proposition \ref{cov-der}  we will show that the above sum is well defined.

The proof is a direct calculation of the terms $ \parallel \nabla_{\bar{Y}}\bar{X} \parallel^{2}, \langle \nabla_{\bar{X}}\bar{X}, \nabla_{\bar{Y}}\bar{Y} \rangle$ that appear in the expression of the curvature in Theorem \ref{curvature1}. We shall subdivide the calculation in several lemmas.

We follow the notations of the previous section. Let $S(t,s)$ be the local surface given in Section 4 tangent to the plane generated
by the vectors $X, Y$, satisfying $S(0,0) = A$, let $\bar{X}, \bar{Y}$ be the local extensions of the vectors $X, Y$ obtained by projecting by the map $\Pi$ the plane generated by $X, Y$ at $T_{A} \mathcal{N}$ into the tangent space of $\mathcal{N}$.

Let us define local extensions $\bar{X}_{i}$ of the vector fields $X_{i}$ in an analogous way we defined the extensions of $X, Y$: let $S_{k}$ be the plane generated by $X_{1}, X_{2}, ..,X_{k}$ and let us project by $ \Pi$ the tangent space of $S_{k}$ into $T\mathcal{N}$ by the differential of the projection into $\mathcal{N}$.

The terms $ \parallel \nabla_{\bar{Y}}\bar{X} \parallel^{2}, \langle \nabla_{\bar{X}}\bar{X}, \nabla_{\bar{Y}}\bar{Y}\rangle$ involve the Cristoffel symbols of the vector fields $\bar{X}, \bar{Y}$, at the point $A$ we have:
$$ \nabla_{\bar{X}_{k}}\bar{X}_{l} = \sum_{i=1}^{\infty} \Gamma_{kl}^{i}X_{i} $$
where $\Gamma_{kl}^{i}  = \langle \nabla_{\bar{X}_{k}}\bar{X}_{l}, \bar{X}_{i} \rangle $ is the Cristoffel coefficient. We follow
\cite{doCa} for the definitions and basic properties of Cristoffel coefficients.

The coefficient $\Gamma_{ij}^{k}$ can be calculated in terms of the coefficients of the first fundamental form of
the metric at $A$, the inner products $g_{ij} = \langle X_{i}, X_{j} \rangle $ by the following formula:
$$ \Gamma_{kl}^{i} = \frac{1}{2}g^{im}(g_{mk,l} + g_{ml,k} - g_{kl,m}) $$
where $g^{im}$ is the coefficient of the inverse of the first fundamental form of index $im$, $g_{mk,l}$ is the derivative
with respect to $\bar{X}_{l}$ of the coefficient $g_{mk}$, and the above notation is Einstein's convention for the sum on the index $m$.

The expression "inverse of the first fundamental form" requires some explanation since we are dealing with an infinite-dimensional Riemannian manifold. One natural rigorous approach is to evaluate the series $\sum_{i=1}^{\infty} \Gamma_{kl}^{i}X_{i}$ as the limit of its partial sums $\sum_{i=1}^{n} \Gamma_{kl}^{i}X_{i}$, that includes the Cristoffel coefficients in the subspace of $T_{A}\mathcal{N}$ generated by $\{ X_{1},X_{2}, ..,X_{n}\}$. The first fundamental form restricted to this subspace is a $n\times n$ matrix that, under our assumptions, is the identity. Its inverse is of course the identity. This allows us to define all the terms in the partial sum, then we take the limit as $n \rightarrow \infty$ to get the series. We shall prove that the series converges absolutely, so the above procedure provides the expression of $\nabla_{\bar{X}_{k}}\bar{X}_{l}$ as an infinite series.

In particular, since the basis $\{ X_{1},X_{2}, ..,X_{n},..\}$ is orthonormal, the indices in the sum of the expression of $\nabla_{\bar{X}_{k}}\bar{X}_{l}$ according to Einstein's convention just reduce to $ii$,$kk$, $ll$, depending on the case, and $g_{kl} = g^{kl} = \delta_{kl}$. So at the point $A$ we get the formula

$$ \Gamma_{kl}^{i} = \frac{1}{2}(g_{ik,l} + g_{il,k} - g_{kl,i}). $$

\begin{lemma} \label{cristoffel}
The term $g_{ik,l}$ at $A$, for any permutation of the indices, is
$$ g_{ik,l} = \int X_{i}X_{k}X_{l} d\mu. $$
Then,
$$ \nabla_{\bar{X}_{k}}\bar{X}_{l} = \frac{1}{2} \sum_{i=1}^{\infty} (\int X_{i}X_{k}X_{l} d\mu_{A}) X_{i}. $$
\end{lemma}

\begin{proof}

We have that $g_{ik,l} = \bar{X}_{l} \langle \bar{X}_{i}, \bar{X}_{k} \rangle = \bar{X}_{l} \int \bar{X}_{i}\bar{X}_{k} d\mu$.  By the Leibniz rule we have
$$ \bar{X}_{l} \int \bar{X}_{i}\bar{X}_{k} d\mu = \int \frac{\partial}{\partial \bar{X}_{l}} (\bar{X}_{i})\bar{X}_{k} d\mu +
\int  \bar{X}_{i}\frac{\partial}{\partial \bar{X}_{l}}(\bar{X}_{k}) d\mu + \int \bar{X}_{i}\bar{X}_{k} \bar{X}_{l} d\mu $$
where $\frac{\partial}{\partial \bar{X}_{l}} (\bar{X}_{i})$ is the derivative of the vector field $\bar{X}_{i}$ in the direction of $\bar{X}_{l}$.

Notice that Lemma \ref{derivative-coor} extends to the submanifolds $S_{k}$ for every $k \in \mathbb{N}$. So we have
\begin{enumerate}
\item $\frac{\partial}{\partial \bar{X}_{l}} (\bar{X}_{i}) = 0 $ if $l \neq i$,
\item $\frac{\partial}{\partial \bar{X}_{l}} (\bar{X}_{i}) = -1$ if $l = i$.
\end{enumerate}

In both cases, since $\int \bar{X}_{i} d\mu =0$ for every $i$, we get $g_{ik,l} = \int X_{i}X_{k}X_{l} d\mu$ as claimed.

The expression for $\nabla_{\bar{X}_{k}}\bar{X}_{l}$ is  straightforward from this formula.

\end{proof}

\begin{corollary} \label{cov-der}
Let us assume that $X=X_{1}$ and $ Y= X_{2}$ are the first two vectors of the orthonormal base $\{ X_{i} \}$. For the normalized potential  $A= S(0,0)$ we get the following expressions
$$ \nabla_{ \bar{X}_{1}} \bar{X}_{1} = \frac{1}{2} \sum_{i=1}^{\infty} (\int X_{1}^{2}X_{i} d\mu_{A}) X_{i}$$
$$ \nabla_{ \bar{X}_{2}} \bar{X}_{2} = \frac{1}{2} \sum_{i=1}^{\infty}(\int X_{2}^{2}X_{i} d\mu_{A}) X_{i}$$
$$\nabla_{ \bar{X}_{1}} \bar{X}_{2} = \frac{1}{2}\sum_{i=1}^{\infty} (\int X_{1}X_{2}X_{i} d\mu_{A}) X_{i}.$$

Moreover, for  any pair $X,Y  \in T_{A} \mathcal{N}$    the sums
$$\sum_{i=1}^\infty (   \int X \,Y\, X_i \,d \mu)^2\,\,\,\text{and}\,\,\,  \sum_{i=1}^\infty   \int X^2 X_i \,d \mu\,  \,\int Y^2 X_i \,d \mu $$
are both finite.
\end{corollary}

\begin{proof}
We consider an extension of the family $X_r$, $r \in \mathbb{N}$, to all $ L^2(\mu)$ and we get a complete orthonormal base of the vector space  $ L^2(\mu)$, given by  $X_r, Y_s$, $r,s \in \mathbb{N}$. The first three expressions in the statement are straightforward from Lemma \ref{cristoffel}.

Given two elements $X,Y  \in T_{A} \mathcal{N}$ consider $f=X\, Y = \sum_r a_r^f X_r + \sum_s b_s^f Y_s \in  L^2(\mu)$, then,
$$    (\int X \,Y\, X_i \,d \mu)^2 =               |a_i^f|^2. $$

It follows that $\sum_{i=1}^\infty (   \int X \,Y\, X_i \,d \mu)^2  = \sum_{i=1}^\infty |a_i^f|^2\leq \parallel f\parallel^2$ is finite.

Denote $g= X^2  =\sum_r a_r^g X_r + \sum_s b_s^g Y_s      $ and $h=Y^2= \sum_r a_r^h X_r + \sum_s b_s^h Y_s$.
Therefore,
$$\int g\, h \,d \mu= \sum_{i=1}^\infty    a_i^g   a_i^h +  \sum_{j=1}^\infty   b_j^g   b_j^h .   $$
Form this follows that   $\sum_{i=1}^\infty    a_i^g   a_i^h $ converges.
Note that  $ \int X^2 X_i \,d \mu\, = a_i^g $ and  $ \int Y^2 X_i \,d \mu\, = a_i^h.$
Then,
$$  \sum_{i=1}^\infty   \int X^2 X_i \,d \mu\,  \,\int Y^2 X_i \,d \mu     = \sum_{i=1}^\infty    a_i^g   a_i^h  $$
converges.

\end{proof}

\smallskip

Theorem \ref{main} follows from direct calculation applying Corollary \ref{cov-der} to the expression of $K(X,Y)$.

\section{A worked example in the Markov case:  an orthonormal basis for the kernel of the Ruelle operator} \label{worked}

From now on  $M=\{0,1\}^\mathbb{N}$ and we denote by $\mathcal{K}$  the set of stationary Markov probabilities taking values in $\{0,1\}$.

In this section, given a probability $\mu_A\in K$, we will  exhibit an orthonormal basis for the tangent space to $\mathcal{N}$ (the kernel of the Ruelle operator) at $\mu_A$.

Given a finite word $x =(x_1,x_2,...,x_k)\in \{0,1\}^k$, $k \in \mathbb{N}$, we denote by $[x]$ the associated cylinder set in
$M=\{0,1\}^\mathbb{N}$.

Consider an invariant
Markov probability $\mu$ obtained from a  row stochastic matrix $(P_{i,j})_{i,j=0,1}$ and an initial left invariant  vector of probability $\pi=(\pi_0,\pi_1)\in \mathbb{R}^2$.

Given $r \in (0,1)$ and $s\in (0,1)$ we denote
 \begin{equation} \label{tororo2} P= \left(
\begin{array}{cc}
P_{0,0} & P_{0,1}\\
P_{1,0} & P_{1,1}
\end{array}\right)=   \left(
\begin{array}{cc}
r & 1-r\\
1-s  & s
\end{array}\right) .
\end{equation}

In this way $(r,s)\in (0,1) \times (0,1)$ parameterize all {\bf row} stochastic matrices.

The explicit expression is
\begin{equation} \label{ppp} \mu [x_1,x_2,..,x_n] = \pi_{x_1} \,P_{x_1,x_2}\,P_{x_2,x_3}\,...\, P_{x_{n-1},x_n}.
\end{equation}

\begin{definition}
 Denote by $J:\{0,1\}^\mathbb{N} \to \mathbb{R}$ the   Jacobian associated to $P$. This function $J$ is  such that
is constant equal
$$J_{i,j}=\frac{\pi_i \,P_{i,j}}{\pi_j}$$
on the cylinder $[i,j]$, $i,j=0,1$.
\end{definition}

 According to our previous notation $\mu_A=\mu_{\log J}$ (which in this section will be called just $\mu$).

\begin{definition} The Ruelle operator for $\log J$ acts on continuous functions $\varphi$ and  is given by: for each $\varphi:M \to \mathbb{R}$, we get that
\begin{equation} \label{RuRui} \op{L}_{\log J} (\varphi) (x_1,x_2,x_3...)= \frac{\pi_0 \,P_{0,x_1}}{\pi_{x_1}} \varphi(0,x_1,x_2,...)+ \frac{\pi_1 \,P_{1,x_1}}{\pi_{x_1}} \varphi(1,x_1,x_2,...).
\end{equation}
\end{definition}

It is known that $\op{L}_{\log J}^* (\mu)=\mu.$ (see \cite{PP})

We  also consider the action of $\op{L}_{\log J}$ on $L^2(\mu)$ and we are interested in the kernel of this operator when acting on Holder functions.






\medskip

 Given a  finite word $x=(x_1,x_2,...,x_n)$, depending of the context $[x]$ will either denote  the word or  the corresponding  cylinder set  in $ \{0,1\}^\mathbb{N}.$  The empty word is also considered a finite word.

We start by recalling that, given a Markov probability $\mu$ on  $\{0,1\}^\mathbb{N}$, the family of Holder functions
\begin{align}
   e_{[x]} =\frac{1}{\sqrt{\mu([x])}} \sqrt{\frac{P_{x_n,1}}{P_{x_n,0}\,}} \, \mathfrak{1}_{[x0]} - \frac{1}{\sqrt{\mu([x])}} \sqrt{\frac{P_{x_n,0}}{P_{x_n,1} }} \, \mathfrak{1}_{[x1]},
    \label{eq52}
    \end{align}
where $x=(x_1,x_2,...,x_n)$ is a finite word on the symbols $\{0,1\}$, is an orthonormal set for $\op{L}^2 (\mu)$ (see \cite{KS} for a general expression  and \cite{CHLS}  for the specific expression we are using here). In order to get a (Haar) basis  we should add
$e_{[\emptyset]}^0=\frac{1}{ \sqrt{\mu ([0])}}  \mathfrak{1}_{[0]}$ and $ e_{[\emptyset]}^1=\frac{1}{ \sqrt{\mu ([1])}}  \mathfrak{1}_{[1]}$ to this family.

\smallskip
\begin{definition} Given a finite word $x=(x_1,x_2,...,x_n)$, we denote
\begin{equation}  \label{luc1} a_x=\frac{ \sqrt{\pi_{x_1}  }} {\sqrt{\pi_{0}}\sqrt{P_{0,x_1} }} \,\, e_{[0,x_1, x_2,..,x_n]} - \,\frac{ \sqrt{\pi_{x_1}  }} {\sqrt{\pi_{1}}\sqrt{P_{1,x_1} }}\,\,e_{[1,x_1, x_2,..,x_n]},
\end{equation}

\end{definition}

 It will follow from \eqref{vade1} and \eqref{vade2} that the terms   $| a_x |$ are uniformly bounded away from zero (the minimum value is $2$). Moreover, they depend just on the first letter  of the word $[x]$.

\begin{definition}
We denote by
\begin{equation} \label{pede} \hat{a}_x=\frac{1}{|a_x|}\, a_x,
\end{equation}
the normalization of $a_x$.
\end{definition}
\smallskip

In order to get a complete orthonomal set for the kernel of the Ruelle operator we will have to add to the functions of the form \eqref{pede} two more functions: $\hat{a}_{[\emptyset]}^0$ and $ \hat{a}_{[\emptyset]}^0$
to be set in Definition \ref{ttw}. To show this result  is our main goal in this section. This family will be later denoted by $\mathcal{F}$ according to Definition \ref{YT}.

In this direction, we first consider the problem of exhibiting an orthogonal family which is a basis for the kernel of the Ruelle operator, and later via normalization, we will get a complete orthonormal family  which is a basis for the kernel  of the Ruelle operator.

Following this line of  reasoning, one of our main tasks in this section is to show  the following:

\begin{theorem} \label{ort-basis}
The  family $ a_x$,
indexed by all  words $x=(x_1,x_2,...,x_n)$, plus the two functions $e_{[\emptyset]}^0$ and $e_{[\emptyset]}^1$, determine an orthogonal set on  the kernel of the Ruelle operator $\op{L}_{\log J}$.
\end{theorem}

We will address first the issue related to the functions $a_x$, and later to
questions regarding the functions $e_{[\emptyset]}^0$ and $e_{[\emptyset]}^1$.

First note that as the family $e_{[x]}$, where $x$ is a finite word, is orthonormal, then, $a_x$, where $x$ is a finite word with size bigger or equal to $1$, is an orthogonal family.

Indeed, it follows from the fact that the family $e_{[x]}$ defined by (\ref{eq52}) is orthogonal, and the bilinearity of the inner product, that
$$ \langle a_x,a_z \rangle=$$
$$
\langle \frac{ \sqrt{\pi_{x_1}  }} {\sqrt{\pi_{0}P_{0,x_1} }}  e_{[0,x]} - \frac{ \sqrt{\pi_{x_1}  }} {\sqrt{\pi_{1}P_{1,x_1} }}e_{[1,x]},\frac{ \sqrt{\pi_{z_1}  }} {\sqrt{\pi_{0}P_{0,z_1} }}  e_{[0,z]} - \frac{ \sqrt{\pi_{z_1}  }} {\sqrt{\pi_{1}P_{1,z_1} }}e_{[1,z]}\rangle  =0,$$
for all $x=(x_1,x_2,..,x_n)\neq z= (z_1,z_2,...,z_k)$.

\medskip

We shall subdivide the proof of Theorem \ref{ort-basis} into several steps. First of all, we have that:


\begin{proposition} \label{esta}

Given $x=[x_1,x_2,..,x_n]$ with a size larger or equal to $1$,

\begin{equation}  \label{lui}  \op{L}_{\log J} \,( e_{[x_1,x_2,..,x_n]} )  =\frac{\sqrt{\pi_{x_1}}}{ \sqrt{\pi_{x_2}  }} \sqrt{P_{x_1,x_2}} e_{[x_2,x_3,..,x_n]}.
\end{equation}

From this follows that all elements in  the orthogonal family
\begin{equation}  \label{luc} a_x=\frac{ \sqrt{\pi_{x_1}  }} {\sqrt{\pi_{0}}\sqrt{P_{0,x_1} }} \,\, e_{[0,x_1, x_2,..,x_n]} - \,\frac{ \sqrt{\pi_{x_1}  }} {\sqrt{\pi_{1}}\sqrt{P_{1,x_1} }}\,\,e_{[1,x_1, x_2,..,x_n]},
\end{equation}
indexed by words $x=(x_1,x_2,...,x_n)$, are in the kernel of the Ruelle operator $\op{L}_{\log J}$.

\end{proposition}

\begin{proof}
We consider finite words $x$  with size larger or equal to $1$.

 Indeed, given the word  $x=(x_1,x_2,..,x_n)$, let $L = \op{L}_{\log J} \,( e_{[x_1,x_2,..,x_n]} )$, then we get
\begin{eqnarray*}
 L & =& \frac{\pi_{x_1}}{ \pi_{x_2}}\, P_{x_1,x_2}\,[\, \frac{1}{\sqrt{\mu([x])}} \sqrt{\frac{P_{x_n,1}}{P_{x_n,0}\,}}\, \mathfrak{1}_{[x_2,..,x_n,0]} \\
& -&  \frac{1}{\sqrt{\mu([x])}} \sqrt{\frac{P_{x_n,0}}{P_{x_n,1} }} \, \mathfrak{1}_{[x_2,..,x_n,1]}\,] \\
& = & \frac{\pi_{x_1}}{ \pi_{x_2}}\sqrt{P_{x_1,x_2}}\,[\, \frac{\sqrt{P_{x_1,x_2}}}{\sqrt{\mu([x])}} \sqrt{\frac{P_{x_n,1}}{P_{x_n,0}\,}}\, \mathfrak{1}_{[x_2,..,x_n,0]} \\
& - & \frac{\sqrt{P_{x_1,x_2}}}{\sqrt{\mu([x])}} \sqrt{\frac{P_{x_n,0}}{P_{x_n,1} }} \, \mathfrak{1}_{[x_2,..,x_n,1]}\,]
\end{eqnarray*}
This is equal to
$$ \frac{\pi_{x_1}}{ \pi_{x_2}}\frac{\sqrt{\pi_{x_2}}}{ \sqrt{\pi_{x_2}  }} \sqrt{P_{x_1,x_2}}\,\, \frac{\sqrt{P_{x_1,x_2}}}{\sqrt{ \pi_{x_1} \,P_{x_1,x_2}\,P_{x_2,x_3}\,...\, P_{x_{n-1},x_n}}} \sqrt{\frac{P_{x_n,1}}{P_{x_n,0}\,}}\, \mathfrak{1}_{[x_2,..,x_n,0]}$$
$$-\frac{\pi_{x_1}}{ \pi_{x_2}}\frac{\sqrt{\pi_{x_2}}}{ \sqrt{\pi_{x_2} }} \sqrt{P_{x_1,x_2}}\frac{\sqrt{P_{x_1,x_2}}}{\sqrt{ \pi_{x_1} \,P_{x_1,x_2}\,P_{x_2,x_3}\,...\, P_{x_{n-1},x_n}}} \sqrt{\frac{P_{x_n,0}}{P_{x_n,1} }} \, \mathfrak{1}_{[x_2,..,x_n,1]}\,$$
which is equivalent to
$$\frac{\pi_{x_1}}{ \pi_{x_2}} \frac{\sqrt{\pi_{x_2}}}{ \sqrt{\pi_{x_1}  }} \sqrt{P_{x_1,x_2}}\,\, \frac{1}{\sqrt{ \pi_{x_2} \,P_{x_2,x_3}\,...\, P_{x_{n-1},x_n}}} \sqrt{\frac{P_{x_n,1}}{P_{x_n,0}\,}}\, \mathfrak{1}_{[x_2,..,x_n,0]}$$
$$-\frac{\pi_{x_1}}{ \pi_{x_2}}\frac{\sqrt{\pi_{x_2}}}{ \sqrt{\pi_{x_1}  }} \sqrt{P_{x_1,x_2}} \frac{1}{\sqrt{ \pi_{x_2} \,\,P_{x_2,x_3}\,...\, P_{x_{n-1},x_n}}} \sqrt{\frac{P_{x_n,0}}{P_{x_n,1} }} \, \mathfrak{1}_{[x_2,..,x_n,1]}\,$$
that yields
\begin{eqnarray*}
L& = &  \frac{\sqrt{\pi_{x_1}}}{ \sqrt{\pi_{x_2}  }} \sqrt{P_{x_1,x_2}} \frac{1}{ \sqrt{ \mu [x_2,x_3,...,x_n]}}  [ \sqrt{\frac{P_{x_n,1}}{P_{x_n,0}\,}}\, \mathfrak{1}_{[x_2,..,x_n,0]}  -  \sqrt{\frac{P_{x_n,0}}{P_{x_n,1} }} \, \mathfrak{1}_{[x_2,..,x_n,1]} ]\\
& = & \frac{\sqrt{\pi_{x_1}}}{ \sqrt{\pi_{x_2}  }} \sqrt{P_{x_1,x_2}}\,\, e_{x_2,x_3,..,x_n}.
\end{eqnarray*}

Then,
$$  \op{L}_{ \log J} (\,\frac{ \sqrt{\pi_{x_2}  }} {\sqrt{\pi_{x_1}}\sqrt{P_{x_1,x_2} }}  \, e_{[x_1,x_2,..,x_n]} )  $$
$$=\frac{1}{ \sqrt{ \mu [x_2,x_3,...,x_n]}}[ \sqrt{\frac{P_{x_n,1}}{P_{x_n,0}\,}}\, \mathfrak{1}_{[x_2,..,x_n,0]} - \sqrt{\frac{P_{x_n,0}}{P_{x_n,1} }} \, \mathfrak{1}_{[x_2,..,x_n,1]} ]= e_{ [x_2,x_3,..,x_n]}$$

and therefore,
$$  \op{L}_{ \log J} (\,\frac{ \sqrt{\pi_{x_2}  }} {\sqrt{\pi_{0}}\sqrt{P_{0,x_2} }}  \, e_{[0,x_2,..,x_n]} )  =\op{L}_{ \log J} (\,\frac{ \sqrt{\pi_{x_2}  }} {\sqrt{\pi_{1}}\sqrt{P_{1,x_2} }}  \, e_{[1,x_2,..,x_n]} ) .$$

For each finite word $(x_1,x_2,..,x_n)$ denote
\begin{eqnarray*}
 a_x & =& \frac{ \sqrt{\pi_{x_1}  }} {\sqrt{\pi_{0}}\sqrt{P_{0,x_1} }} \,\, e_{[0,x_1, x_2,..,x_n]} - \,\frac{ \sqrt{\pi_{x_1}  }} {\sqrt{\pi_{1}}\sqrt{P_{1,x_1} }}\,\,e_{[1,x_1, x_2,..,x_n]}\\
& = &  \frac{ \sqrt{\pi_{x_1}  }} {\sqrt{\pi_{0}}\sqrt{P_{0,x_1} }} \,   \frac{1}{\sqrt{\mu([0 x])}}\,[ \sqrt{\frac{P_{x_n,1}}{P_{x_n,0}\,}} \, \mathfrak{1}_{[0x0]} -  \sqrt{\frac{P_{x_n,0}}{P_{x_n,1} }} \, \mathfrak{1}_{[0x1]} ]
\end{eqnarray*}
 \begin{equation} \label{tororo15}  -\,\frac{ \sqrt{\pi_{x_1}  }} {\sqrt{\pi_{1}}\sqrt{P_{1,x_1} }} \,\frac{1}{\sqrt{\mu([1 x])}} \,[\,\sqrt{\frac{P_{x_n,1}}{P_{x_n,0}\,}} \, \mathfrak{1}_{[1x0]} -  \sqrt{\frac{P_{x_n,0}}{P_{x_n,1} }} \, \mathfrak{1}_{[1x1]}  ]. \end{equation}

From the above reasoning, it follows that the family $a_x$ is in the kernel of the Ruelle operator.
\end{proof}

For words, $x$ of size greater or equal to $1$ the function $a_x$ is constant in cylinder sets of size equal to the length of $x$ plus 2.

As an example, we get that
\begin{eqnarray*}
 a_0 & =&\frac{ \sqrt{\pi_{0}  }} {\sqrt{\pi_{0}}\sqrt{P_{0,0} }} \,   \frac{1}{\sqrt{\mu([0 0])}}\,[ \sqrt{\frac{P_{0,1}}{P_{0,0}\,}} \, \mathfrak{1}_{[000]} -  \sqrt{\frac{P_{0,0}}{P_{0,1} }} \, \mathfrak{1}_{[001]} ]
\end{eqnarray*}
 \begin{equation} \label{tororo1514}  -\,\frac{ \sqrt{\pi_{0}  }} {\sqrt{\pi_{1}}\sqrt{P_{1,0} }} \,\frac{1}{\sqrt{\mu([1 0])}} \,[\,\sqrt{\frac{P_{0,1}}{P_{0,0}\,}} \, \mathfrak{1}_{[100]} -  \sqrt{\frac{P_{0,0}}{P_{0,1} }} \, \mathfrak{1}_{[101]}  ] \end{equation}
is constant on cylinders of size $3.$

Note that if $x$ and $z$ are different words, then, $1x$, $0x$, $0z$ and $1z$ are four different words.

\medskip

Note that
\begin{align}
    e_{[x]}^2  =\frac{1}{\mu([x])} \frac{P_{x_n,1}}{P_{x_n,0}\,} \, \mathfrak{1}_{[x0]} + \frac{1}{\mu([x])} \frac{P_{x_n,0}}{P_{x_n,1} } \, \mathfrak{1}_{[x1]}.
    \label{eq524}
    \end{align}

Therefore,
\begin{eqnarray*} \label{lucato}
 a_x^2=a_x \, a_x & = & \frac{\pi_{x_1} } {\pi_{0} P_{0,x_1} } \,\, e_{[0,x_1, x_2,..,x_n]}^2  + \,\frac{\pi_{x_1}  } {\pi_{1}P_{1,x_1} }\,\,e_{[1,x_1, x_2,..,x_n]}^2 \\
& = &  [ \,  \frac{\pi_{x_1}  } {\pi_{0} P_{0,x_1} } \frac{1}{\mu([0 x])} \frac{P_{x_n,1}}{P_{x_n,0}\,} \mathfrak{1}_{[0x0]}  +   \,\frac{\pi_{x_1}  } {\pi_{0} P_{0,x_1} } \frac{1}{\mu([0 x])} \frac{P_{x_n,0}}{P_{x_n,1}\,}  \mathfrak{1}_{[0x1]}\,]
\end{eqnarray*}
\begin{equation}  \label{lucato} +[\, \,\frac{\pi_{x_1} } {\pi_{1} P_{1,x_1} } \frac{1}{\mu([1 x])} \frac{P_{x_n,1}}{P_{x_n,0}\,} \, \mathfrak{1}_{[1x0]}
\,+ \frac{\pi_{x_1}  } {\pi_{1} P_{1,x_1} } \frac{1}{\mu([1 x])} \frac{P_{x_n,0}}{P_{x_n,1}\,} \, \mathfrak{1}_{[1x1]}\,] .
\end{equation}

From the above it follows that
\begin{equation} \label{estanao} | a_x |=\sqrt{\frac{\pi_{x_1} } {\pi_{0} P_{0,x_1} } +   \frac{\pi_{x_1}  } {\pi_{1}P_{1,x_1} } }.
\end{equation}

Using the notation in the variables  $r,s$ for the matrix $P$,  when $x_1=0$ we get
\begin{equation} \label{vade1}|a_x|= \sqrt{\frac{\pi_{x_1} } {\pi_{0} P_{0,x_1} } +   \frac{\pi_{x_1}  } {\pi_{1}P_{1,x_1} } }=(\sqrt{ r\, (1-r)})^{-1}
\end{equation}
and when $x_1=1$ we get
\begin{equation} \label{vade2}|a_x|= \sqrt{\frac{\pi_{x_1} } {\pi_{0} P_{0,x_1} } +   \frac{\pi_{x_1}  } {\pi_{1}P_{1,x_1} } }=(\sqrt{ s\, (1-s)})^{-1}.
\end{equation}

\smallskip

\begin{definition}
We denote by $\tilde{\mathcal{F}}$ the orthonormal set of normalized functions $ \hat{a}_x$,
where $x= (x_1,x_2,...,x_k)$ is a finite word with size equal or larger than $1$.
\end{definition}

As we mentioned before, we will  have to add two more functions in order to get a basis (a completely orthogonal set in the Hilbert space) for the kernel of the Ruelle operator $\op{L}_{\log J}$.

We claim that the orthogonal pair (constant in cylinders of size $2$)
 $$ V_1 = \pi_1 P_{1,0} \mathfrak{1}_{[00]} \,\, -  \,\,   \pi_0 P_{0,0} \mathfrak{1}_{[10]}$$
 \begin{equation} \label{tororo}
V_2 = \pi_0 P_{0,1} \mathfrak{1}_{[11]} \,\, -  \,\,   \pi_1 P_{1,1} \mathfrak{1}_{[01]}\end{equation}
is in the kernel of the Ruelle operator  (see Proposition \ref{yyu}). The functions $V_1$ and $V_2$ are orthogonal to all $\hat{a}_x \in \tilde{\mathcal{F}}$ and they depend on the first two coordinates $x_1,x_2$ of $x$.

The vectors $\hat{V}_1 =  \frac{1}{|V_1|  }$ and $\hat{V}_2 =  \frac{1}{|V_2|  }$ are normalized and orthogonal to all $\hat{a}_x$.
This claim will be proved in Proposition \ref{yyu}.

\medskip
One can  show that
\begin{equation} \label{V1}|V_1| = \sqrt{      \pi_1^2 P_{1,0}^2 \pi_0 P_{0,0}            +   \pi_0^2 P_{0,0}^2 \, \pi_1 P_{1,0} }= \sqrt{ \frac{(1-r) r (s-1)^3}{(-2 + r + s)^3}  }
\end{equation}
and


\begin{equation} \label{V2}|V_2|= \sqrt{      \pi_0^2 P_{0,1}^2 \pi_1 P_{1,1}            +   \pi_1^2 P_{1,1}^2 \, \pi_0 P_{0,1} } = \sqrt{ \frac{(1-s) s (r-1)^3}{(-2 + r + s)^3}  }.
\end{equation}

\medskip

\begin{definition} \label{ttw}
 As a matter of notation we denote $\hat{a}_{[\emptyset]}^0= \hat{V_1}$ and
$\hat{a}_{[\emptyset]}^1= \hat{V_2}$.
\end{definition}
\medskip

These two functions are constant in cylinders of size $2$

\begin{definition} \label{YT}
We  add $\hat{a}_{[\emptyset]}^0$ and $\hat{a}_{[\emptyset]}^1$ to the family
$\tilde{\mathcal{F}}$ in   order to get the family $\mathcal{F}$.

\end{definition}
\medskip

{\bf  \begin{remark} \label{luit} The elements  in $\mathcal{F}$ range in all possible words of size larger or equal to zero.    A generic element in $\mathcal{F}$ is denoted by $\hat{a}_x$, and by this we mean that $\hat{a}_x$ can eventually represent  $\hat{a}_{[\emptyset]}^0$ or  $\hat{a}_{[\emptyset]}^1.$
\end{remark}

}

\medskip

\begin{proposition} \label{yyu}
The orthogonal  pair

 $$ V_1 = \pi_1 P_{1,0} \mathfrak{1}_{[00]} \,\, -  \,\,   \pi_0 P_{0,0} \mathfrak{1}_{[10]}$$
 \begin{equation} \label{tororo}
V_2 = \pi_0 P_{0,1} \mathfrak{1}_{[11]} \,\, -  \,\,   \pi_1 P_{1,1} \mathfrak{1}_{[01]}\end{equation}
is such that, each one  of them is orthogonal to the other elements $\hat{a}_x$, where $x$ ranges in  all finite words with size bigger or equal to $1$. $V_1$ and $V_2$ are on the kernel of the Ruelle operator $\op{L}_{\log J}$.

\end{proposition}

\begin{proof}

Note first that $\mathfrak{1}_{[00]} $ is orthogonal to all $a_x$, where $x=(x_1,x_2,...,x_n)$ is a word with size equal or greater then $1$.  This claim  follows from \eqref{tororo15}.  Indeed, if $x_1=0$, we get that
$$\langle \mathfrak{1}_{[00]} ,  \sqrt{\frac{P_{x_n,1}}{P_{x_n,0}\,}} \, \mathfrak{1}_{[0x0]} -  \sqrt{\frac{P_{x_n,0}}{P_{x_n,1} }} \, \mathfrak{1}_{[0x1]}\rangle= $$
$$ \sqrt{P_{x_n,1}}\,  \pi_0 \,  P_{0,x_1 }\, P_{x_1,x_2 } \,...P_{x_{n-1},x_n } \sqrt{ P_{x_n,0 }}  \,\,-$$
$$  \sqrt{P_{x_n,0}}\,  \pi_0 \,  P_{0,x_1 }\, P_{x_1,x_2 } \,...P_{x_{n-1},x_n } \sqrt{ P_{x_n,1 }}=0. $$

If $x_1=1$ the claim follows at once.

Using the same reasoning one can show that $\mathfrak{1}_{[01]},\mathfrak{1}_{[10]},\mathfrak{1}_{[11]} $ are orthogonal to all $a_x$, where length of $x$ is bigger than zero. It follows that linear combinations of this functions are also orthogonal to all $a_x$.
It follows that $V_1$ and $V_2$ are orthogonal to all $a_x$, where the length of $x$ is bigger than zero.

We will show that
 $V_1$ is in the kernel of the Ruelle operator  (for $V_2$ the proof is similar).
  Given $y=(y_1,y_2,...,y_n,..)\in M$, suppose first that $y_1=0$, then,  we get
$$\op{L}_{\log J} (V_1)= \op{L}_{\log J}  ( \pi_1 P_{1,0} \mathfrak{1}_{[00]} \,\, -  \,\,   \pi_0 P_{0,0} \mathfrak{1}_{[10]}) (y)= $$
$$  \pi_1 P_{1,0} ( J_{0,y_1} \mathfrak{1}_{[00]}(0,y_1,y_2,...) + J_{1,y_1} \mathfrak{1}_{[00]} (1,y_1,y_2,...) ) -  $$
$$  \pi_0 P_{0,0} ( J_{0,y_1} \mathfrak{1}_{[10]}(0,y_1,y_2,...) + J_{1,y_1} \mathfrak{1}_{[10]} (1,y_1,y_2,...) ) =  $$
$$ \pi_1 P_{1,0}  J_{0,0} -  \pi_0 P_{0,0}  J_{1,0}= \pi_1 P_{1,0} \frac{\pi_0 P_{0,0} }{\pi_0} -   \pi_0 P_{0,0}\frac{\pi_1 P_{1,0} }{\pi_0} =0.$$

In the case $y_1=1$, we get
$$\op{L}_{\log J} (V_1)=  \pi_1 P_{1,0} ( J_{0,y_1} \mathfrak{1}_{[00]}(0,y_1,y_2,...) + J_{1,y_1} \mathfrak{1}_{[00]} (1,y_1,y_2,...) ) -  $$
$$  \pi_0 P_{0,0} ( J_{0,y_1} \mathfrak{1}_{[10]}(0,y_1,y_2,...) + J_{1,y_1} \mathfrak{1}_{[10]} (1,y_1,y_2,...) ) = 0. $$

\end{proof}

\begin{remark} \label{jaque}
 A  function of the form $w=r_1  \mathfrak{1}_{[0]} + r_2\,
\mathfrak{1}_{[1]}$ is in the kernel of $\op{L}_{\log J}$ only in the case where $P_{01}=(1-r)= s=P_{11}$. In this case
\begin{equation} \label{w} w =(1-r) \mathfrak{1}_{[0]} - (1-s)\,
\mathfrak{1}_{[1]}
\end{equation}
is such that
$\op{L}_{\log J} (w)=0.$

We do not have to take into account in our future reasoning this function because
$$w = \frac{1}{r} V_1 + \frac{1}{r-1} V_2.$$

\end{remark}

\begin{proposition} \label{erer} The family of elements in $ \mathcal{F}$ (see Definition \ref{YT} and Remark \ref{luit}) is an orthonormal basis for the kernel of the Ruelle operator
$\op{L}_{\log J}$.
\end{proposition}

\begin{proof}

From Proposition \ref{esta} we know that given $x=[x_1,x_2,..,x_n]$
\begin{equation} \label{muq} \op{L}_{\log J} \,( e_{[x_1,x_2,..,x_n]} )  =\frac{\sqrt{\pi_{x_1}}}{ \sqrt{\pi_{x_2}  }} \sqrt{P_{x_1,x_2}} e_{[x_2,x_3,..,x_n]}.
\end{equation}

Suppose $\varphi$ is in the kernel of the Ruelle operator. We will show that $\varphi$ can be expressed as an infinite linear combination of the normalized functions $\hat{a}_x \in \mathcal{F}$.

We can express $\varphi$ as
$$\varphi = \sum_{\text{words}\,\,y}\, c_y  e_{[y]}  . $$

When applying $\op{L}_{\log J}$ on $\varphi$ we separate the infinite sum in subsums of the form
$$ \,c_{0,\alpha_2,..,\alpha_n } e_{[0,\alpha_2,..,\alpha_n]}    +   c_{1,\alpha_2,..,\alpha_n } e_{[1,\alpha_2,..,\alpha_n]}.$$

 Assuming that $\varphi$ is in the kernel of $\op{L}_{\log J}$, we get from (\ref{muq}) that
\begin{eqnarray*}
0 & = &  \op{L}_{\log J}(\, \sum_n  \,\sum_{\alpha_2,..,\alpha_n} [\,c_{0,\alpha_2,..,\alpha_n } e_{[0,\alpha_2,..,\alpha_n]}    +   c_{1,\alpha_2,..,\alpha_n } e_{[1,\alpha_2,..,\alpha_n]}\,]\,) \\
& = & \, \sum_n\, \sum_{\alpha_2,..,\alpha_n} [ \frac{\sqrt{\pi_{0}}}{ \sqrt{\pi_{\alpha_2}  }} \sqrt{P_{0,\alpha_2}}  \,c_{0,\alpha_2,..,\alpha_n } e_{[\alpha_2,..,\alpha_n]}    + \frac{\sqrt{\pi_{1}}}{ \sqrt{\pi_{\alpha_2}  }} \sqrt{P_{1,\alpha_2}}  c_{1,\alpha_2,..,\alpha_n } e_{[\alpha_2,..,\alpha_n]}\,]\, \\
& = & \, \sum_n\, \sum_{\alpha_2,..,\alpha_n} [ \frac{\sqrt{\pi_{0}}}{ \sqrt{\pi_{\alpha_2}  }} \sqrt{P_{0,\alpha_2}}  \,c_{0,\alpha_2,..,\alpha_n }   + \frac{\sqrt{\pi_{1}}}{ \sqrt{\pi_{\alpha_2}  }} \sqrt{P_{1,\alpha_2}}  c_{1,\alpha_2,..,\alpha_n } ]\,\,\,e_{[\alpha_2,..,\alpha_n]}\,\,.
\end{eqnarray*}

Then, for fixed $n$ and $(\alpha_2,\alpha_3,..,\alpha_n)$
$$ \frac{\sqrt{\pi_{0}}}{ \sqrt{\pi_{\alpha_2}  }} \sqrt{P_{0,\alpha_2}}  \,c_{0,\alpha_2,..,\alpha_n }  = - \frac{\sqrt{\pi_{1}}}{ \sqrt{\pi_{\alpha_2}  }} \sqrt{P_{1,\alpha_2}}  c_{1,\alpha_2,..,\alpha_n } ,$$
which means
$$   \,c_{0,\alpha_2,..,\alpha_n }  = - \frac{\sqrt{\pi_{1}}}{ \sqrt{\pi_{\alpha_2}  }} \sqrt{P_{1,\alpha_2}} \frac{\sqrt{\pi_{\alpha_2}}}{ \sqrt{\pi_{0}  }} \frac{1}{\sqrt{P_{0,\alpha_2}}}  c_{1,\alpha_2,..,\alpha_n } .$$

Then, the sum $$c_{0,\alpha_2,..,\alpha_n } e_{[0,\alpha_2,..,\alpha_n]}  +   c_{1,\alpha_2,..,\alpha_n } e_{[1,\alpha_2,..,\alpha_n]} $$
is equal to
$$  -\,    c_{1,\alpha_2,..,\alpha_n }\,[  \frac{\sqrt{\pi_{1}}}{ \sqrt{\pi_{\alpha_2}  }} \sqrt{P_{1,\alpha_2}} \frac{\sqrt{\pi_{\alpha_2}}}{ \sqrt{\pi_{0}  }} \frac{1}{\sqrt{P_{0,\alpha_2}}}              e_{[0,\alpha_2,..,\alpha_n]}  - e_{[1,\alpha_2,..,\alpha_n]}].
$$

Multiplying the above expression by $\frac{\sqrt{\pi_{\alpha_2}}}{ \sqrt{\pi_{1}  }} \frac{1}{\sqrt{P_{1,\alpha_2}}}       $ we get
$$ \frac{\sqrt{\pi_{\alpha_2}}}{ \sqrt{\pi_{1}  }} \frac{1}{\sqrt{P_{1,\alpha_2}}}    \,[\,c_{0,\alpha_2,..,\alpha_n } e_{[0,\alpha_2,..,\alpha_n]}    +   c_{1,\alpha_2,..,\alpha_n } e_{[1,\alpha_2,..,\alpha_n]}]$$
 which is equal to
\begin{eqnarray*}
  & - & \,    c_{1,\alpha_2,..,\alpha_n }\,[  \frac{\sqrt{\pi_{\alpha_2}}}{ \sqrt{\pi_{0}  }} \frac{1}{\sqrt{P_{0,\alpha_2}}}        e_{[0,\alpha_2,..,\alpha_n]}  - \frac{\sqrt{\pi_{\alpha_2}}}{ \sqrt{\pi_{1}  }} \frac{1}{\sqrt{P_{1,\alpha_2}}}     e_{[1,\alpha_2,..,\alpha_n]}]\\
& = &  -\,    c_{1,\alpha_2,..,\alpha_n }\,a_{[a_2,..,a_n]}.
\end{eqnarray*}

Then, $(\,c_{0,\alpha_2,..,\alpha_n } e_{[0,\alpha_2,..,\alpha_n]}    +   c_{1,\alpha_2,..,\alpha_n } e_{[1,\alpha_2,..,\alpha_n]}\,)$ is a multiple of the function $\hat{a}_{[\alpha_2,..,\alpha_n]}.$
Since the above reasoning was done for a generic choice of $(\alpha_2,\alpha_3,..,\alpha_n),$  we conclude that for each $n$ the sum
$\sum_{\text{words}\,\,y\, \text{of length}\, n}\, c_y  e_{[y]}$ can be expressed as a linear combination of elements $\hat{a}_x$, using words of length $n-1$, $n>1$.

From this follows that each element in the kernel of $\op{L}_{\log J} $ can be expressed as an infinite linear combination of the functions $\hat{a}_x$.

\end{proof}

Theorem \ref{ort-basis} follows from the combination of Propositions \ref{esta} and \ref{erer}

\medskip

The above shows that the set $\mathcal{F}$ is a complete orthonormal set  for the kernel of the Ruelle operator acting on $\op{L}^2(\mu).$

\medskip

\medskip

\section{A worked example in the Markov case: preliminary calculations of the terms in $K(X,Y)$} \label{eet}

In this section, we shall devote ourselves to the calculation of the sectional curvatures in the case of Markov stationary probabilities on $M=\{0,1\}^\mathbb{N}$.

We denote by $K\subset \mathcal{N},$ the set of  Markov invariant probabilities.
We will consider this section  the sectional curvature for points in $\mathcal{K}$ for general orthogonal  pairs of tangent vectors to $\mathcal{N}.$

We can also consider $\mathcal{K}$ as a two-dimensional manifold  carrying the Riemannian structure induced by $\mathcal{N}.$ From this point of view, there exists  just one orthonormal pair to be considered.
One of our main results (see Theorem \ref{meme}) claims that for the two dimensional manifold  $K,$   for any point in $\mathcal{K}$, the sectional curvature for the pair of tangent vectors to $\mathcal{K}$  is always zero.

\medskip

 We will consider in our reasoning  the empty word as a regular word.
$\hat{a}_\emptyset^0 $ and $\hat{a}_\emptyset^1 $ are two elements in $\mathcal{F}$ associated to the empty word.

\begin{definition} We say that $z$ is a  subprefix of $x$, if $x$ and $z$ satisfy
$$  [x]=[x_1,x_2,...x_k,x_{k+1},..., x_n] \subset [z]=[x_1,x_2,...,x_k],$$
where $n \geq k$.
\end{definition}

Note that, even when $z$ is not a subprefix of $x$ and $x$ is not a subprefix of $z$, they can share some common subprefix. Note also that if $x$ and $z$ do not share a common subprefix, then $z$ is not a subprefix of $x$ and $x$ is not a subprefix of $z$.

If $[x]=[z]$, then, $x$ is a subprefix of $z.$

\begin{definition} We say that $z$ is a  strict subprefix of $x$, if $x$ and $z$ satisfy
$$  [x]=[x_1,x_2,...x_k,x_{k+1},..., x_n] \subset [z]=[x_1,x_2,...,x_k],$$
where $n > k$.
\end{definition}

Two different  words with the same length can not be subprefix of each other. If the length of $z$ is strictly larger than the length of $x$, then,  $z$ can not be a subprefix of $x$.

\begin{definition} Given the finite words $x,z$ we denote by $D[x,z]$ the set of all finite words $y$ such that are  subprefix of $x$ and $z$.
\end{definition}

If for example   $ x=(0,0,0)$ and $z=(0,0,0,1)$, then
$$D[x,z]=\{\hat{a}_\emptyset^0,(0), (0,0),(0,0,0)\}.$$
In the case
$x=(0,1,0,0,1) $ and $z=(0,1,1)$ we get that $D[x,z]= \{\hat{a}_\emptyset^0(0),(0,1),\}.$

Another example:  $D[a_{0,0}, \hat{a}_\emptyset^0]=\{\hat{a}_\emptyset^0 \}$ and  $D[a_{0,0}, \hat{a}_\emptyset^1]=\emptyset.$
\smallskip

Note that in the case $z=(z_1,z_2,..,z_k)$ is a subprefix of $x=(x_1,x_2,..,x_n)$, $n>k$, then, $z_1=x_1$.  Then, it follows from \eqref{estanao} that $|a_x|=|a_z|$.

\begin{proposition} \label{louc1} Assume that $x$ is not a subprefix of $z$ and  $z$ is not a subprefix of $x$. Then,
$$a_x \, a_z=0.$$

\end{proposition}

\begin{proof}

Note that $a_z$ is a linear combination of  $\mathfrak{1}_{[0z0]}, \mathfrak{1}_{[0z1]}, \mathfrak{1}_{[1z0]}$ and $\mathfrak{1}_{[1z1]}$. As $a_x$ is a linear combination of  $\mathfrak{1}_{[0x0]}, \mathfrak{1}_{[0x1]}, \mathfrak{1}_{[1x0]}$ and $\mathfrak{1}_{[1x1]}$ the result follows.

\end{proof}

Note that the hypothesis of the last proposition is equivalent to saying that the
cylinders $[x]$ and $[z]$ are disjoint.

\begin{corollary} \label{louc} Given a word $x$ assume that $x$ is not a subprefix of $y$ and $y$ is not a subprefix of $x$. Then,
$$\hat{a}_x^2 \, \hat{a}_y=0.$$

\end{corollary}

\begin{proof} This follows from at once from Proposition \ref{louc1}.
\end{proof}

Note that if $x$ and $y$ have the same length, but they are different, then $\int \hat{a}_x^2 \,\hat{a}_y \, d \mu=0.$
\medskip

From Proposition \ref{louc1} it follows:

\begin{corollary}   \label{Gla1} Assume that $x$ is not a subprefix of $z$ and  $z$ is not a subprefix of $x$.
Then, we get that the products (part of the first sum contribution in (\ref{tororo172})) satisfy
\begin{equation}
\int \hat{a}_x \,\hat{a}_z \, \hat{a}_y d \mu=0,
\end{equation}
for all word $y$.
\end{corollary}

Remember that $\mathcal{F}$ (defined in last section) is the set of all  functions of the form
\begin{equation} \label{tororo66} \hat{a}_x = \frac{1}{ \sqrt{\frac{\pi_{x_1} } {\pi_{0} P_{0,x_1} } +   \frac{\pi_{x_1}  } {\pi_{1}P_{1,x_1} } }}  a_x,\end{equation}
where $x= (x_1,x_2,...,x_k)$ is a general finite word,
plus the functions $\hat{a}_{[\emptyset]}^0 $ and $\hat{a}_{[\emptyset]}^1$.
\medskip

Remember that  Proposition \ref{erer} of last section  claims that
the family of functions $\mathcal{F}$
 determine  an  orthonormal basis  for the Kernel  of the Ruelle operator.

\smallskip

We want to estimate for $X=\hat{a}_x,$ $Y= \hat{a}_z \in \mathcal{F}$ and the orthogonal basis $X_i= \hat{a}_y\in \mathcal{F}$   the explicit expression of the curvature which was described in Theorem \ref{main}


 \begin{equation} \label{tororo1}K(X,Y)  = \frac{1}{4}[\, \sum_{i=1}^\infty (   \int X \,Y\, X_i \,d \mu)^2 -   \sum_{i=1}^\infty   \int X^2 X_i \,d \mu\,  \,\int Y^2 X_i \,d \mu \,] .\end{equation}

\bigskip

 We will not present the explicit   expression of the sectional curvature $K(X,Y)$  for any pair of vectors $X,Y$ in the kernel, but just for the case where the functions $X,Y$ are part of the family
 $\hat{a}_x\in \mathcal{F}$.
\smallskip

An important issue is: $0=\langle  \hat{a}_z^2, \hat{a}_y\rangle = \int \hat{a}_z^2 \hat{a}_y \,d \mu$,   when the  length of $y$ is strictly  larger  than the  length of $z$ (as will be proved  in Sections  \ref{xxy} and \ref{xyz}). We mention this point to stress the point that the last sum in expression \eqref{tororo172} is a  sum of a finite number  of terms.

\smallskip

Our main result in this section concerns the Markov case:

\begin{theorem} \label{memo}
For a  fixed pair $\hat{a}_x, \hat{a}_z\in \mathcal{F}$ (with $z$ different from $x$)  the value
 \begin{equation} \label{tororo172} K(\hat{a}_z,\hat{a}_x)  = \frac{1}{4}[\, \sum_{\text{word} \, y}(   \int \hat{a}_x \,
 \hat{a}_z\, \hat{a}_y \,d \mu)^2 -   \sum_{\text{word} \, y}   \int \hat{a}_z^2 \hat{a}_y \,d \mu\,  \,\int \hat{a}_x^2 \hat{a}_y \,d \mu \,].\end{equation}
In the case   the length of $x$ is strictly  larger  than the  length of $z$ we get  that \eqref{tororo172}  can be expressed in a more simplified form as:
 \begin{equation} \label{tororo182} \frac{1}{4}   [  (\int \hat{a}_x^2 \, \hat{a}_z \,d \mu)^2  -   \sum_{y\,\in\,D[x,z]}   \int \hat{a}_x^2 \hat{a}_y \,d \mu\,  \,\int \hat{a}_z^2 \hat{a}_y \,d \mu].  \end{equation}

In this case the above  expression is a  sum of a finite number  of terms.

In the general case the value $\int \hat{a}_x^2 \hat{a}_z d \mu$ is  zero  if $z$ is not a subprefix of $x$.  If $z$ is a strict subprefix of $x$ and $y$ is a strict subprefix of $z$, then  the term
 \begin{equation} \label{ae11}-    \int \hat{a}_x^2 \hat{a}_y \,d \mu\,  \,\int \hat{a}_z^2 \hat{a}_y \,d \mu
 \end{equation}
 is non positive. Moreover, by \eqref{KLGT}  we get $\int \hat{a}_z^2 \, \hat{a}_x \,d \mu=0$.  Then, it
 follows that  \eqref{tororo172}  is a  sum of a finite number  of terms,
 for any given $x$ and $z$, with $z \neq x$.
\end{theorem}

The proof of this result will take several  sections and subsections.
Proposition \ref{pop24} will will summarize several  explicit computations that are necessary in our reasoning.

We will also provide an explicit expression for the curvature (\ref{tororo182}) in terms of the words $x,z$ and the probability $\mu$ (which is indexed by
$(r,s)$ of expression \eqref{tororo2}). This will follow from explicit expressions for   $(\int a_x \,a_z\, a_y \,d \mu)^2 $,  $ \int a_z^2 a_y \,d \mu\,$ and  $\,\int a_x^2 a_y \,d \mu$, for all  finite words $x,z,y$, that will be presented the Propositions \ref{pop23} and \ref{pop24}  (which will be proved in sections \ref{xxy} and \ref{xyz}).

It will also follow that when $x$ and $z$ do not share a common subprefix $y$, then the curvature $K(\hat{a}_z,\hat{a}_x)$  is equal to $0$ (see Proposition \ref{bor}).

There are examples (for instance, the case  $ x=(0,1,0)$ and $z=(0,1,0,0)$) where the curvature $K(\hat{a}_z,\hat{a}_x) $ is positive for some values of the parameters $(r,s)$ and negative for others (see Example \ref{bor1}). We can show from the explicit expressions we obtain that for fixed values of the parameters $(r,s)$ the curvature
$K(\hat{a}_z,\hat{a}_x)$ can be very negative if both words $x,z$ have large lengths and share common subprefix with large length (see Remark \ref{krat}).
In Example \ref{bor2} we  show that $K( \hat{a}_{(0)}, \hat{a}_{(0,0)} )= -0.205714...$, when $r=0.1,s=0.3$.  In Proposition \ref{zeze} we show the curvature $K(\hat{a}_{[\emptyset]}^0 ,\hat{a}_0)$  can be  positive for  some pairs $r,s\in(0,1)$. It follows from the expressions of Proposition \ref{pop24} that all sectional curvatures   $K( \hat{a}_{z}, \hat{a}_{x} )$ are equal to $-1/2$,  when $r=1/2=s$, the size of $z$ is bigger than $1$ and $z$ is a strict subprefix of $x$.
See also Proposition \ref{zeze}, when $r=1/2=s$, for the computation of
$ K(\hat{a}_{[\emptyset]}^0,\hat{a}_0)=1/2$.





\begin{remark} \label{retr}
Expression (\ref{tororo109}) in Subsection \ref{ylong1}  shows that in the case the
length of $x$ is larger than the length of $z$, then  $(\int \hat{a}_z^2 \, \hat{a}_x \,d \mu)^2=0.$
\end{remark}

\begin{proposition} \label{pop23} Assume that  the
length of $x$ is larger than the length of $z$. The first sum on expression (\ref{tororo172}) is given by
  \begin{equation} \label{gagatr11} \sum_{\text{word} \, y\,}(   \int \hat{a}_x \,\hat{a}_z\, \hat{a}_y \,d \mu)^2 = (\int \hat{a}_x^2 \, \hat{a}_z \,d \mu)^2 + (\int \hat{a}_z^2 \, \hat{a}_x \,d \mu)^2=  (\int \hat{a}_x^2 \, \hat{a}_z \,d \mu)^2.
 \end{equation}

\end{proposition}

For a proof of this claim see  expression (\ref{gagatr}) in section \ref{xyz}.  This term in the sum  (\ref{tororo172}) is the part that contributes to the curvature to be more positive. The second term in the  sum   (\ref{tororo172}) will contribute to the curvature becoming more negative  (see proposition \ref{pop24}).

Note that  \eqref{gagatr11} does not depend on $y$. Note also that from expression (\ref{ppp}) one can get explicitly  the values \eqref{gagatr11}  as a function of $(r,s)$.

In Proposition  \ref{louc1} we show that if $x$ is not a subprefix of $z$ and  $z$ is not a subprefix of $x$, we get that
$\int \hat{a}_x^2 \,\hat{a}_z \, d \mu=0$. In this case the contribution of (\ref{gagatr11}) for the curvature will be null.

\smallskip

\begin{proposition} \label{bor}
 When $z$ and  $x$ do not share common  subprefix
the curvature
$$K(\hat{a}_{z}, \hat{a}_{x})=0.$$

\end{proposition}
\begin{proof}

 When $z$ and  $x$ do not share a common  subprefix, it follows that $x$ is not a subprefix of $z$ and $z$ is not a subprefix of $x$.

We will show  that in this case $K(\hat{a}_z,\hat{a}_x)=0$.
Indeed, from Proposition \ref{louc1}  we get that  $(\int \hat{a}_x^2 \, \hat{a}_z \,d \mu)^2 + (\int \hat{a}_z^2 \, \hat{a}_x \,d \mu)^2=0.$
Fix the words $z$, $x$ and  consider  a variable word $y$. In order to estimate the second sum in expression  (\ref{tororo182})  we have to consider all different possible words $y$ such that are subprefix of $x$ and $z$. But there is no such kind of $y$.

Therefore, $K(\hat{a}_z,\hat{a}_x)=0$.

See also Proposition \ref{zeze}, when $r=1/2=s$, for the computation of
other sectional curvatures.
\end{proof}

\smallskip

\begin{remark} \label{rte}
It follows from Remark \ref{retr} that
 $ \sum_{\text{word} \, y}   \int \hat{a}_z^2 \hat{a}_y \,d \mu\,  \,\int \hat{a}_x^2 \hat{a}_y \,d \mu \,$
 is a sum of a finite number of terms; because when estimating $\int \hat{a}_z^2 \hat{a}_y \,d \mu\,  \int \hat{a}_x^2 \hat{a}_y \,d \mu$ we do not have to take into account words $y$ with length strictly larger than the minimum of the  lengths
 of $x$ and $z$.  It also follows from Proposition \ref{bor}   that if $x$ is not a subprefix of $y$ and  $y$ is not a subprefix of $x$, we get that $ \int a_x^2 \,a_y \, d \mu=0.$

Note that the above makes clear that in expression \eqref{tororo172}, the second sum has nonzero terms only when $y \in D[x,z]$. This justifies the simplified expression \eqref{tororo182}.

 \smallskip

 With all this in mind, in order to have explicit expressions, the next proposition deals just with the words $y$ with lengths smaller than or equal to the length of a given word  $x$.
\end{remark}

\begin{proposition}   \label{pop24} Assume that the length of $x$ is larger or equal  to the length of $y$. Then we have:
\bigskip

 a) $\int \hat{a}_x^2 \, \hat{a}_y d \mu=0$, if $y$ is not a subprefix of $x$. This also includes the case where $x\neq y$ and  length of $x$ is  equal  to the length of $y$.
\bigskip

b.0) Assume  that $[x]=[x_1,x_2,...x_k,x_{k+1},..., x_n] \subset [y]=[x_1,x_2,...,x_k]$, where $n > k$, and $x_{k+1}=0$.  Note that from \eqref{estanao} we get that $|a_x| =|a_y|$.  Then,

$$ \int \hat{a}_x^2 \hat{a}_y d \mu  =$$
$$  \frac{1}{|a_y|^3} \frac{\sqrt{P_{x_k,1}}}{ \sqrt{P_{x_k,0}} }\{ ( \frac{\pi_{x_1}  } {\pi_{0} P_{0,x_1} })^{3/2}   \frac{1}{\sqrt{\mu([0 y])}} -\, ( \frac{\pi_{x_1}  } {\pi_{1} P_{1,x_1} } )^{3/2}   \frac{1}{\sqrt{\mu([1 y])}}  \} =$$
 \begin{equation} \label{tororo9} \frac{1}{|a_x|^3} \sqrt{P_{x_k,1}}\{ ( \frac{\pi_{x_1}  } {\pi_{0} P_{0,x_1} })^{3/2}  \frac{1}{\sqrt{\mu([0 y 0])}} -\, ( \frac{\pi_{x_1}  } {\pi_{1} P_{1,x_1} } )^{3/2}  \frac{1}{\sqrt{\mu([1 y 0])}}  \}  .
\end{equation}

\smallskip

b.1) Assume,  that $[x]=[x_1,x_2,...x_k,x_{k+1},..., x_n] \subset [y]=[x_1,x_2,...,x_k]$, where $n > k$, and $x_{k+1}=1$. Then,
\bigskip

$$ \int \hat{a}_x^2 \hat{a}_y d \mu  = $$
$$\frac{1}{|a_y|^3} \,\frac{\sqrt{P_{x_k,0}}}{ \sqrt{P_{x_k,1}} }\{  -  ( \frac{\pi_{x_1}  } {\pi_{0} P_{0,x_1} })^{3/2}
\frac{1}{\sqrt{\mu([0 y])}}   +( \frac{\pi_{x_1}  } {\pi_{1} P_{1,x_1} } )^{3/2}   \frac{1}{\sqrt{\mu([1 y ])}}\,\}=  $$
\begin{equation} \label{tororo91}\frac{1}{|a_x|^3} \,\sqrt{P_{x_k,0}}\{  -  ( \frac{\pi_{x_1}  } {\pi_{0} P_{0,x_1} })^{3/2}
\frac{1}{\sqrt{\mu([0 y 1])}}   +( \frac{\pi_{x_1}  } {\pi_{1} P_{1,x_1} } )^{3/2}   \frac{1}{\sqrt{\mu([1 y 1])}}
    \}\end{equation}


b.2)  Assume, $[x]=[x_1,x_2,...x_n] = [y].$

Then,

$$ \int \hat{a}_x^2 \hat{a}_y d \mu= \int \hat{a}_y^3 d \mu=$$

 \begin{equation} \label{tororo1491}    \frac{1}{|a_x|^3} \{( \frac{\pi_{x_1}  } {\pi_{0} P_{0,x_1} })^{3/2} [  \frac{P_{x_n,1}^{3/2}}{\sqrt{\mu([0 y 0])}} \,  -   \frac{P_{x_n,0}^{3/2}}{\sqrt{\mu([0 y 1])}}  ]  -(\frac{\pi_{x_1}  } {\pi_{1} P_{1,x_1} } )^{3/2}  [    \frac{P_{x_n,1}^{3/2}}{\sqrt{\mu([1 y 0])}}    -
 \frac{P_{x_n,0}^{3/2}}{\sqrt{\mu([1 y 1])}}  ]\} .
\end{equation}

b3)  If $x_1=0$, then

\begin{equation} \label{rret} \int \hat{a}_x^2\hat{a}_{[\emptyset]}^0 d\mu=\, \frac{1}{|a_x|^2\, | V_1|  }( \frac{ \pi_1\, P_{1,0} } {P_{0,0} }  \,\, \,
\,+    \, \frac{\pi_0^2\, P_{0,0}  } {\pi_1 P_{1,0} })=\,\,
\end{equation}
$$ \frac{ (s-1) (1- 2 r + 2 r^2) }{\sqrt{ \frac{(1-r) r (s-1)^3}{(-2 + r + s)^3}  } \,(-2 + r +s)}>0,$$
and
$$\int \hat{a}_x^2\hat{a}_{[\emptyset]}^1 d\mu=0.$$

When $r=1/2=s$ we get that for any word $x$ (with size bigger or equal to $1$), such that, $x_1=0$
\begin{equation} \label{rret1}
\int \hat{a}_x^2\hat{a}_{[\emptyset]}^0 d\mu = \sqrt{2}.
\end{equation}

\smallskip

For the  proof of this proposition  see sections \ref{ylong2} and \ref{saco1}.
\end{proposition}

 \smallskip

 \begin{remark} \label{krat1}
\noindent

\begin{itemize}
\item We point out that \eqref{tororo9} and  \eqref{tororo91} do not depend on
$x_{k+2},...,x_{n-1},x_n$.

\item  If $y\in D[x,z]-\{ z\}$, then, the product  $\int \hat{a}_x^2 \hat{a}_y d \mu  \int \hat{a}_z^2 \hat{a}_y d \mu$ is non negative for any choice of $(r,s)$
(the product will not depend on $x$ and $z$). This follows from the expressions in b.0) and b.1). This shows \eqref{ae11}.

\item The term $\int \hat{a}_x^2 \hat{a}_z d \mu  \, \int \hat{a}_z^2 \hat{a}_z d \mu$ may be sometimes negative.

\end{itemize}

 \end{remark}

 \medskip

We previously denoted by $\mathcal{K}$  the two dimensional manifold of  Markov invariant probabilities (the set of equilibrium probabilities  for potentials depending on two coordinates and parametrized by $r,s$)

 Given the Markov invariant probability $\mu$ associated to the parameters $r,s$, the set of  vectors which are tangent to $\mathcal{K}$ at this point is the set of functions that depend on two coordinates $(x_1,x_2)$. The ones that are
 on the kernel of $\op{L}_{\log J}$ are  $\hat{a}_{[\emptyset]}^0 $
 and  $\hat{a}_{[\emptyset]}^1.$

 \medskip

 \begin{theorem} \label{meme}

Given the two-dimensional manifold of  Markov invariant probabilities   $M,$   for any point in $\mathcal{K}$ the sectional curvature for the pair of tangent vectors to $\mathcal{K}$  is always zero.

 \end{theorem}

 \begin{proof}

 Remember that $ V_1 = \pi_1 P_{1,0} \mathfrak{1}_{[00]} \,\, -  \,\,   \pi_0 P_{0,0} \mathfrak{1}_{[10]}$ and
$
V_2 = \pi_0 P_{0,1} \mathfrak{1}_{[11]} \,\, -  \,\,   \pi_1 P_{1,1} \mathfrak{1}_{[01]}$
determine an orthogonal basis for the tangent space to $\mathcal{K}$ at $\mu$.

 We claim that the curvature $K(\hat{a}_\emptyset^0,\hat{a}_\emptyset^1)=  0$.

Indeed, take $X_i= \hat{a}_z$, for some finite word $z=(x_1,x_2,...,x_k)$. If we assume that $x_1=1$, then,

\begin{eqnarray*}
\, V_1^2\, a_z & =& [ \, \pi_1^2 P_{1,0}^2\, \, (\frac{\pi_{x_1}  } {\pi_{0} P_{0,x_1} } \frac{1}{\mu([0 z])} \frac{P_{x_k,1}}{P_{x_k,0}\,})^{1/2} \mathfrak{1}_{[0z0]}\, \mathfrak{1}_{[00]} \\
& - &   \pi_1^2 P_{1,0}^2\, \, \,(\frac{\pi_{x_1}  } {\pi_{0} P_{0,x_1} } \frac{1}{\mu([0 z])} \frac{P_{x_k,0}}{P_{x_k,1}\,})^{1/2}  \mathfrak{1}_{[0z1]}\,\mathfrak{1}_{[00]}] \\
& - & [\pi_0^2 P_{0,0}^2\, \, \, \,(\frac{\pi_{x_1}} {\pi_{1} P_{1,x_1} } \frac{1}{\mu([1 z])} \frac{P_{x_k,1}}{P_{x_k,0}\,})^{1/2} \, \mathfrak{1}_{[1z0]} \mathfrak{1}_{[10]}\\
& -&  \pi_0^2 P_{0,0}^2\, \,  \,(\frac{\pi_{x_1}  } {\pi_{1} P_{1,x_1} } \frac{1}{\mu([1 z])} \frac{P_{x_k,0}}{P_{x_k,1}\,} )^{1/2} \, \mathfrak{1}_{[1z1]}\mathfrak{1}_{[00]}\,]\,=0.
\end{eqnarray*}

Above we use the fact that $\mathfrak{1}_{[0 1 x_2...x_k 0]}\, \mathfrak{1}_{[00]}=0$, etc.

Therefore, it follows that:
$$\int\, \hat{V}_1^2\, \hat{a}_y\, d \mu  \,\int \hat{V}_2^2\, \hat{a}_y\,\, d \mu =0.$$

 If we assume that $x_1=0$, then, in a similar way $\int \hat{a}_z\, V_2^2\, d \mu =0$, and therefore,
 $\int \hat{a}_z\, \hat{V}_1^2\, d \mu  \,\int \hat{a}_z\, \hat{V}_2^2\, d \mu =0.$

 Note that $V_1 V_2=0.$

Then,   for any word $y$ we get $\int \hat{V}_1\,\hat{V}_2\, \hat{a}_y \, d \mu=0.$ In the same way
$\int \hat{V}_1^2\,\hat{V}_2\,  \, d \mu=0$ and $\int \hat{V}_2^2\,\hat{V}_1\,  \, d \mu=0$.

Finally, we get,

  \begin{equation} \label{tororo18233} \frac{1}{4}   [  (\int (\hat{a}_\emptyset^0)^2\,\hat{a}_\emptyset^1 \,d \mu)^2  -   \sum_{y}   (\int (\hat{a}_\emptyset^0)^2 \hat{a}_y \,d \mu\,)  \,(\int (\hat{a}_\emptyset^1)^2 \hat{a}_y \,d \mu)] =0.  \end{equation}

 \end{proof}
\bigskip

\begin{remark} \label{erre} Recall that the expression of the Gauss sectional curvature $K_{M}(X,Y)$ of an isometric immersion $(M,g_{M})$, submanifold of the Riemannian manifold $(N,g)$, at the plane generated by two orthogonal vector fields $X, Y$ tangent to $\mathcal{K}$, is given by
$$ K_{M} (X,Y) = K(X,Y) + \langle \nabla^{\perp}_{X}X, \nabla^{\perp}_{Y}Y \rangle - \parallel \nabla^{\perp}_{X}Y \parallel^{2}$$ according to Gauss formula (see for instance \cite{doCa}). Here, the operator $\nabla^{\perp}_{X}Y$ is the component of
the covariant derivative $\nabla_{X}Y$ of the Riemannian manifold $(N,g)$ that is normal to $(M,g_{M})$.

Notice that the sectional curvature
$$  K(X,Y)=  \parallel \nabla_{\bar{Y}}\bar{X} \parallel^{2} - \langle \nabla_{\bar{X}}\bar{X}, \nabla_{\bar{Y}}\bar{Y} \rangle
$$
includes all the terms of the normal component of the covariant derivative of $X,Y$. By Theorem \ref{meme}, all the components of the covariant derivative of a certain pair of orthogonal vector fields tangent to the surface of Markov probabilities vanish. In particular, all the terms of the normal covariant derivative of $X,Y$ vanish. Therefore, Theorem \ref{meme} yields that the Gaussian curvature of the surface of Markov probabilities vanishes, its intrinsic curvature as an isometric immersion of the manifold of normalized potentials
is zero. This is a remarkable fact, which implies for instance that the surface would be totally geodesic in the manifold of normalized
potentials provided that geodesics exist. We won't consider the problem of the existence of geodesics in this article, we shall study this
problem in further papers.
\end{remark}

 \begin{proposition} \label{maxma}

 When $r=0.5$ and $s=0.5$, we get that
 \begin{equation} \label{lli}
 K(\hat{a}_y,\hat{a}_x) =-1/2,
 \end{equation}
 for words $x,y$ with size bigger or equal to $1$
 \end{proposition}
 \begin{proof}

 It follows from the above proposition that due to symmetry,  when $r=0.5$ and $s=0.5$, we get $\int \hat{a}_x^2 \hat{a}_y d \mu  \, \int \hat{a}_z^2 \hat{a}_y d \mu=0$, for words with size bigger or equal to $1$. Moreover, $\int \hat{a}_x^2 \,w d \mu=0$, for  any $a_x$. In this case, if $x_1=0$ and $x$ is a subprefix of $y$, we get that for words with size bigger or equal to $1$
  (see \eqref{rret1}),
\begin{equation}\label{enfii} K(\hat{a}_y,\hat{a}_x)= - 1/4 \,\int \hat{a}_x^2\hat{a}_{[\emptyset]}^0 d\mu\,\int \hat{a}_y^2\hat{a}_{[\emptyset]}^0 d\mu =-1/2.
\end{equation}
\end{proof}

\begin{remark} \label{krat} From the explicit expressions we obtain  (for fixed values of the parameters $(r,s)$) the curvature
$K(\hat{a}_z,\hat{a}_x)$ can be very negative if both words $x,z$ have large lengths and have common subprefix $y$ with large length. Indeed, for fixed $\hat{a}_z,\hat{a}_x$,    as  $\int \hat{a}_x^2 \hat{a}_y d \mu  \int \hat{a}_z^2 \hat{a}_y d \mu$ is non negative for any common word $y$, in the calculus of the curvature  $K(\hat{a}_z,\hat{a}_x)$, we get  a sum of several expressions  $\int \hat{a}_x^2 \hat{a}_y d \mu  \int \hat{a}_z^2 \hat{a}_y d \mu$. Note that $\int \hat{a}_x^2 \hat{a}_y d \mu  \int \hat{a}_z^2 \hat{a}_y d \mu$ does  depend on $y$ (but not on $x$ and $z$). Note also that for fixed $x$ the expression (\ref{tororo9}) can be very large if
the length of $y$ is very large (and, so $\mu([a y b])$, $a,b=0,1$, is very small).
\end{remark}
\smallskip





{\bf

\begin{proposition} \label{zeze} The curvature $K(\hat{a}_{[\emptyset]}^0,\hat{a}_0)=1/2$, when $r=1/2=s$.
\end{proposition}

\begin{proof} Note that
\eqref{tororo182} can be expressed as
$$ K(\hat{a}_{[\emptyset]}^0,\hat{a}_0)=\frac{1}{4}   [  \,(\int \hat{a}_0^2 \,\hat{a}_{[\emptyset]}^0  \,d \mu)^2  - \, \int \hat{a}_0^2 \,\hat{a}_{[\emptyset]}^0  \,d \mu\,\int (\hat{a}_{[\emptyset]}^0)^2\,\hat{a}_{[\emptyset]}^0  \,d \mu) ].$$

For any $r,s$, it is known from \eqref{rret} that
\begin{equation} \label{rret3} \int \hat{a}_0^2\hat{a}_{[\emptyset]}^0 d\mu=\, \frac{1}{|a_0|^2\, | V_1|  }( \frac{ \pi_1\, P_{1,0} } {P_{0,0} }  \,\, \,
\,+    \, \frac{\pi_0^2\, P_{0,0}  } {\pi_1 P_{1,0} })>0.
\end{equation}

Note that
$$
V_1^2\,V_1=$$
 $$
 (\pi_1^2 P_{1,0}^2 \mathfrak{1}_{[00]} \,\, +  \,\,   \pi_0^2 P_{0,0}^2
 \mathfrak{1}_{[10]})\times
 ( \pi_1 P_{1,0} \mathfrak{1}_{[00]} \,\, -  \,\,   \pi_0 P_{0,0} \mathfrak{1}_{[10]})=$$
$$ \pi_1^3 P_{1,0}^3 \mathfrak{1}_{[00]} \,\, -  \,\,   \pi_0^3 P_{0,0}^3
 \mathfrak{1}_{[10]}.$$

Then,
$$ \int V_1^2\,V_1 = \pi_1^3 P_{1,0}^3 \mu([00]) \,\, -  \,\,   \pi_0^3 P_{0,0}^3
 \mu([10]), $$
 which is equal to $ \frac{1}{2}^6 \frac{1}{2}^2 - \frac{1}{2}^6 \frac{1}{2}^2 =0$, in the case $r=1/2=s$.

 Therefore,
 $$ K(\hat{a}_{[\emptyset]}^0,\hat{a}_0)=\frac{1}{4}   (  \int \hat{a}_0^2 \,\hat{a}_{[\emptyset]}^0  \,d \mu)^2 = \frac{1}{4} \sqrt{2}^2=1/2>0.$$

\end{proof}
}

\bigskip

In other examples, we used the software Mathematica for getting explicit computations.


\begin{example} \label{bor1}
Consider the case where $ z=(0,1,0)$ and $x=(0,1,0,0).$

$\int \hat{a}_x^2 \hat{a}_y d \mu  \int \hat{a}_z^2 \hat{a}_y d \mu  =0$, unless $\hat{a}_y$ is such that
$$y\in D[(0,1,0),(0,1,0,0)]=\{(0),(0,1),(0,1,0)\},\,\,\text{or}\,\,\hat{a}_y=\hat{a}_{[\emptyset]}^0.$$
Note that    $\int \hat{a}_z^2\,\hat{a}_{[\emptyset]}^1 \,d\, \mu=\int \hat{a}_x^2\,\hat{a}_{[\emptyset]}^1 \,d\, \mu=0.$


\smallskip

Using Mathematica and the formulas of Proposition \ref{pop24}, we made computations when $r=0.1$ and $s=0.3$. In this case $\pi_0=0.4375$ and $\pi_1=0.5625$ and from \eqref{pede} we get
 $|a_{(0,1,0)}|=|a_{(0,1,0,0)}|=3.33.. $ and $|V_1|= 0.086...$
 Finally, $\frac{1}{|a_x|^2 |a_z|}= \frac{1}{|a_z|^3}=\frac{1}{|a_x|^2 |a_y|}=\frac{1}{|a_z|^2 |a_y|}=0.027... $

 We will show that $K( \hat{a}_{(0,1,0)}, \hat{a}_{(0,1,0,0)} )=35.9142....$

 We get the following values:
 $$\text{using\,} \eqref{tororo1491}\,\,\,\int \hat{a}_z^2 \hat{a}_z d \mu= \frac{1}{|a_z|^3 } \int a_{(0,1,0)}^2 a_{(0,1,0)} d \mu = 107,51...,$$

 $$\text{using\,} \eqref{tororo9}\,\,\,\int \hat{a}_x^2 \hat{a}_z d \mu = \frac{1}{|a_x|^2 |a_z|}   \int a_{(0,1,0,0)}^2 a_{(0,1,0)} d \mu  =120.949...,$$
 $$     (\int \hat{a}_x^2 \hat{a}_z d \mu)^2 =  (\int \hat{a}_{(0,1,0,0)}^2 \hat{a}_{(0,1,0)} d \mu)^2=  (16.93...)^2 = 14628.7...,$$
 $$ \text{using\,} \eqref{tororo9}\,\,\,\int \hat{a}_{(0,1,0,0)}^2 \hat{a}_{(0,1)} d \mu  = 38.2473...,$$
 $$\text{using\,} \eqref{tororo9}\,\,\,\int \hat{a}_{(0,1,0)} ^2 \hat{a}_{(0,1)} d \mu    =38.2473... $$
 $$ \text{using\,} \eqref{tororo91}\,\,\, \int \hat{a}_{(0,1,0,0)}^2 \hat{a}_{(0)} d \mu= -1.34387...$$


 $$\text{using\,} \eqref{tororo91}\,\,\, \int \hat{a}_{(0,1,0)} ^2 \hat{a}_{(0)} d \mu = -1.34387...,$$
 and finally, using \eqref{rret}
$$\int  \hat{a}_{(0,1,0)}^2  \hat{a}_{[\emptyset]}^0 d \mu=\int  \hat{a}_{(0,1,0,0)}^2  \hat{a}_{[\emptyset]}^0 d \mu=\frac{1}{|a_{(0,1,0)}|^2\, |V_1|}\int  a_{(0,1,0)}^2  V_1 d \mu= 4.13241...$$
Using \eqref{rret} and \eqref{V1}  (note that $x_1=0$),  we get that the expression  (\ref{tororo182}) can be written in this case as
\begin{eqnarray*}
K( \hat{a}_{(0,1,0)}, \hat{a}_{(0,1,0,0)} ) & = &\frac{1}{4}   (\int \hat{a}_{(0,1,0,0)}^2 \, \hat{a}_{(0,1,0)} \,d \mu)^2  \\
&  -  &  \frac{1}{4}[ \int \hat{a}_{(0,1,0,0)}^2 \hat{a}_{(0,1,0)} \,d \mu\,  \,\int \hat{a}_{(0,1,0)}^2 \hat{a}_{(0,1,0)} \,d \mu]\\
&  - & \frac{1}{4}[\int \hat{a}_{(0,1,0,0)}^2 \hat{a}_{(0,1)} \,d \mu\,  \,\int \hat{a}_{(0,1,0)}^2 \hat{a}_{(0,1)} \,d \mu ] \\
&  - &  \frac{1}{4}[\int \hat{a}_{(0,1,0,0)}^2 \hat{a}_{(0)} \,d \mu\,  \,\int \hat{a}_{(0,1,0)}^2 \hat{a}_{(0)} \,d \mu      ]\,\\
&-&\frac{1}{4}\,\int  \hat{a}_{(0,1,0,0)}^2  \hat{a}_{[\emptyset]}^0 d \mu \, \int \hat{a}_{(0,1,0)}^2  \hat{a}_{[\emptyset]}^0 d \mu = 35.9142...\\
\end{eqnarray*}

 Taking $r=0.8,s=0.5$  we get $K( \hat{a}_{(0,1,0)}, \hat{a}_{(0,1,0,0)} )=-3.17713...$
 When $r=1/2=s$ we get $K( \hat{a}_{(0,1,0)}, \hat{a}_{(0,1,0,0)} )=-1/2$.

$\,\,\,\,\,\,\,\,\,\,\,\,\,\,\,\,\,\,\,\,\,\,\,\,\,\,\,\,\,\,\,\,\,\,\,\,\,\,\,\,\,\,\,\,\,\,\,\,\,\,\,\,\,\,\,\,\,\,\,\,\,\,\,\,\,\,\,\,\,\,\,\,\,\,\,\,\,\,\,\,\,\,\,\,\,\,\,\,\,\,\,\,\,\,\,\,\,\,\,\,\,\,\,\,\,\,\,\,\,\,\,\,\,\,\,\,\,\,\,\,\,\,\,\,\,\,\,\,\,\,\,\,\,\,\,\,\,\,\,\,\,\,\,\,\,\,\,\,\,\,\,\,\,\,\,\,\,\,\,\,\,\,\,\,\,\,\,\,\,\,\,\,\,\,\,\,\,\,\,\,\,\,\,\,\,\,\,\,\,\,\,\,\,\,\,\,\,\,\,\,\,\,\,\,\diamondsuit$
\end{example}

\begin{example} \label{bor2}
Consider the case where $ z=(0)$ and $x=(0,0).$ Then,
 $D[(0),(0,0)] =\{\hat{a}_0, \hat{a}_{[\emptyset]}^0\}.$
Therefore,
\begin{eqnarray*}
K( \hat{a}_{(0)}, \hat{a}_{(0,0)} ) & = &\frac{1}{4}   (\int \hat{a}_{(0,0)}^2 \, \hat{a}_{(0)} \,d \mu)^2  \\
&  -  &  \frac{1}{4}[ \int \hat{a}_{(0,0)}^2 \hat{a}_{(0)} \,d \mu\,  \,\int \hat{a}_{(0)}^2 \hat{a}_{(0)} \,d \mu]\\
&-&\frac{1}{4}\,\int  \hat{a}_{(0,0)}^2  \hat{a}_{[\emptyset]}^0 d \mu \, \int \hat{a}_{(0)}^2  \hat{a}_{[\emptyset]}^0 d \mu.
\end{eqnarray*}

In this case, using Mathematica, one can show that $K( \hat{a}_{(0)}, \hat{a}_{(0,0)} )\leq 0$, for all values $r,s\in(0,1)$. For $r=0.1$, $s=0.3$, we will show that $K( \hat{a}_{(0)}, \hat{a}_{(0,0)} )= -0.205714...$

When, $r=0.1,s=0.3$, we get
$$ |a_0|=3.333...,
$$
$$|V_1|= 0.086...,$$
$$\int \hat{a}_{(0,0)}^2 \, \hat{a}_{(0)} \,d \mu=\frac{1}{ |a_{(0)}|^3}  $$
$$\int \hat{a}_{(0)}^2 \, \hat{a}_{(0)} \,d \mu= \frac{1}{|a_{(0)}|^3\, }$$
$$\int  \hat{a}_{(0,0)}^2  \hat{a}_{[\emptyset]}^0 d \mu  = \frac{1}{|a_{(0)}|^2\, |V_1|} \int  a_{(0)}^2  V_1 d \mu = \frac{1}{|a_{(0)}|^2\, |V_1|}\, 3.96.$$

Finally, when $r=0.1$, $s=0.3$ we get $K( \hat{a}_{(0)}, \hat{a}_{(0,0)} )= -0.205714...$

$\,\,\,\,\,\,\,\,\,\,\,\,\,\,\,\,\,\,\,\,\,\,\,\,\,\,\,\,\,\,\,\,\,\,\,\,\,\,\,\,\,\,\,\,\,\,\,\,\,\,\,\,\,\,\,\,\,\,\,\,\,\,\,\,\,\,\,\,\,\,\,\,\,\,\,\,\,\,\,\,\,\,\,\,\,\,\,\,\,\,\,\,\,\,\,\,\,\,\,\,\,\,\,\,\,\,\,\,\,\,\,\,\,\,\,\,\,\,\,\,\,\,\,\,\,\,\,\,\,\,\,\,\,\,\,\,\,\,\,\,\,\,\,\,\,\,\,\,\,\,\,\,\,\,\,\,\,\,\,\,\,\,\,\,\,\,\,\,\,\,\,\,\,\,\,\,\,\,\,\,\,\,\,\,\,\,\,\,\,\,\,\,\,\,\,\,\,\,\,\,\,\,\,\,\diamondsuit$
\end{example}

\section{Computations for the integral  $\int X^2 Y $} \label{xxy}

Our purpose in this section is to evaluate the integral
 \begin{equation} \label{gaga1} \sum_{\text{word} \, y\,}  \int \hat{a}_x^2 \, \hat{a}_y \,d \mu \,\,  \int \hat{a}_z^2 \, \hat{a}_y \,d \mu,
 \end{equation}
for any given pair of words $x,z$. This corresponds to the second term in the sum given by expression (\ref{tororo172}).

We assume that $x$ is different from $z$.

From proposition \ref{louc} if $x$ is not a subprefix of $y$ and $y$ is not a subprefix of $x$, and  $x\neq y$, then:
$$\hat{a}_x^2 \, \hat{a}_y=0.$$

In the same way, if   $z$ is not a subprefix of $y$ and $y$ is not a subprefix of $z$,  and  $z\neq y$, then:
$$\hat{a}_z^2 \, \hat{a}_y=0.$$

If $y$ has the same length as $x$ but $y\neq x$, then $\hat{a}_x^2 \, \hat{a}_y=0.$

In this way, for a fixed  pair of words $x,z$, several words $y$ do not contribute to the  sum \eqref{gaga}.

\subsection{The value of $\langle  \hat{a}_x^2, \hat{a}_y\rangle $ when  length of $x$ is larger or equal  than the  length of $y$} \label{ylong2}

 We want to compute $\langle  \hat{a}_x^2, \hat{a}_y\rangle =\int \hat{a}_x^2 \, \hat{a}_y d \mu$ in the case where  the  length of $x$ is larger or equal to the length of $y$.

 Our computation is in fact for $\langle  a_x^2, a_y\rangle $ and after  that, of course, to get $\langle  \hat{a}_x^2, \hat{a}_y\rangle $ it will be necessary to divide by  $|a_x|^2\, |a_y|.$

We assume that $[x]=[x_1,x_2,...x_k,x_{k+1},..., x_n] \subset [y]=[x_1,x_2,...,x_k]$, where $n \geq k$  (otherwise we get zero).

Note that these assumptions  include the integral $\int \hat{a}_x^3 d \mu$, that is, the case $x=y$ (see III) below).

I) Case $n>k$ -
We will assume first that $x_{k+1}=0$ in the word $[x]$.

Given the words $z=(v_1,...,v_t)$ and $v=(v_1,v_2,...,v_t, v_{t+1},...,v_m)$, assume $v_{t+1}=0$, then, from (\ref{eq52}) and  (\ref{eq524})
$$    e_{[v]}^2 e_{[z]}=[\frac{1}{\mu([v])} \frac{P_{v_m,1}}{P_{v_m,0}\,} \, \mathfrak{1}_{[v_1,...,v_t,0,v_{t+2},...,v_m, 0]} + \frac{1}{\mu([v])} \frac{P_{v_m,0}}{P_{v_m,1} } \, \mathfrak{1}_{[ v_1,...,v_t,0,v_{t+2},...,v_m ,1]}]  $$
$$\times [\frac{1}{\sqrt{\mu([z])}} \sqrt{\frac{P_{v_t,1}}{P_{v_t,0}\,}} \, \mathfrak{1}_{[v_1,...,v_t, 0]} - \frac{1}{\sqrt{\mu([z])}} \sqrt{\frac{P_{v_t,0}}{P_{v_t,1} }} \, \mathfrak{1}_{[v_1,...,v_t , 1]}  ]  $$
$$=(\frac{1}{\mu([v])} \frac{P_{v_m,1}}{P_{v_m,0}\,}) (  \frac{1}{\sqrt{\mu([z])}} \sqrt{\frac{P_{v_t,1}}{P_{v_t,0}\,}} )\, \mathfrak{1}_{[v_1,...,0, v_{t+2},...,v_m,0]} $$
\begin{align}
  +   ( \frac{1}{\mu([v])} \frac{P_{v_m,0}}{P_{v_m,1} }) \,(\frac{1}{\sqrt{\mu([z])}} \sqrt{\frac{P_{v_t,1}}{P_{v_t,0} }} )\, \mathfrak{1}_{[v_1,...,0,v_{t+2},...,v_m, 1]}.
    \label{eq524}
    \end{align}

Note that in the above reasoning when going from the second to the third line the term multiplying $\mathfrak{1}_{[v_1,...,v_t , 1]} $ disappear because we assume that $v_{t+1}=0.$

We are going to apply the above when $z=[0 y], z=[1 y], v=[0 x], v=[1 x],  m=n$ and $t+1=k$.

Then, from (\ref{luc}), (\ref{lucato}), (\ref{eq524}) and using the fact that
$$e_{[0,x_1, x_2,...,x_k,0,x_{k+2},...,x_n]}^2\,\, \,e_{[1,x_1, x_2,..,x_k]}=0,$$
$$e_{[1,x_1, x_2,...,x_k,0,x_{k+2},...,x_n]}^2\,\, \,e_{[0,x_1, x_2,..,x_k]}=0,$$
we get
$$ a_x^2 \,a_y=[\frac{\pi_{x_1} } {\pi_{0} P_{0,x_1} } \,\, e_{[0,x_1, x_2,..,x_k, 0,x_{k+2},...,x_n]}^2 + \,\frac{\pi_{x_1}  } {\pi_{1}P_{1,x_1} }\,\,e_{[1,x_1, x_2,...,x_k,0,x_{k+2},...,x_n]}^2] $$
$$   \times  [\frac{ \sqrt{\pi_{x_1}  }} {\sqrt{\pi_{0}}\sqrt{P_{0,x_1} }} \,\, e_{[0,x_1, x_2,..,x_k]} - \,\frac{ \sqrt{\pi_{x_1}  }} {\sqrt{\pi_{1}}\sqrt{P_{1,x_1} }}\,\,e_{[1,x_1, x_2,..,x_k]}]        $$
$$=  \, ( \frac{\pi_{x_1}  } {\pi_{0} P_{0,x_1} })^{3/2} [ (\frac{1}{\mu([0 x])} \frac{P_{x_n,1}}{P_{x_n,0}\,})   (\frac{1}{\sqrt{\mu([0 y])}} \sqrt{\frac{P_{x_k,1}}{P_{x_k,0}\,}}\, ) \mathfrak{1}_{[0x0]}   $$
$$+  (\frac{1}{\mu([0 x])} \frac{P_{x_n,0}}{P_{x_n,1} }) ( \frac{1}{\sqrt{\mu([0 y])}} \sqrt{\frac{P_{x_k,1}}{P_{x_k,0}\,}}) \, \mathfrak{1}_{[0x1]}\,]
$$
$$-( \frac{\pi_{x_1}  } {\pi_{1} P_{1,x_1} } )^{3/2}  [\,  (\frac{1}{\mu([1 x])} \frac{P_{x_n,1}}{P_{x_n,0}\,} )  ( \frac{1}{\sqrt{\mu([1 y])}} \sqrt{\frac{P_{x_k,1}}{P_{x_k,0}\,}) }\, \mathfrak{1}_{[1x0]} $$
$$ +  (\frac{1}{\mu([1 x])} \frac{P_{x_n,0}}{P_{x_n,1} })\,(\frac{1}{\sqrt{\mu([1 y])}} \sqrt{\frac{P_{x_k,1}}{P_{x_k,0} }}) \mathfrak{1}_{[1x1]}\,] .
$$

Finally, as the matrix $P$ is row stochastic
$$ \int a_x^2 a_y d \mu=   \, ( \frac{\pi_{x_1}  } {\pi_{0} P_{0,x_1} })^{3/2} [  (\frac{P_{x_n,1}}{\sqrt{\mu([0 y])}} \sqrt{\frac{P_{x_k,1}}{P_{x_k,0}\,}}\, ) +   (\frac{P_{x_n,0}}{\sqrt{\mu([0 y])}} \sqrt{\frac{P_{x_k,1}}{P_{x_k,0}\,}}) \, \,]
$$
 $$-( \frac{\pi_{x_1}  } {\pi_{1} P_{1,x_1} } )^{3/2}  [\,   ( \frac{P_{x_n,1}}{\sqrt{\mu([1 y])}} \sqrt{\frac{P_{x_k,1}}{P_{x_k,0}\,}}) \,  +
 \,(\frac{P_{x_n,0}}{\sqrt{\mu([1 y])}} \sqrt{\frac{P_{x_k,1}}{P_{x_k,0} }}) \,]
$$
$$=\, (P_{x_n,1} +P_{x_n,0}  )  \sqrt{P_{x_k,1}}\{\, ( \frac{\pi_{x_1}  } {\pi_{0} P_{0,x_1} })^{3/2} [  \frac{1}{\sqrt{\mu([0 y 0])}}\,  \, \,]
$$
$$ -( \frac{\pi_{x_1}  } {\pi_{1} P_{1,x_1} } )^{3/2}  [\,    \frac{1}{\sqrt{\mu([1 y 0])}}   \,] \,\}=
$$
 \begin{equation} \label{tororo999}\,   \sqrt{P_{x_k,1}}\{\, ( \frac{\pi_{x_1}  } {\pi_{0} P_{0,x_1} })^{3/2}   \frac{1}{\sqrt{\mu([0 y 0])}}\,    \, \, -( \frac{\pi_{x_1}  } {\pi_{1} P_{1,x_1} } )^{3/2}  \,   \frac{1}{\sqrt{\mu([1 y 0])}} \, \,\} .
\end{equation}


II) Case $n>k$  -
If we assume $x_{k+1}=1$ in the word $[x]$, then  we get in a similar way as before
$$ \int a_x^2 a_y d \mu= $$

\begin{equation} \label{tororo149}\sqrt{P_{x_k,0}}\{  -  ( \frac{\pi_{x_1}  } {\pi_{0} P_{0,x_1} })^{3/2}
\frac{1}{\sqrt{\mu([0 y 1])}}   +( \frac{\pi_{x_1}  } {\pi_{1} P_{1,x_1} } )^{3/2}   \frac{1}{\sqrt{\mu([1 y 1])}}
    \}   \end{equation}

Indeed, given the words $z=(v_1,...,v_t)$ and $v=(v_1,v_2,...,v_t, v_{t+1},...,v_m)$, assume $v_{t+1}=1$, then, from (\ref{eq52}) and  (\ref{eq524})
$$    e_{[v]}^2 e_{[z]}=[\frac{1}{\mu([v])} \frac{P_{v_m,1}}{P_{v_m,0}\,} \, \mathfrak{1}_{[v_1,...,v_t,1,v_{t+2},...,v_m, 0]} + \frac{1}{\mu([v])} \frac{P_{v_m,0}}{P_{v_m,1} } \, \mathfrak{1}_{[ v_1,...,v_t,1,v_{t+2},...,v_m ,1]}]  $$
$$\times [\frac{1}{\sqrt{\mu([z])}} \sqrt{\frac{P_{v_t,1}}{P_{v_t,0}\,}} \, \mathfrak{1}_{[v_1,...,v_t, 0]} - \frac{1}{\sqrt{\mu([z])}} \sqrt{\frac{P_{v_t,0}}{P_{v_t,1} }} \, \mathfrak{1}_{[v_1,...,v_t , 1]}  ]  $$
$$=\,-\,[(\frac{1}{\mu([v])} \frac{P_{v_m,1}}{P_{v_m,0}\,}) (  \frac{1}{\sqrt{\mu([z])}} \sqrt{\frac{P_{v_t,0}}{P_{v_t,1}\,}} )\, \mathfrak{1}_{[v_1,...,1, v_{t+2},...,v_m,0]} \,]$$
\begin{align}
  +   ( \frac{1}{\mu([v])} \frac{P_{v_m,0}}{P_{v_m,1} }) \,(\frac{1}{\sqrt{\mu([z])}} \sqrt{\frac{P_{v_t,0}}{P_{v_t,1} }} )\, \mathfrak{1}_{[v_1,...,1,v_{t+2},...,v_m, 1]}.
    \label{eq52467}
    \end{align}

 We are going to apply the above when $z=[0 y], z=[1 y], v=[0 x], v=[1 x],  m=n$ and $t=k+1$.

Then, from (\ref{luc}), (\ref{lucato}), (\ref{eq52467}) and using the fact that
$$e_{[1,x_1, x_2,...,x_k,1,x_{k+2},...,x_n]}^2\,\, \,e_{[0,x_1, x_2,..,x_k]}=0,$$
$$e_{[0,x_1, x_2,...,x_k,1,x_{k+2},...,x_n]}^2\,\, \,e_{[1,x_1, x_2,..,x_k]}=0,$$
we get
$$ a_x^2 \,a_y=[\frac{\pi_{x_1} } {\pi_{0} P_{0,x_1} } \,\, e_{[0,x_1, x_2,..,x_k, 1,x_{k+2},...,x_n]}^2 + \,\frac{\pi_{x_1}  } {\pi_{1}P_{1,x_1} }\,\,e_{[1,x_1, x_2,...,x_k,1,x_{k+2},...,x_n]}^2] $$
$$   \times  [\frac{ \sqrt{\pi_{x_1}  }} {\sqrt{\pi_{0}}\sqrt{P_{0,x_1} }} \,\, e_{[0,x_1, x_2,..,x_k]} - \,\frac{ \sqrt{\pi_{x_1}  }} {\sqrt{\pi_{1}}\sqrt{P_{1,x_1} }}\,\,e_{[1,x_1, x_2,..,x_k]}]        $$
$$= \,- \, ( \frac{\pi_{x_1}  } {\pi_{0} P_{0,x_1} })^{3/2} [\,\, (\frac{1}{\mu([0 x])} \frac{P_{x_n,1}}{P_{x_n,0}\,})   (\frac{1}{\sqrt{\mu([0 y])}} \sqrt{\frac{P_{x_k,0}}{P_{x_k,1}\,}}\, ) \mathfrak{1}_{[0x0]}   $$
$$+  (\frac{1}{\mu([0 x])} \frac{P_{x_n,0}}{P_{x_n,1} }) ( \frac{1}{\sqrt{\mu([0 y])}} \sqrt{\frac{P_{x_k,0}}{P_{x_k,1}\,}}) \, \mathfrak{1}_{[0x1]}\,\,]
$$
$$+( \frac{\pi_{x_1}  } {\pi_{1} P_{1,x_1} } )^{3/2}  [\, \, (\frac{1}{\mu([1 x])} \frac{P_{x_n,1}}{P_{x_n,0}\,} )  ( \frac{1}{\sqrt{\mu([1 y])}} \sqrt{\frac{P_{x_k,0}}{P_{x_k,1}\,}) }\, \mathfrak{1}_{[1x0]} $$
$$+  (\frac{1}{\mu([1 x])} \frac{P_{x_n,0}}{P_{x_n,1} })\,(\frac{1}{\sqrt{\mu([1 y])}} \sqrt{\frac{P_{x_k,0}}{P_{x_k,1} }}) \mathfrak{1}_{[1x1]}\,] .
$$

Finally, as the matrix $P$ is row  stochastic
$$ \int a_x^2 a_y d \mu=  - \, ( \frac{\pi_{x_1}  } {\pi_{0} P_{0,x_1} })^{3/2} [  (\frac{P_{x_n,1}}{\sqrt{\mu([0 y])}} \sqrt{\frac{P_{x_k,0}}{P_{x_k,1}\,}}\, ) + (\frac{P_{x_n,0}}{\sqrt{\mu([0 y])}} \sqrt{\frac{P_{x_k,0}}{P_{x_k,1}\,}}) \, \,]
$$
 $$+( \frac{\pi_{x_1}  } {\pi_{1} P_{1,x_1} } )^{3/2}  [\,   ( \frac{P_{x_n,1}}{\sqrt{\mu([1 y])}} \sqrt{\frac{P_{x_k,0}}{P_{x_k,1}\,}}) \,  +
 \,(\frac{P_{x_n,0}}{\sqrt{\mu([1 y])}} \sqrt{\frac{P_{x_k,0}}{P_{x_k,1} }}) \,]=
$$
$$\sqrt{P_{x_k,0}}\{  -  ( \frac{\pi_{x_1}  } {\pi_{0} P_{0,x_1} })^{3/2} [
\frac{P_{x_n,1}+ P_{x_n,0}}{\sqrt{\mu([0 y 1])}}   ]+( \frac{\pi_{x_1}  } {\pi_{1} P_{1,x_1} } )^{3/2}  [   \frac{P_{x_n,1}+ P_{x_n,0}}{\sqrt{\mu([1 y 1])}}
  ]  \}=
$$
\begin{equation} \label{tororo99900}\sqrt{P_{x_k,0}}\{  -  ( \frac{\pi_{x_1}  } {\pi_{0} P_{0,x_1} })^{3/2}
\frac{1}{\sqrt{\mu([0 y 1])}}   +( \frac{\pi_{x_1}  } {\pi_{1} P_{1,x_1} } )^{3/2}   \frac{1}{\sqrt{\mu([1 y 1])}}
    \}\end{equation}

 III) Case $n=k$ - We  assume $[x]=[x_1,x_2,...x_n] = [y]$, otherwise $ \int \hat{a}_x^2 \hat{a}_y d \mu=0. $
Then, one can show that
\begin{equation} \label{tororo144}  \int a_x^2 a_y d \mu= \int a_x^3 d \mu=
\end{equation}
$$ ( \frac{\pi_{x_1}  } {\pi_{0} P_{0,x_1} })^{3/2} [  \frac{P_{x_n,1}^{3/2}}{\sqrt{\mu([0 x 0])}} \,  -   \frac{P_{x_n,0}^{3/2}}{\sqrt{\mu([0 x 1])}}  \, \,]\,-( \frac{\pi_{x_1}  } {\pi_{1} P_{1,x_1} } )^{3/2}  [\,    \frac{P_{x_n,1}^{3/2}}{\sqrt{\mu([1 x 0])}}  \,  -
 \,\frac{P_{x_n,0}^{3/2}}{\sqrt{\mu([1 x 1])}}  \,] .$$

Indeed, note first that from \eqref{eq524}, $v=[v_1,v_2,...,v_m]$
$$    e_{[v]}^2 e_{[v]}=[\frac{1}{\mu([v])} \frac{P_{v_m,1}}{P_{v_m,0}\,} \, \mathfrak{1}_{[v_1,...,v_m, 0]} + \frac{1}{\mu([v])} \frac{P_{v_m,0}}{P_{v_m,1} } \, \mathfrak{1}_{[ v_1,...,v_m ,1]}]  $$
$$\times [\frac{1}{\sqrt{\mu([v])}} \sqrt{\frac{P_{v_m,1}}{P_{v_m,0}\,}} \, \mathfrak{1}_{[v_1,...,v_m, 0]} - \frac{1}{\sqrt{\mu([v])}} \sqrt{\frac{P_{v_m,0}}{P_{v_m,1} }} \, \mathfrak{1}_{[v_1,...,v_m , 1]}  ]  $$
$$=(\frac{1}{\mu([v])} \frac{P_{v_m,1}}{P_{v_m,0}\,}) (  \frac{1}{\sqrt{\mu([v])}} \sqrt{\frac{P_{v_m,1}}{P_{v_m,0}\,}} )\, \mathfrak{1}_{[v_1,...,v_m,0]} $$
\begin{align}
  -   ( \frac{1}{\mu([v])} \frac{P_{v_m,0}}{P_{v_m,1} }) \,(\frac{1}{\sqrt{\mu([v])}} \sqrt{\frac{P_{v_m,0}}{P_{v_m,1} }} )\, \mathfrak{1}_{[v_1,...,v_m, 1]}.
    \label{eq5246799}
    \end{align}

Then, from (\ref{luc}), (\ref{lucato}), (\ref{eq52467})
$$ a_x^2 a_x=[\frac{\pi_{x_1} } {\pi_{0} P_{0,x_1} } \,\, e_{[0,x_1, x_2,...,x_n]}^2 + \,\frac{\pi_{x_1}  } {\pi_{1}P_{1,x_1} }\,\,e_{[1,x_1, x_2,...,x_n]}^2] $$
$$   \times  [\frac{ \sqrt{\pi_{x_1}  }} {\sqrt{\pi_{0}}\sqrt{P_{0,x_1} }} \,\, e_{[0,x_1, x_2,..,x_n]} - \,\frac{ \sqrt{\pi_{x_1}  }} {\sqrt{\pi_{1}}\sqrt{P_{1,x_1} }}\,\,e_{[1,x_1, x_2,..,x_n]}]        $$
$$=  \, ( \frac{\pi_{x_1}  } {\pi_{0} P_{0,x_1} })^{3/2} [ (\frac{1}{\mu([0 x])} \frac{P_{x_n,1}}{P_{x_n,0}\,})   (\frac{1}{\sqrt{\mu([0 x])}} \sqrt{\frac{P_{x_n,1}}{P_{x_n,0}\,}}\, ) \mathfrak{1}_{[0x0]}   $$
$$-  (\frac{1}{\mu([0 x])} \frac{P_{x_n,0}}{P_{x_n,1} }) ( \frac{1}{\sqrt{\mu([0 x])}} \sqrt{\frac{P_{x_n,0}}{P_{x_n,1}\,}}) \, \mathfrak{1}_{[0x1]}\,]
$$
$$-( \frac{\pi_{x_1}  } {\pi_{1} P_{1,x_1} } )^{3/2}  [\,  (\frac{1}{\mu([1 x])} \frac{P_{x_n,1}}{P_{x_n,0}\,} )  ( \frac{1}{\sqrt{\mu([1 x])}} \sqrt{\frac{P_{x_n,1}}{P_{x_n,0}\,}) }\, \mathfrak{1}_{[1x0]} $$
$$ -  (\frac{1}{\mu([1 x])} \frac{P_{x_n,0}}{P_{x_n,1} })\,(\frac{1}{\sqrt{\mu([1 x])}} \sqrt{\frac{P_{x_n,0}}{P_{x_n,1} }}) \mathfrak{1}_{[1x1]}\,] .
$$

Therefore,
$$ \int a_x^2 a_x d \mu=   \, ( \frac{\pi_{x_1}  } {\pi_{0} P_{0,x_1} })^{3/2} [  (\frac{P_{x_n,1}}{\sqrt{\mu([0 x])}} \sqrt{\frac{P_{x_n,1}}{P_{x_n,0}\,}}\, ) -   (\frac{P_{x_n,0}}{\sqrt{\mu([0 x])}} \sqrt{\frac{P_{x_n,0}}{P_{x_n,1}\,}}) \, \,]
$$
 $$-( \frac{\pi_{x_1}  } {\pi_{1} P_{1,x_1} } )^{3/2}  [\,   ( \frac{P_{x_n,1}}{\sqrt{\mu([1 x])}} \sqrt{\frac{P_{x_n,1}}{P_{x_n,0}\,}}) \,  -
 \,(\frac{P_{x_n,0}}{\sqrt{\mu([1 x)}} \sqrt{\frac{P_{x_n,0}}{P_{x_n,1} }}) \,]=
$$
$$  \, ( \frac{\pi_{x_1}  } {\pi_{0} P_{0,x_1} })^{3/2} [  \frac{P_{x_n,1}^{3/2}}{\sqrt{\mu([0 x 0])}} \,  -   \frac{P_{x_n,0}^{3/2}}{\sqrt{\mu([0 x 1])}}  \, \,]\,-( \frac{\pi_{x_1}  } {\pi_{1} P_{1,x_1} } )^{3/2}  [\,   ( \frac{P_{x_n,1}^{3/2}}{\sqrt{\mu([1 x 0])}} ) \,  -
 \,(\frac{P_{x_n,0}^{3/2}}{\sqrt{\mu([1 x 1])}}  \,] =$$

 \begin{equation} \label{tororo9990700}   \, ( \frac{\pi_{x_1}  } {\pi_{0} P_{0,x_1} })^{3/2} [  \frac{P_{x_n,1}^{3/2}}{\sqrt{\mu([0 x 0])}} \,  -   \frac{P_{x_n,0}^{3/2}}{\sqrt{\mu([0 x 1])}}  \, \,]\,-( \frac{\pi_{x_1}  } {\pi_{1} P_{1,x_1} } )^{3/2}  [\,    \frac{P_{x_n,1}^{3/2}}{\sqrt{\mu([1 x 0])}}  \,  -
 \,\frac{P_{x_n,0}^{3/2}}{\sqrt{\mu([1 x 1])}}  \,] .
\end{equation}

\medskip

The above reasoning shows III).

\medskip

Given the word
 $[x]=[x_1,x_2,...x_k,x_{k+1},..., x_n] $ we get $n$ words $y$, such that,  the cylinder $[x]\subset [y]=[x_1,x_2,...,x_k]$, where $n\geq k$.

Given $x$ and $z$, with length larger than $y$, then   $\int \hat{a}_{x}^2 \, \hat{a}_y d \mu \, \int \hat{a}_{z}^2 \, \hat{a}_y d \mu$  will be nonzero only for the subprefixes
$y$ which are common to both  $x$ and $z$ (see Proposition \ref{louc}). If there are no common subprefixes for  $x$ and $z$, then   the contribution
$\int \hat{a}_{x}^2 \, \hat{a}_y d \mu \, \int \hat{a}_{z}^2 \, \hat{a}_y d \mu$, for words $y$ of length strictly smaller than the length    of   $x$ and $z$,  in the  sum \eqref{tororo182} is null.

\subsection{The values of $\langle  \hat{a}_x^2, \hat{a}_{[\emptyset]}^0\rangle $ and $\langle  \hat{a}_x^2, \hat{a}_{[\emptyset]}^1\rangle $ when $x$ is a finite word}\label{saco1}

Denote $[x]=[x_1,x_2,..., x_n]$.  We assume that $n\geq 2.$

 In fact, we will compute  $\langle  a_x^2, V_1\rangle $ and   $\langle  a_x^2, V_2\rangle $.
 In order to compute $\langle  \hat{a}_x^2, \hat{a}_{[\emptyset]}^0\rangle $ and $\langle  \hat{a}_x^2, \hat{a}_{[\emptyset]}^1\rangle $ it will be necessary to normalize.

I) Case  $\langle  a_x^2, V_1\rangle $

We will consider first the case $x_1=0$.

Denote $y=(y_1,y_2,..,y_k)$.
If we assume $y_1=0,y_2=0$, then, from   (\ref{eq524})
$$    e_{[y]}^2 \, V_1=[\frac{1}{\mu([y])} \frac{P_{y_k,1}}{P_{y_k,0}\,} \, \mathfrak{1}_{[y_1,y_2,...,y_k, 0]} + \frac{1}{\mu([y])} \frac{P_{y_k,0}}{P_{y_k,1} } \, \mathfrak{1}_{[ y_1,y_2,...,y_k ,1]}]  $$
$$\times [ \pi_1 P_{1,0}\, \, \mathfrak{1}_{[0,0]} - \pi_0 P_{0,0}\, \, \mathfrak{1}_{[1,0]} ]=
  $$
\begin{align}\frac{1}{\mu([y])} \frac{P_{y_k,1}}{P_{y_k,0}\,} \, \pi_1\, P_{1,0}\,\, \mathfrak{1}_{[y_1, y_2,...,y_k,0]}
  +    \frac{1}{\mu([y])} \frac{P_{y_k,0}}{P_{y_k,1} }\,  \pi_1\, P_{1,0}\, \,\, \mathfrak{1}_{[y_1,y_2,...,y_k, 1]}.
    \label{eq5246700}
    \end{align}

If we assume $y_1=1,y_2=0$, then, from   (\ref{eq524})
$$    e_{[y]}^2 \, V_1=[\frac{1}{\mu([y])} \frac{P_{y_k,1}}{P_{y_k,0}\,} \, \mathfrak{1}_{[y_1,y_2,...,y_k, 0]} + \frac{1}{\mu([y])} \frac{P_{y_k,0}}{P_{y_k,1} } \, \mathfrak{1}_{[ y_1,y_2,...,y_k ,1]}]  $$
$$\times [ \pi_1 P_{1,0}\, \, \mathfrak{1}_{[0,0]} - \pi_0 P_{0,0}\, \, \mathfrak{1}_{[1,0]} ]=
  $$
\begin{align}- \,[\frac{1}{\mu([y])} \frac{P_{y_k,1}}{P_{y_k,0}\,} \, \pi_0\, P_{0,0}\,\, \mathfrak{1}_{[y_1, y_2,...,y_k,0]}
  +    \frac{1}{\mu([y])} \frac{P_{y_k,0}}{P_{y_k,1} }\,  \pi_0\, P_{0,0}\, \,\, \mathfrak{1}_{[y_1,y_2,...,y_k, 1]}].
    \label{eq5246710}
    \end{align}

\smallskip

As we assume that $x_1=0$,
then, from (\ref{luc}), (\ref{lucato}), (\ref{eq52467})
we get
$$ a_x^2 V_1=[\frac{\pi_{x_1} } {\pi_{0} P_{0,x_1} } \,\, e_{[0,x_1, x_2,..,x_k, 0,x_{k+2},...,x_n]}^2 + \,\frac{\pi_{x_1}  } {\pi_{1}P_{1,x_1} }\,\,e_{[1,x_1, x_2,...,x_k,0,x_{k+2},...,x_n]}^2] $$
$$\times [ \pi_1 P_{1,0}\, \, \mathfrak{1}_{[0,0]} - \pi_0 P_{0,0}\, \, \mathfrak{1}_{[1,0]} ]=
  $$
$$=  \, \frac{\pi_{x_1}  } {\pi_{0} P_{0,x_1} } [ \frac{1}{\mu([0x])} \frac{P_{x_n,1}}{P_{x_n,0}\,} \, \pi_1\, P_{1,0}\,\, \mathfrak{1}_{[0,x_1, x_2,...,x_n,0]}
  +    \frac{1}{\mu([0x])} \frac{P_{x_n,0}}{P_{x_n,1} }\,  \pi_1\, P_{1,0}\, \,\, \mathfrak{1}_{[0,x_1,x_2,...,x_n, 1]}]+  $$
$$  \, \frac{\pi_{x_1}  } {\pi_{1} P_{1,x_1} } [ \frac{1}{\mu([1x])} \frac{P_{x_n,1}}{P_{x_n,0}\,} \, \pi_0\, P_{0,0}\,\, \mathfrak{1}_{[1,x_1, x_2,...,x_n,0]}
  +    \frac{1}{\mu([1x])} \frac{P_{x_n,0}}{P_{x_n,1} }\,  \pi_0\, P_{0,0}\, \,\, \mathfrak{1}_{[1,x_1,x_2,...,x_n, 1]}]\,.$$

Therefore,
$$\int a_x^2 V_1 d\mu=$$
$$ \, \frac{\pi_{x_1}  } {\pi_{0} P_{0,x_1} } [ \frac{1}{\mu([0x])} \frac{P_{x_n,1}}{P_{x_n,0}\,} \, \pi_1\, P_{1,0}\,\, \mu[0,x,0]\,
  +    \frac{1}{\mu([0x])} \frac{P_{x_n,0}}{P_{x_n,1} }\,  \pi_1\, P_{1,0}\, \,\, \mu[0,x, 1]\,]\,+ $$
$$  \, \frac{\pi_{x_1}  } {\pi_{1} P_{1,x_1} } [ \frac{1}{\mu([1x])} \frac{P_{x_n,1}}{P_{x_n,0}\,} \, \pi_0\, P_{0,0}\,\, \mu[1,x,0]
  +    \frac{1}{\mu([1x])} \frac{P_{x_n,0}}{P_{x_n,1} }\,  \pi_0\, P_{0,0}\, \,\, \mu[1,x, 1]]\,=$$
$$ \, \frac{\pi_{x_1}  } {\pi_{0} P_{0,x_1} } [  P_{x_n,1} \, \pi_1\, P_{1,0}\,\, \,
  +    P_{x_n,0} \, \pi_1\, P_{1,0}\,]\,+    \, \frac{\pi_{x_1}  } {\pi_{1} P_{1,x_1} } [ P_{x_n,1} \, \pi_0\, P_{0,0}\,\,
  +     P_{x_n,0}\,  \pi_0\, P_{0,0}\, \,]\,=$$
$$ \, \frac{\pi_{x_1}  } {\pi_{0} P_{0,x_1} }  \pi_1\, P_{1,0}\,\, \,
\,+    \, \frac{\pi_{x_1}  } {\pi_{1} P_{1,x_1} }  \pi_0\, P_{0,0}.$$

As  we assumed that $x_1=0$ we get
\begin{equation} \label{rret11} \int \hat{a}_x^2\hat{a}_{[\emptyset]}^0 d\mu=\, \frac{1}{|a_x|^2\, | V_1|  }( \frac{ \pi_1\, P_{1,0} } {P_{0,0} }  \,\, \,
\,+   \, \frac{\pi_0^2\, P_{0,0}  } {\pi_1 P_{1,0} })  \,\, \, .
\end{equation}

\medskip

II) $\langle  a_x^2, V_2\rangle $

Now we will compute $ \int a_x^2 V_2 d\mu.$


Denote $y=(y_1,y_2,..,y_k)$.
If we assume $y_1=0,y_2=0$, then, from   (\ref{eq524})
$$    e_{[y]}^2 \, V_2=[\frac{1}{\mu([y])} \frac{P_{y_k,1}}{P_{y_k,0}\,} \, \mathfrak{1}_{[y_1,y_2,...,y_k, 0]} + \frac{1}{\mu([y])} \frac{P_{y_k,0}}{P_{y_k,1} } \, \mathfrak{1}_{[ y_1,y_2,...,y_k ,1]}]  $$
$$\times [ \pi_0 P_{0,1}\, \, \mathfrak{1}_{[1,1]} - \pi_1 P_{1,1}\, \, \mathfrak{1}_{[0,1]} ]=0.
  $$

If we assume $y_1=1,y_2=0$, then, from   (\ref{eq524})
$$    e_{[y]}^2 \, V_2=[\frac{1}{\mu([y])} \frac{P_{y_k,1}}{P_{y_k,0}\,} \, \mathfrak{1}_{[y_1,y_2,...,y_k, 0]} + \frac{1}{\mu([y])} \frac{P_{y_k,0}}{P_{y_k,1} } \, \mathfrak{1}_{[ y_1,y_2,...,y_k ,1]}]  $$
$$\times [ \pi_0 P_{0,1}\, \, \mathfrak{1}_{[1,1]} - \pi_1 P_{1,1}\, \, \mathfrak{1}_{[0,1]} ]=0.
  $$

\smallskip

As we assumed that $x_1=0$, then, from (\ref{luc}), (\ref{lucato}), (\ref{eq52467})
we get
$$ a_x^2 V_2=[\frac{\pi_{x_1} } {\pi_{0} P_{0,x_1} } \,\, e_{[0,x_1, x_2,...,x_n]}^2 + \,\frac{\pi_{x_1}  } {\pi_{1}P_{1,x_1} }\,\,e_{[1,x_1, x_2,...,x_n]}^2] $$
$$\times [ \pi_0 P_{0,1}\, \, \mathfrak{1}_{[1,1]} - \pi_1 P_{1,1}\, \, \mathfrak{1}_{[0,1]} ]=0.
  $$
Therefore, if $x_1=0$ we get
\begin{equation} \label{poii}\int \hat{a}_x^2 \hat{a}_{[\emptyset]}^1 \,d\mu=0.
\end{equation}
\smallskip

The case $x_1=1$ is left for the reader.

\medskip

\subsection{The value of $\langle  \hat{a}_z^2, \hat{a}_y\rangle $   when  length of $y$ is strictly  larger  than the  length of $z$} \label{ylong1}

 Now we want to estimate $\langle  \hat{a}_z^2, \hat{a}_y\rangle =\int \hat{a}_z^2 \, \hat{a}_y d \mu$ in the case that the length of $y$ is strictly larger  than the length of $z$. We will show that $\int \hat{a}_z^2 \, \hat{a}_y d \mu=0.$

We assume that $[y]=[x_1,x_2,...x_k,x_{k+1},..., x_n] \subset [z]=[x_1,x_2,...,x_k]$, where $n>k$  (otherwise we get  that$\int \hat{a}_z^2 \, \hat{a}_y d \mu$ is zero from Proposition  \ref{louc}).

In fact, we will show that $\int a_z^2 \, a_y d \mu=0.$

I)
If we assume $x_{k+1}=0$ in the word $[y]$, then, from (\ref{eq524})
$$    e_z^2 e_y =[\frac{1}{\mu([z])} \frac{P_{x_k,1}}{P_{x_k,0}\,} \, \mathfrak{1}_{[x_1,...,x_k, 0]} + \frac{1}{\mu([z])} \frac{P_{x_k,0}}{P_{x_k,1} } \, \mathfrak{1}_{[ x_1,...,x_k ,1]}]  $$
$$\times [\frac{1}{\sqrt{\mu([y])}} \sqrt{\frac{P_{x_n,1}}{P_{x_n,0}\,}} \, \mathfrak{1}_{[x_1,...,x_k,0,x_{k+2},...,x_n, 0]} - \frac{1}{\sqrt{\mu([y])}} \sqrt{\frac{P_{x_n,0}}{P_{x_n,1} }} \, \mathfrak{1}_{[x_1,...,x_k,0,x_{k+2},...,x_n,  1]}  ]  $$
$$= (\frac{1}{\mu([z])} \frac{P_{x_k,1}}{P_{x_k,0}\,}) (  \frac{1}{\sqrt{\mu([y])}} \sqrt{\frac{P_{x_n,1}}{P_{x_n,0}\,}} )\,
\mathfrak{1}_{[x_1,...,x_k,0, x_{k+2},...,x_n,0]} $$
\begin{align}\label{eq52467808} -
(\frac{1}{\mu([z])} \frac{P_{x_k,1}}{P_{x_k,0}\,}) (  \frac{1}{\sqrt{\mu([y])}} \sqrt{\frac{P_{x_n,0}}{P_{x_n,1}\,}} )\, \mathfrak{1}_{[x_1,...,x_k,0, x_{k+2},...,x_n,1]}.
\end{align}

Note that above, from the second to the third line, we use the fact that
$$  \mathfrak{1}_{[ x_1,...,x_k ,1]}  \mathfrak{1}_{[x_1,...,x_k,0,x_{k+2},...,x_n, 0]} =0$$ and
$$  \mathfrak{1}_{[ x_1,...,x_k ,1]}  \mathfrak{1}_{[x_1,...,x_k,0,x_{k+2},...,x_n,  1]} =0.$$

Then, from (\ref{luc}), \eqref{eq52467808}  and (\ref{lucato})

$$  a_z^2 a_y=  [\frac{\pi_{x_1}  } {\pi_{0} P_{0,x_1} } \,\, e_{[0,x_1, x_2,..,x_k]}^2 + \,\frac{ \pi_{x_1}  } {\pi_{1}P_{1,x_1} }\,\,e_{[1,x_1, x_2,..,x_k]}^2]         $$
$$\times    [\frac{\sqrt{\pi_{x_1}} } {\sqrt{\pi_{0} P_{0,x_1}} } \,\, e_{[0,x_1, x_2,..,x_k, 0,x_{k+2},...,x_n]} - \,\frac{\sqrt{\pi_{x_1} } } {\sqrt{\pi_{1}P_{1,x_1} }}\,\,e_{[1,x_1, x_2,...,x_k,0,x_{k+2},...,x_n]}] $$
$$  =\, ( \frac{\pi_{x_1}  } {\pi_{0} P_{0,x_1} })^{3/2} [ (\frac{1}{\mu([0 z])} \frac{P_{x_k,1}}{P_{x_k,0}\,})   (\frac{1}{\sqrt{\mu([0 y])}} \sqrt{\frac{P_{x_n,1}}{P_{x_n,0}\,}}\, ) \mathfrak{1}_{[0y0]}   $$
$$- (\frac{1}{\mu([0 z])} \frac{P_{x_k,1}}{P_{x_k,0} }) ( \frac{1}{\sqrt{\mu([0 y])}} \sqrt{\frac{P_{x_n,0}}{P_{x_n,1}\,}}) \, \mathfrak{1}_{[0y1]}\,]
$$
$$-( \frac{\pi_{x_1}  } {\pi_{1} P_{1,x_1} } )^{3/2}  [\,  (\frac{1}{\mu([1 z])} \frac{P_{x_k,1}}{P_{x_k,0}\,} )  ( \frac{1}{\sqrt{\mu([1 y])}} \sqrt{\frac{P_{x_n,1}}{P_{x_n,0}\,}) }\, \mathfrak{1}_{[1y0]} $$
$$- (\frac{1}{\mu([1 z])} \frac{P_{x_k,1}}{P_{x_k,0} })\,(\frac{1}{\sqrt{\mu([1 y])}} \sqrt{\frac{P_{x_n,0}}{P_{x_n,1} }}) \mathfrak{1}_{[1y1]}\,] .
$$

Finally,
$$ \int a_z^2 a_y d \mu $$
$$ = \, ( \frac{\pi_{x_1}  } {\pi_{0} P_{0,x_1} })^{3/2}    [ \frac{P_{x_k,1}}{\mu([0 z 0])}
\sqrt{\mu([0 y 0])} \sqrt{P_{x_n,1}} -  \frac{P_{x_k,1}}{\mu([0 z 0]}   \sqrt{\mu([0 y 1])} \sqrt{P_{x_n,0}\,}\, \,]
$$
$$+ ( \frac{\pi_{x_1}  } {\pi_{1} P_{1,x_1} } )^{3/2}  \,
[- \frac{P_{x_k,1}}{\mu([1 z 0])} \sqrt{\mu([1 y 0])} \sqrt{P_{x_n,1} }\, +\,\frac{P_{x_k,1}}{\mu([1 z 0])} \sqrt{\mu([1 y 1])} \sqrt{P_{x_n,0}}\, \,]
$$
$$=P_{x_k,1}\sqrt{\mu([y])} \,\{  \, ( \frac{\pi_{x_1}  } {\pi_{0} P_{0,x_1} })^{3/2}    [ \frac{\sqrt{P_{x_n,1}        P_{0,x_1} P_{x_n,0}} }{\mu([0 z 0])}
 -  \frac{\sqrt{P_{x_n,0}  P_{0,x_1} P_{x_n,1}\,}}{\mu([0 z 0]}    \, \,]
$$
 \begin{equation} \label{tororo109} + P_{x_k,1} ( \frac{\pi_{x_1}  } {\pi_{1} P_{1,x_1} } )^{3/2}  \,
[- \frac{\sqrt{P_{x_n,1} P_{1,x_1} P_{x_n,0} }}{\mu([1 z 0])}  \, +\,\frac{ \sqrt{P_{x_n,0}        P_{1,x_1} P_{x_n,1}}}{\mu([1 z 0])} \, \,] \,\} =0.
\end{equation}

II)
If we assume $x_{k+1}=1$ in the word $[y]$, then, from (\ref{eq524})
$$    e_z^2 e_y =[\frac{1}{\mu([z])} \frac{P_{x_k,1}}{P_{x_k,0}\,} \, \mathfrak{1}_{[x_1,...,x_k, 0]} + \frac{1}{\mu([z])} \frac{P_{x_k,0}}{P_{x_k,1} } \, \mathfrak{1}_{[ x_1,...,x_k ,1]}] $$
$$ \times [\frac{1}{\sqrt{\mu([y])}} \sqrt{\frac{P_{x_n,1}}{P_{x_n,0}\,}} \, \mathfrak{1}_{[x_1,...,x_k,1,x_{k+2},...,x_n, 0]} - \frac{1}{\sqrt{\mu([y])}} \sqrt{\frac{P_{x_n,0}}{P_{x_n,1} }} \, \mathfrak{1}_{[x_1,...,x_k,1,x_{k+2},...,x_n,  1]}  ]  $$
$$=(\frac{1}{\mu([z])} \frac{P_{x_k,0}}{P_{x_k,1}\,}) (  \frac{1}{\sqrt{\mu([y])}} \sqrt{\frac{P_{x_n,1}}{P_{x_n,0}\,}} )\,
\mathfrak{1}_{[x_1,...,x_k,1, x_{k+2},...,x_n,0]} $$
\begin{align}\label{eq524678} -
(\frac{1}{\mu([z])} \frac{P_{x_k,0}}{P_{x_k,1}\,}) (  \frac{1}{\sqrt{\mu([y])}} \sqrt{\frac{P_{x_n,0}}{P_{x_n,1}\,}} )\, \mathfrak{1}_{[x_1,...,x_k,1, x_{k+2},...,x_n,1]}
\end{align}

Then, from (\ref{luc}), (\ref{eq524678}) and (\ref{lucato})

$$  a_z^2 a_y=  [\frac{\pi_{x_1}  } {\pi_{0} P_{0,x_1} } \,\, e_{[0,x_1, x_2,..,x_k]}^2 + \,\frac{ \pi_{x_1}  } {\pi_{1}P_{1,x_1} }\,\,e_{[1,x_1, x_2,..,x_k]}^2]         $$
$$\times    [\frac{\sqrt{\pi_{x_1}} } {\sqrt{\pi_{0} P_{0,x_1}} } \,\, e_{[0,x_1, x_2,..,x_k, 0,x_{k+2},...,x_n]} - \,\frac{\sqrt{\pi_{x_1} } } {\sqrt{\pi_{1}P_{1,x_1} }}\,\,e_{[1,x_1, x_2,...,x_k,0,x_{k+2},...,x_n]}] $$
$$=  \, ( \frac{\pi_{x_1}  } {\pi_{0} P_{0,x_1} })^{3/2} [ (\frac{1}{\mu([0 z])} \frac{P_{x_k,0}}{P_{x_k,1}\,})   (\frac{1}{\sqrt{\mu([0 y])}} \sqrt{\frac{P_{x_n,1}}{P_{x_n,0}\,}}\, ) \mathfrak{1}_{[0y0]}   $$
$$ - (\frac{1}{\mu([0 z])} \frac{P_{x_k,0}}{P_{x_k,1} }) ( \frac{1}{\sqrt{\mu([0 y])}} \sqrt{\frac{P_{x_n,0}}{P_{x_n,1}\,}}) \, \mathfrak{1}_{[0y1]}\,]
$$
$$-( \frac{\pi_{x_1}  } {\pi_{1} P_{1,x_1} } )^{3/2}  [\,  (\frac{1}{\mu([1 z])} \frac{P_{x_k,0}}{P_{x_k,1}\,} )  ( \frac{1}{\sqrt{\mu([1 y])}} \sqrt{\frac{P_{x_n,1}}{P_{x_n,0}\,}) }\, \mathfrak{1}_{[1y0]} $$
$$- (\frac{1}{\mu([1 z])} \frac{P_{x_k,0}}{P_{x_k,1} })\,(\frac{1}{\sqrt{\mu([1 y])}} \sqrt{\frac{P_{x_n,0}}{P_{x_n,1} }}) \mathfrak{1}_{[1y1]}\,] .
$$

Finally,
$$ \int a_z^2 a_y d \mu$$
$$ =  \, ( \frac{\pi_{x_1}  } {\pi_{0} P_{0,x_1} })^{3/2}    [ \frac{P_{x_k,0}}{\mu([0 z 1])}
\sqrt{\mu([0 y 0])} \sqrt{P_{x_n,1}} -  \frac{P_{x_k,0}}{\mu([0 z 1]}   \sqrt{\mu([0 y 1])} \sqrt{P_{x_n,0}\,}\, \,]
$$
$$+( \frac{\pi_{x_1}  } {\pi_{1} P_{1,x_1} } )^{3/2}  \,
[- \frac{P_{x_k,1}}{\mu([1 z 0])} \sqrt{\mu([1 y 0])} \sqrt{P_{x_n,1} }\, +\,\frac{P_{x_k,1}}{\mu([1 z 0])} \sqrt{\mu([1 y 1])} \sqrt{P_{x_n,0}}\, \,]
$$
$$=P_{x_k,0}\sqrt{\mu([y])} \,\{  \, ( \frac{\pi_{x_1}  } {\pi_{0} P_{0,x_1} })^{3/2}    [ \frac{\sqrt{P_{x_n,1}        P_{0,x_1} P_{x_n,0}} }{\mu([0 z 1])}
 -  \frac{\sqrt{P_{x_n,0}  P_{0,x_1} P_{x_n,1}}  \,}{\mu([0 z 1]}  \, \,]
$$
 \begin{equation} \label{tororo109}+ P_{x_k,1}( \frac{\pi_{x_1}  } {\pi_{1} P_{1,x_1} } )^{3/2}  \,
[- \frac{\sqrt{P_{x_n,1} P_{1,x_1} P_{x_n,0} }}{\mu([1 z 1])}  \, +\,\frac{\sqrt{P_{x_n,0}        P_{1,x_1} P_{x_n,1}}  }{\mu([1 z 1])}\, \,] \,\}=0 .
\end{equation}

\smallskip

\section{Computations for the integral $\int X Y Z $ } \label{xyz}
 Our purpose on this section is: given $x$ and $z$ we want to compute for all $y$
 \begin{equation} \label{gaga} \sum_{\text{word} \, y\,}(   \int \hat{a}_x \,\hat{a}_z\, \hat{a}_y \,d \mu)^2,
 \end{equation}
which corresponds to the first term in the sum given by expression (\ref{tororo172}).

Remember that from Corollary  \ref{pop24} if $x$ is not a subprefix of $z$ and  $z$ is not a subprefix of $x$, we get that for any $y$
$$\int \hat{a}_x \,\hat{a}_z \, \hat{a}_y d \mu=0.$$

Without loss of generality, we  assume that $z$ is a subprefix of $x$ (see Proposition \ref{louc1}). The only possible nonzero value for \eqref{gaga}  is $\int \hat{a}_x^2 \, \hat{a}_z \,d \mu.$ This justify the first term in the sum \eqref{tororo182}.

We assume first that:

$[y]=[x_1,x_2,...x_k,x_{k+1},..., x_n,x_{n+1},...,x_j]   \subset [x]=[x_1,x_2,...x_k,x_{k+1},..., x_n] \subset [z]=[x_1,x_2,...,x_k]$,
where $j>n\geq k$.
\smallskip

 We will show in all cases  that $ \int a_x \,a_y\, a_z \,d \mu=0.$  This includes the case \begin{equation} \label{KLGT} \int a_z^2 \, a_x \,d \mu=0.
\end{equation}
\smallskip

I) First we assume that $x_{k+1} =0= x_{n+1}.$

Then,
$$ a_x \,a_y\, a_z =$$
$$    [\frac{ \sqrt{\pi_{x_1}  }} {\sqrt{\pi_{0}}\sqrt{P_{0,x_1} }} \,\, e_{[0,x_1, x_2,..,x_k]} - \,\frac{ \sqrt{\pi_{x_1}  }} {\sqrt{\pi_{1}}\sqrt{P_{1,x_1} }}\,\,e_{[1,x_1, x_2,..,x_k]}]        $$
$$   \times   [\frac{ \sqrt{\pi_{x_1}  }} {\sqrt{\pi_{0}}\sqrt{P_{0,x_1} }} \,\, e_{[0,x_1, x_2,..,x_j]} - \,\frac{ \sqrt{\pi_{x_1}  }} {\sqrt{\pi_{1}}\sqrt{P_{1,x_1} }}\,\,e_{[1,x_1, x_2,..,x_j]}]         $$
$$   \times  [\frac{ \sqrt{\pi_{x_1}  }} {\sqrt{\pi_{0}}\sqrt{P_{0,x_1} }} \,\, e_{[0,x_1, x_2,..,x_n]} - \,\frac{ \sqrt{\pi_{x_1}  }} {\sqrt{\pi_{1}}\sqrt{P_{1,x_1} }}\,\,e_{[1,x_1, x_2,..,x_n]}]          $$

\medskip

$$  =\, ( \frac{\pi_{x_1}  } {\pi_{0} P_{0,x_1}  \,})^{3/2} \,   \frac{1}{\sqrt{\mu([0x])\mu([0z])\mu([0y])}} \,[ \sqrt{\frac{P_{x_n,1}P_{x_k,1} P_{x_j,1}}{P_{x_n,0}P_{x_k,0}P_{x_j,0}}} \, \mathfrak{1}_{[0y0]}  $$
$$-  \sqrt{\frac{P_{x_n,1}P_{x_k,1}P_{x_j,0}}{P_{x_n,0}P_{x_k,0} P_{x_j,1}}} \, \mathfrak{1}_{[0y 1]}\,] $$
$$ - \, ( \frac{\pi_{x_1}  } {\pi_{1} P_{1,x_1} } \,)^{3/2} \, \frac{1}{\sqrt{\mu([1x])\mu([1z])\mu([1y])}} \,[ \sqrt{\frac{P_{x_n,1}P_{x_k,1} P_{x_j,1}}{P_{x_n,0}P_{x_k,0}P_{x_j,0}}} \, \mathfrak{1}_{[1y0]}   $$
$$- \sqrt{\frac{P_{x_n,1}P_{x_k,1}P_{x_j,0}}{P_{x_n,0}P_{x_k,0}P_{x_j,1} }} \, \mathfrak{1}_{[1y 1]}\,] .$$

Note that for all $j$
\begin{equation} \label{ooop} \sqrt{\frac{P_{x_j,1}}{P_{x_j,0}  }} \, \mu([0y0])=  \sqrt{P_{x_j,1 }\,P_{x_j,0}  } \, \mu([0y])= \sqrt{\frac{P_{x_j,0}}{P_{x_j,1}  }} \, \mu([0y1]).
\end{equation}
and
\begin{equation} \label{ooop1}\sqrt{\frac{ P_{x_j,1}}{P_{x_j,0}}  }  \mu([1y0]) =  \sqrt{P_{x_j,1 }\,P_{x_j,0}  } \, \mu([1y]) = \sqrt{\frac{ P_{x_j,0}}{P_{x_j,1}}  }  \mu([1y1])
\end{equation}

Finally, from \eqref{ooop}  and \eqref{ooop1}
$$ \int a_x \,a_y\, a_z \,d \mu=$$
$$  \, ( \frac{\pi_{x_1}  } {\pi_{0} P_{0,x_1}  \,})^{3/2} \,   \frac{1}{\sqrt{\mu([0x])\mu([0z])\mu([0y])}} \,[ \sqrt{\frac{P_{x_n,1}P_{x_k,1} P_{x_j,1}}{P_{x_n,0}P_{x_k,0}P_{x_j,0}}} \, \mu([0y0])  $$
$$- \sqrt{\frac{P_{x_n,1}P_{x_k,1}P_{x_j,0}}{P_{x_n,0}P_{x_k,0} P_{x_j,1}}} \,\mu([0y1])\,] $$
$$-  \, ( \frac{\pi_{x_1}  } {\pi_{1} P_{1,x_1} } \,)^{3/2} \, \frac{1}{\sqrt{\mu([1x])\mu([1z])\mu([1y])}} \,[ \sqrt{\frac{P_{x_n,1}P_{x_k,1} P_{x_j,1}}{P_{x_n,0}P_{x_k,0}P_{x_j,0}}} \, \mu([1y0])  $$
$$ -\sqrt{\frac{P_{x_n,1}P_{x_k,1}P_{x_j,0}}{P_{x_n,0}P_{x_k,0}P_{x_j,1} }} \, \mu([1y1])\,] =$$
$$  \, ( \frac{\pi_{x_1}  } {\pi_{0} P_{0,x_1}  \,})^{3/2} \,   \frac{1}{\sqrt{\mu([0x])\mu([0z])\mu([0y])}} \,[ \sqrt{\frac{P_{x_n,1}P_{x_k,1} }{P_{x_n,0}P_{x_k,0}}} \,\sqrt{\frac{ P_{x_j,1}}{P_{x_j,0}}  } \mu([0y0])  $$
$$- \sqrt{\frac{P_{x_n,1}P_{x_k,1}}{P_{x_n,0}P_{x_k,0} }} \,\sqrt{\frac{ P_{x_j,0}}{P_{x_j,1}}  } \,\mu([0y1])\,] $$
$$-  \, ( \frac{\pi_{x_1}  } {\pi_{1} P_{1,x_1} } \,)^{3/2} \, \frac{1}{\sqrt{\mu([1x])\mu([1z])\mu([1y])}} \,[ \sqrt{\frac{P_{x_n,1}P_{x_k,1} }{P_{x_n,0}P_{x_k,0}}} \,\,\sqrt{\frac{ P_{x_j,1}}{P_{x_j,0}}  }  \mu([1y0])  $$
$$ -\sqrt{\frac{P_{x_n,1}P_{x_k,1}}{P_{x_n,0}P_{x_k,0}}} \,\,\sqrt{\frac{ P_{x_j,0}}{P_{x_j,1}}  }  \mu([1y1])\,] = 0-0=0 .$$
\smallskip

II) Now we assume that $x_{k+1} =1= x_{n+1}.$ In a similar way as before

$$ \int a_x \,a_y\, a_z \,d \mu$$
$$ =\, ( \frac{\pi_{x_1}  } {\pi_{0} P_{0,x_1}  \,})^{3/2} \,   \frac{1}{\sqrt{\mu([0x])\mu([0z])\mu([0y])}} \,[ \sqrt{\frac{P_{x_n,0}P_{x_k,0} P_{x_j,1}}{P_{x_n,1}P_{x_k,1}P_{x_j,0}}} \, \mu([0y0]) $$
$$ -\sqrt{\frac{P_{x_n,0}P_{x_k,0}P_{x_j,0}}{P_{x_n,1}P_{x_k,1} P_{x_j,1}}} \,\mu([0y1])\,] $$
$$ - \, ( \frac{\pi_{x_1}  } {\pi_{1} P_{1,x_1} } \,)^{3/2} \, \frac{1}{\sqrt{\mu([1x])\mu([1z])\mu([1y])}} \,[ \sqrt{\frac{P_{x_n,0}P_{x_k,0} P_{x_j,1}}{P_{x_n,1}P_{x_k,1}P_{x_j,0}}} \, \mu([1y0])  $$
$$- \sqrt{\frac{P_{x_n,0}P_{x_k,0}P_{x_j,0}}{P_{x_n,1}P_{x_k,1}P_{x_j,1} }} \, \mu([1y1])\,] $$
$$ = \, ( \frac{\pi_{x_1}  } {\pi_{0} P_{0,x_1}  \,})^{3/2} \,   \frac{\sqrt{\mu([0y0]) \, P_{x_j,1}}}{\sqrt{\mu([0x])\mu([0z])}} \,[ \sqrt{\frac{P_{x_n,0}P_{x_k,0} }{P_{x_n,1}P_{x_k,1}}} \, -  \sqrt{\frac{P_{x_n,0}P_{x_k,0}}{P_{x_n,1}P_{x_k,1} }} \,\,] $$
$$ - \, ( \frac{\pi_{x_1}  } {\pi_{1} P_{1,x_1} } \,)^{3/2} \,        \frac{\sqrt{\mu([1y0]) \, P_{x_j,1}}}{\sqrt{\mu([1 x])\mu([1 z])}} \,[ \sqrt{\frac{P_{x_n,0}P_{x_k,0} }{P_{x_n,1}P_{x_k,1}}} \,  -  \sqrt{\frac{P_{x_n,0}P_{x_k,0}}{P_{x_n,1}P_{x_k,1}}} \, \,]= 0-0=0 .$$
\smallskip

III) Now if we assume that $x_{k+1} =0 $ and $ x_{n+1}=1$ or that $x_{k+1} =1 $ and $ x_{n+1}=0$, we get that in a similar way that
$$ \int \hat{a}_x \,\hat{a}_y\, \hat{a}_z \,d \mu=0.$$

\smallskip

After all these computations,
for fixed $\hat{a}_x$ and $\hat{a}_z$,  we want to compute  $K(\hat{a}_z,\hat{a}_x)$. In this direction  we have to consider (\ref{gaga})  which is the first sum in expression (\ref{tororo172})

We wonder for which $y$ we have that $ (\int \hat{a}_x \,\hat{a}_z\, \hat{a}_y \,d \mu)^2\neq 0.$
We assumed without loss of generality that $z$ is a subprefix of $x$. In this case, the length of $x$ is strictly  larger than the length of $z$.

Considering first the case where the length of $y$ is larger than $z$ and $x$, it follows from the  above that
$$\sum_{\text{word} \, y\,\text{with length larger than}\, x \, \text{and}\,z}(   \int \hat{a}_x \,\hat{a}_z\, \hat{a}_y \,d \mu)^2=0.$$

Now we consider  the case where the length of $y$ is strictly smaller than the length of  $z$ and $x$.

For the case where the length of $y$ is strictly smaller than $z$ and $x$ we need to assume that $y$ is a subprefix of $z$ (otherwise $\hat{a}_y \hat{a}_z=0$ and we get $ (\int \hat{a}_x \,\hat{a}_z\, \hat{a}_y \,d \mu)^2=0$). If $y$ is a strict subprefix of $z$ and $z$ is a strict subprefix of $s$ we get from the above that $ (\int \hat{a}_x \,\hat{a}_z\, \hat{a}_y \,d \mu)^2=0$.

Finally, we assume that  the length of $y$ is strictly smaller than $x$  and strictly larger than $z$. In this case we have to assume that
$x$ is a subprefix of $y$ and $y$ is a subprefix of $z$ (otherwise  by Proposition \ref{louc1} we have   $ \int \hat{a}_x \,\hat{a}_z\, \hat{a}_y \,d \mu^2=0$).
It follows from the above that  also in this case  $ (\int \hat{a}_x \,\hat{a}_z\, \hat{a}_y \,d \mu)^2=0$.

Therefore, in the estimation of expression (\ref{gaga}) it follows from our reasoning that all elements in this sum are zero up to expressions
 $(\int \hat{a}_x^2 \, \hat{a}_z \,d \mu)^2$ and $(\int \hat{a}_z^2 \, \hat{a}_x \,d \mu)^2$, that is, the cases where $y=x$ or $y=z$.  From Proposition \ref{louc} we have to assume that $x$ is a subprefix of $z$ or vice versa. The explicit expressions for these  two cases were analyzed in sections \ref{ylong2} and
 \ref{ylong1}.

 If the length of $x$ is larger than the length of $z$, then, from (\ref{tororo109}) we get
 $(\int \hat{a}_z^2 \, \hat{a}_x \,d \mu)^2 =0.$

 The final conclusion is that
  \begin{equation} \label{gagatr} \sum_{\text{word} \, y\,}(   \int \hat{a}_x \,\hat{a}_z\, \hat{a}_y \,d \mu)^2 = (\int \hat{a}_x^2 \, \hat{a}_z \,d \mu)^2 + (\int \hat{a}_z^2 \, \hat{a}_x \,d \mu)^2= (\int \hat{a}_x^2 \, \hat{a}_z \,d \mu)^2.
 \end{equation}

\end{document}